\documentclass[preprint]{imsart}

\usepackage{amsmath,amssymb}
\usepackage{parskip,csquotes}
\usepackage{marvosym}
\usepackage[hang,small,bf]{caption}
\usepackage{mathtools}
\usepackage{enumitem}
\usepackage[colorlinks]{hyperref}
\hypersetup{
	colorlinks,%
	citecolor=blue,%
	filecolor=black,%
	linkcolor=red,%
	urlcolor=blue
}

\RequirePackage[OT1]{fontenc}
\usepackage[top=2.5cm, bottom=2.5cm, left=2.5cm, right=2.5cm]{geometry}

\usepackage{amsthm}
\usepackage{amsmath}
\usepackage{amssymb}
\usepackage{thmtools} % required to make autoref work
\usepackage{subcaption}
\usepackage{textcomp}

\usepackage{graphicx}
\usepackage{verbatim}
\usepackage{array, float}
\usepackage{fontenc}
\usepackage[toc,page]{appendix}
\usepackage[nottoc]{tocbibind}
\usepackage[english]{babel}
\usepackage{marvosym}
\usepackage[pagewise]{lineno}
\usepackage{mathtools}
\allowdisplaybreaks
\usepackage{bbm}
\usepackage[noabbrev,capitalize]{cleveref}

\theoremstyle{exampstyle}

\newtheorem{theorem}{Theorem}
\newtheorem{example}{Example}
\newtheorem{lemma}{Lemma}
\newtheorem{corollary}{Corollary}
\newtheorem{remark}{Remark}

\newtheorem{prop}{Proposition}

\newtheorem{definition}{Definition}

\numberwithin{equation}{section}
\numberwithin{example}{section}
\numberwithin{theorem}{section}
\numberwithin{lemma}{section}
\numberwithin{corollary}{section}
\numberwithin{prop}{section}
\numberwithin{definition}{section}
\numberwithin{remark}{section}

\usepackage{tikz}
\tikzset{
	treenode/.style = {shape=rectangle, rounded corners,
		draw, align=center,
		top color=white, bottom color=blue!20},
	root/.style     = {treenode, font=\Large, bottom color=yellow},
	env/.style      = {treenode, font=\ttfamily\normalsize},
	con/.style      = {treenode, font=\ttfamily, bottom color=green!25},
	nocon/.style    = {treenode, font=\ttfamily, bottom color=red!30},
	dummy/.style    = {circle,draw}
}

\startlocaldefs

\newcommand{\eat}[1]{}

\DeclareMathOperator*{\argmax}{\arg\!\max}

\renewcommand{\bar}[1]{\overline{#1}}

\renewcommand{\tilde}[1]{\widetilde{#1}}
\newcommand{\E}{\mathbb{E}}
\renewcommand{\P}{\mathbb{P}}

\newcommand{\BN}{\mathcal{A}_N}

\newcommand{\R}{\mathbb{R}}
\newcommand{\ut}{t}
\newcommand{\tbm}{\tilde{\mathbf{m}}}

\newcommand{\vo}{\mathbf{1}}

\newcommand{\tw}{\tilde{W}}

\@addtoreset{proofpart}{theorem}

\makeatletter
\newcommand*{\rom}[1]{\expandafter\@slowromancap\romannumeral #1@}
\makeatother

\def\argmax{\mathop{\rm argmax}}

\newcommand{\mmx}{\mathbf{x}}

\newcommand{\ms}{\boldsymbol{\sigma}}
\newcommand{\oms}{\overline{\boldsymbol{\sigma}}}
\newcommand{\mm}{\overline{\boldsymbol{m}}(\boldsymbol{\sigma})}
\newcommand{\mq}{\boldsymbol{q}}

\newcommand{\fa}{\lVert A_N\rVert_F^2}

\newcommand{\fd}{\lVert D_N\rVert_F^2}
\newcommand{\ta}{\textbf{(A1)}}

\newcommand{\tg}{\textbf{(A5)}}
\newcommand{\thh}{\textbf{(A6)}}

\newcommand{\mmp}{\mathbbm{P}}
\newcommand{\mmq}{\mathbbm{Q}}

\newcommand{\mme}{\mathbbm{E}}
\newcommand{\mE}{\mathcal{E}}
\newcommand{\tnp}{T_{N,\P}}
\newcommand{\tnpp}{T_{N,\P}'}
\newcommand{\tnq}{T_{N,\mmq}}
\newcommand{\tnqp}{T_{N,\mmq}'}
\newcommand{\mc}{\mathbf{c}}

\newcommand{\sech}{sech}

\definecolor{LightCyan}{rgb}{0.88,1,1}
\definecolor{Gray}{gray}{0.9}

\endlocaldefs

\usepackage{abstract}
\RequirePackage[numbers]{natbib}
\usepackage{xr}

\begin{document}

\renewcommand{\abstractname}{}    % clear the title
\renewcommand{\absnamepos}{empty}

\begin{frontmatter}
%%%%%%%%%%%%%%%%%%%%%%%%%%%%%%%%%%%%%%%%%%%%%%
%%                                          %%
%% Enter the title of your article here     %%
%%                                          %%
%%%%%%%%%%%%%%%%%%%%%%%%%%%%%%%%%%%%%%%%%%%%%%
\title{Fluctuations in Mean-Field Ising Models}
%\title{A sample article title with some additional note\thanksref{T1}}
\runauthor{Deb and Mukherjee}
\runtitle{Fluctuations in Mean-Field Ising Models}
%\thankstext{T1}{A sample of additional note to the title.}

\begin{aug}
	
	\author{\fnms{Nabarun} \snm{Deb}\ead[label=e1]{nd2560@columbia.edu}}
	\and
	\author{\fnms{Sumit} \snm{Mukherjee} \thanksref{t1}\ead[label=e2]{sm3949@columbia.edu}}
	\thankstext{t1}{Research partially supported by NSF grant DMS-1712037}
	
	%\runauthor{Deb and Mukherjee}
	
	%\affiliation{Columbia University}
	
	%\address{Department of Statistics\\ Stanford University,\\ California, USA,\\ CA-94305\\ \printead{e1}}
	
	%\address{Department of Mathematics and Statistics\\ Stanford University,\\ California, USA,\\ CA-94305\\ \printead{e2}}
	
	\address{Department of Statistics, Columbia University\\ \printead{e1,e2}}
%%%%%%%%%%%%%%%%%%%%%%%%%%%%%%%%%%%%%%%%%%%%%%%
%% Only one address is permitted per author. %%
%% Only division, organization and e-mail is %%
%% included in the address.                  %%
%% Additional information can be included in %%
%% the Acknowledgments section if necessary. %%
%% ORCID can be inserted by command:         %%
%% \orcid{0000-0000-0000-0000}               %%
%%%%%%%%%%%%%%%%%%%%%%%%%%%%%%%%%%%%%%%%%%%%%%%
%\author[A]{\fnms{Nabarun}~\snm{Deb}\ead[label=e1]{nd2560@columbia.edu}}
%\and
%\author[B]{\fnms{Sumit}~\snm{Mukherjee}\ead[label=e2]{sm3949@columbia.edu}}
%%%%%%%%%%%%%%%%%%%%%%%%%%%%%%%%%%%%%%%%%%%%%%
%% Addresses                                %%
%%%%%%%%%%%%%%%%%%%%%%%%%%%%%%%%%%%%%%%%%%%%%%
%\address[A]{Department of Statistics, Columbia University\printead[presep={,\ }]{e1}}

%\address[B]{Department of Statistics, Columbia University\printead[presep={,\ }]{e2}}
\end{aug}

\begin{abstract}
	\vspace{0.1in}
In this paper, we study the fluctuations of the average magnetization in an Ising model on an approximately $d_N$ regular graph $G_N$ on $N$ vertices. In particular, if $G_N$ satisfies a \enquote{spectral gap} condition, we show that whenever $d_N\gg \sqrt{N}$, the fluctuations are universal and same as that of the Curie-Weiss model in the entire Ferromagnetic parameter regime. We give a counterexample to demonstrate that the condition $d_N\gg \sqrt{N}$ is tight, in the sense that the limiting distribution changes if $d_N\sim \sqrt{N}$ except in the high temperature regime. By refining our argument, we extend universality in the high temperature regime up to $d_N\gg N^{1/3}$. Our results conclude universal fluctuations of the average magnetization in Ising models on regular graphs, Erd\H{o}s-R\'enyi graphs (directed and undirected), stochastic block models, and sparse regular graphons. In fact, our results apply to general matrices with non-negative entries, including Ising models on a Wigner matrix, and the block spin Ising model. As a by-product of our proof technique, we obtain Berry-Esseen bounds for these fluctuations, exponential concentration for the average of spins, tight error bounds for the Mean-Field approximation of the partition function, and tail bounds for various statistics of interest.
\end{abstract}



\end{frontmatter}
%%%%%%%%%%%%%%%%%%%%%%%%%%%%%%%%%%%%%%%%%%%%%%
%% Please use \tableofcontents for articles %%
%% with 50 pages and more                   %%
%%%%%%%%%%%%%%%%%%%%%%%%%%%%%%%%%%%%%%%%%%%%%%
%\tableofcontents

%%%%%%%%%%%%%%%%%%%%%%%%%%%%%%%%%%%%%%%%%%%%%%
%%%% Main text entry area:
\section{Introduction}\label{sec:intro}

The Ising model is a discrete Markov random field which was initially introduced as a mathematical model of Ferromagnetism in Statistical Physics, and has received extensive attention in Probability~(c.f.~\citep{Rados2019,Bresler2019,Cha2011, AmirAndrea2010, Gheissari2018, giardina2015annealed, Lowe2018, Sly2014, kirsch2020two} and references therein)~and Statistics~(c.f.~\cite{Berthet2019, Comets1991, Ellis1978, jain2019mean,Mukherjee2018, mukherjee2019testing,Ravikumar2010} and references therein). % as well as finding applications in Computer Science, spatial modeling, signal processing, and neural networks. 
The model can be described by the following probability mass function in $\ms:=(\sigma_1,\cdots,\sigma_N)\in \{-1,1\}^N$:% $\ms^{\top}:= (\sigma_1,\sigma_2,\ldots ,\sigma_N)$ is a vector of dependent $\pm 1$ valued random variables where the dependence is modeled through an exponential family of distributions given by, 
\begin{equation}\label{eq:model}
\mmp(\ms):=\frac{1}{Z_N(\beta,B)}\exp\left(\frac{\beta}{2}\ms^{\top} A_N\ms+B\sum\limits_{i=1}^N \sigma_i\right).
\end{equation}
%for $\mt\in \{-1,1\}^N$, where
Here $A_N$ is a symmetric $N\times N$ matrix with non-negative entries, and has zeroes on its diagonal, and $\beta>0$ and $B\in \R$ are scalar parameters often referred to in the Statistical Physics literature as \textit{inverse temperature} and \textit{external magnetic field} respectively. The factor $Z_N(\beta,B)$ is the normalizing constant/partition function of the model. The most common choice of the coupling matrix $A_N$ is the adjacency matrix of a graph $G_N$ on $N$ vertices, scaled by the average degree $\bar{d}_N:=\frac{1}{N}\sum_{i,j=1}^NG_N(i,j)$. Here and throughout the rest of the paper, we use the notation $G_N$ to denote both a graph and its adjacency matrix.
%
%For example, one can choose $A_N$ to be the (scaled) adjacency matrix of a simple graph with $N$ vertices [cf.~\cite{Cha2011,kabluchko2019fluctuations}]. In Ising model literature, it is also called the \textit{coupling matrix}.
%	\item $\beta$ and $B$ are the parameters of the model, often referred to as \textit{inverse temperature} and \textit{external magnetic field} respectively. We will assume $(\beta,B)\in [0,\infty)\times (-\infty,\infty)$.
%	\item $Z_N(\beta,B)$ is the normalizing constant, which is a function of both $\beta$ and $B$. It is also referred to as the \textit{partition function}. 
%	\item We also assume that we are given one realization from this model [cf.~\cite{Gidas1988}], which is so often the case in real world networks.
%\end{itemize}
A pivotal quantity of interest which has attracted extensive attention in the literature is the average sum of spins/magnetization density, defined by	
\begin{equation*}\label{eq:pivot}
\oms:=\frac{\sum_{i=1}^N \sigma_i}{N}.
\end{equation*}
%It is often referred to as the \textit{global magnetization density}. As defined above, the fluctuations of $\Theta_N$ have close connections to the estimation of $(\beta,B)$ [\cite{ghosal2018joint}], sparse signal detection [\cite{Mukherjee2018,mukherjee2019testing}], understanding phase transition phenomena [\cite{Ellis1978,kabluchko2019fluctuations}], etc. 
The fluctuations for $\oms$ are mostly known for very few choices of the graph $G_N$, including the complete graph (see e.g.,~\cite{Cha2011,Peter2010,Ellis1978}), the directed Erd\H{o}s-R\'enyi graph %on a restricted range for $(\beta\in (0, 1], B=0)$ %[see~\cref{eg:erg}~and~\cref{rem:ergmod}] where fluctuations have been studied only when $B=0$ and $\beta\leq 1$
(see~\cite{Kabluchko2019fluctuations}), sparse Erd\H{o}s-R\'enyi graphs (see~\cite{giardina2015annealed}). In this paper, we focus on studying fluctuations of $\oms$, when $A_N$ is the scaled adjacency matrix of an approximately regular graph $G_N$. The motivation for this work is the recent paper \cite{Basak2017}, where the authors show universal asymptotics of the partition function $Z_N(\beta,B)$ on {\it any} sequence of approximately regular graphs with diverging average degree, which is governed by the Mean-Field prediction formula. % (see \cite[Theorem 1.1]{Basak2017} for an exact statement).
In particular, it follows from \cite[Theorem 2.1]{Basak2017} that the Mean-Field prediction formula is asymptotically universal in the sense that $$\frac{1}{N}\log{Z_N(\beta,B)}\overset{N\to\infty}{\longrightarrow}\sup_{\mathbf{x}\in [-1,1]} \left\{\frac{\beta x^2}{2}+Bx-\frac{1+x}{2}\log{\frac{1+x}{2}}-\frac{1-x}{2}\log{\frac{1-x}{2}}\right\}$$ for any sequence of approximately $d_N$ regular graphs $G_N$ with $d_N\rightarrow\infty$. A natural follow up question is to what extent this universality extends to other properties of such \enquote{Mean-Field} Ising models. In this paper we try to address this question by studying the universal behavior of the statistic $\oms$.

Our main results (see Theorems~\ref{theo:uniq2}~---~\ref{theo:uniq1}) show that $\oms$ exhibits universal fluctuations for a large class of ``approximately regular" graphs with $\bar{d}_N$ diverging ``fast enough", across all parameter regimes for $(\beta,B)$. Our proof techniques yield tight error bounds for the Mean-Field approximation of the partition function (see~\cref{thm:part}), exponential concentration for the average of spins (see~\cref{thm:conc}~and~\cref{cor:expconc}) and tail bounds for various statistics of interest (see~Lemmas~\ref{lem:qformuni}~---~\ref{lem:linfbdgen}). One of our main contributions is that our results hold even if the minimum and maximum eigenvalue of $A_N$ have the same magnitude asymptotically (see~\cref{rem:connected}). Our assumptions on $A_N$ are thus significantly weaker than the expander type assumptions prevalent in the literature. For ease of exposition, in~\cref{sec:pfideas}, we outline our proof techniques in the special case where $G_N$ is regular.

\subsection{Main results}\label{sec:mres}

We begin with a definition which partitions the parameter set $\{(\beta,B):\beta>0, B\in \R\}$ into different domains.
\begin{definition}
	Let \begin{align*}\Theta_{11}:=\{(\beta,0):0<\beta<1\},&\qquad  \Theta_{12}:=\{(\beta,B):\beta>0, B\ne 0\},\\ \Theta_2:=\{(\beta,0):\beta>1\}, &\qquad \Theta_3:=(1,0).\end{align*} Finally, let $\Theta_1:=\Theta_{11}\cup \Theta_{12}$. We will refer to $\Theta_1$ as the uniqueness regime, $\Theta_2$ as the non uniqueness regime, and $\Theta_3$ as the critical point.
	The names of the different regimes are motivated by the next lemma, the proof of  which follows from simple calculus (see for e.g. ~\cite[Page 144, Section 1.1.3]{AmirAndrea2010}).
	\\
	
\end{definition}
\begin{lemma}\label{lem:fixsol}
	Consider the fixed point equation 
	\begin{equation}\label{eq:mainfix}
	\phi(x)=0,\text{ where } \phi(x):=x-\tanh(\beta x+B).
	\end{equation}
	\begin{enumerate}
		\item[(a)] If $(\beta,B)\in \Theta_{11}$, then~\eqref{eq:mainfix} has a unique solution at $t=0$, and $\phi'(0)>0$.
		\item[(b)] If $(\beta,B)\in \Theta_{12}$, then~\eqref{eq:mainfix} has a unique root $t$ with the same sign as that of $B$, and $\phi'(t)>0$.
		\item[(c)] If $(\beta,B)\in \Theta_2$, then \eqref{eq:mainfix} has two non zero roots $\pm t$ of this equation, where $t>0$, and $\phi'(\pm t)>0$.
		\item[(d)] If $(\beta,B)\in \Theta_3$,  then~\eqref{eq:mainfix} has a unique solution at $t=0$, and $\phi'(0)=0$.
	\end{enumerate} 
	% Further, with $t$ defined as above, we have $\beta(1-t^2)\le1$, with equality iff $(\beta,B)\in \Theta_3$.
\end{lemma} 
We will use $t$ as defined in the above lemma throughout the paper, noting that $t$ does depend on $(\beta,B)$. The following result summarizes the fluctuations of $\oms$ in the Curie-Weiss model (see \cite{Ellis1978}), which is the Ising model on the complete graph.
\begin{lemma}\label{lem:cw_known}
	Suppose $\ms$  is a random vector from the Curie Weiss model $\P^{CW}$ with p.m.f.
	\begin{align}\label{eq:cw_new}
	\P^{CW}(\ms)=\frac{1}{Z_N^{CW}(\beta,B)}\exp\Big(\frac{N\beta}{2}\bar{\ms}^2+B\sum_{i=1}^N\sigma_i\Big).
	\end{align}
	Let $Z_\tau\sim N(0,\tau)$ with $\tau:=\frac{1-t^2}{1-\beta(1-t^2)}$ for $(\beta,B)\notin \Theta_3$, and let $W$ be a continuous random variable with density proportional to $e^{-x^4/12}$. Then the following holds:
	\begin{eqnarray*}%\label{eq:highs}
		\sqrt{N}\Big(\oms-t\Big)\stackrel{d}{\rightarrow}&Z_\tau&\text{ if }(\beta,B)\in \Theta_1,\\
		%\sqrt{N}(\oms-t)\stackrel{d}{\rightarrow}&Z&\text{ if }B\ne 0,\\
		\sqrt{N}\Big(\oms-M(\ms)\Big)\stackrel{d}{\rightarrow}&Z_\tau&\text{ if }(\beta,B)\in \Theta_2,\\
		N^{1/4}\oms\stackrel{d}{\rightarrow}&W&\text{ if }(\beta,B)\in \Theta_3.
	\end{eqnarray*}
	Here $M(\ms)$ is a random variable which equals $t$ if~$\oms\ge 0$, and $-t$ otherwise, whenever $(\beta,B)\in\Theta_2$.
\end{lemma}

We will now explore to what extent the fluctuations of $\bar{\ms}$ are universal. We need the following notations to state our main results.
\\

\begin{definition}
	\begin{enumerate}
		\item[(i)]
		Given two positive sequences $x_N,y_N$, we use the notation $x_N\lesssim y_N$ to denote the existence of a finite constant $C$ free of $N$, such that $x_N\le C y_N$.
		
		\item[(ii)]Given a symmetric matrix $A_N$, 
		let $
		R_i:=\sum_{j=1}^NA_N(i,j)$ denote the row sums of $A_N$, and let $(\lambda_1(A_N), \cdots,\lambda_N(A_N))$ denote its eigenvalues arranged in decreasing order. Let $\lVert A_N\rVert_F$ and $\lVert A_N\rVert_{\text{op}}$ denote the Frobenius norm and the operator norm of $A_N$ respectively.

		\item[(iii)]Given two real valued random variables $X,Y$, define the Kolmogorov-Smirnov distance between $X$ and $Y$ by
		\[d_{KS}(X,Y):=\sup_{x\in \R}|\P(X\le x)-P(Y\le x)|.\]
	\end{enumerate}
\end{definition}

\begin{theorem}\label{theo:uniq2}
	Suppose that $(\beta,B)\in \Theta_1$. Assume further that the sequence of matrices $A_N$ satisfies the following two conditions:
	%\item[\tz] $\sum_{i=1}^N R_i=N+o(v_N)$.
	\begin{align}\label{eq:A1}
	\max_{1\leq i\leq N}R_i\lesssim 1,\\
	\label{eq:A2}\lim_{N\rightarrow\infty}\lambda_1(A_N)=1.
	\end{align}
	%\item[\tb] $\sum_{i} (R_i-1)^2=o(k_N)$ where $k_N$ will be specified depending on the theorem.
	%\item[\tc] $\lambda_1(A_N)\to 1$ as $N\to\infty$.
	%\item[\td] $\fa=o(r_N)$ where $r_N$ will be specified depending on the theorem.
	If $\ms$ is a random vector from the Ising model \eqref{eq:model}, then we have
	\begin{equation}\label{eq:uniq2main}
	d_{KS}\Big(\sqrt{N}(\oms-t), Z_\tau\Big)\lesssim \frac{1}{\sqrt{N}}\left(\fa+\sum_{i=1}^N (R_i-1)^2+t\bigg|\sum_{i=1}^N (R_i-1)\bigg|\right),
	%\sup\limits_{x\in\mathbb{R}}\Big|\mmp \Bigg(\sqrt{N}(\bar{\sigm-t)\leq x\Big)-\mmp \Big(\sqrt{\frac{1-t^2}{1-\beta(1-t^2)}}\cdot Z\leq x\Big)\Big|.
	\end{equation}
	where $Z_{\tau}$ is defined as in~\cref{lem:cw_known}.
	%By a simple symmetry argument, if $B<0$, the above result holds with $\sle$ instead of $\sg$. In fact, the above theorem also holds for $B=0$ and $\beta<1$ with $\sg$ replaced by $0$.
\end{theorem}
%Thus  \eqref{eq:uniques} holds, as soon as the right hand side of \eqref{eq:uniq2main} converges to $0$. Using this result, it follows that \eqref{eq:uniques} holds for any sequence of $d_N$ regular graphs with $d_N\gg \sqrt{N}$ (see section \ref{sec:exmp}). 
Note that Theorem \ref{theo:uniq2} leaves out the parameter regime $\Theta_2\cup \Theta_3$. The following example shows that such a universal behavior is not expected in this parameter regime, unless we assume some notion of connectivity for $A_N$.
\\

\begin{example}\label{eg:connectedness}
	With $N$ even, let $A_N$ be the adjacency matrix of two disjoint complete graphs $K_{N/2}$, scaled by $N/2$. Then the following holds:
	\begin{enumerate}
		\item[(a)]
		If $(\beta,B)\in \Theta_2$, then
		$\oms\stackrel{d}{\rightarrow}\frac{1}{2}\delta_0+\frac{1}{4}(\delta_t+\delta_{-t})$.
		
		\item[(b)]
		If $(\beta,B)\in \Theta_3$, then
		$N^{1/4}\overline{\ms}\stackrel{d}{\rightarrow}(W_1+W_2)/2^{3/4}$, where $W_1,W_2$ are i.i.d. with the same distribution as that of $W$, with $W$ defined as in~\cref{lem:cw_known}.
	\end{enumerate}
	
\end{example}
The above example shows that if we want universal fluctuations in the regimes $\Theta_2\cup \Theta_3$, the matrix $A_N$ needs to be \enquote{connected} in some asymptotic sense. If $A_N$ is exactly the adjacency matrix of a $d_N$ regular graph $G_N$ scaled by $d_N$, then $\lambda_1(A_N)=1$, and it is easy to check that the graph $G_N$ is connected iff there is a spectral gap, i.e., $\lambda_2(A_N)<1$. Motivated by this, we propose the following asymptotic notion of a spectral gap.
\\

%In~\cref{eg:connectedness}, we show that the above assumptions are not enough to yield universal fluctuations for $\Theta_N$ when $B=0$ and 
\begin{definition}%$\beta\geq 1$. A notion of ``well-connectedness" is of essence in that regime. Therefore, we add in the following assumption:
	We say a sequence of symmetric matrices $\{A_N\}_{N\ge 1}$ with non- negative entries satisfies the spectral gap condition, if
	\begin{align}\label{eq:well_connect}
	\limsup\limits_{N\to\infty}\frac{\lambda_2(A_N)}{\lambda_1(A_N)}<1.
	\end{align}
\end{definition}

%It is well-known that the second eigenvalue of the adjacency matrix captures the strength of connectedness in a regular graph. The above assumption aims to mimic the same. It is worth pointing out that~\te~is quite robust, 
We note that assumption \eqref{eq:well_connect} is somewhat weak in the sense that it does not imply connectivity in general. In particular this allows the existence of small disconnected sub-graphs in $G_N$, as shown in the following example.
\begin{example}\label{prop:robseceg}
	Let $G_N$ denote a graph which is the disjoint union of a $d_N$ regular graph $G_{1,N_1}$ on $N_1$ vertices, and an arbitrary graph $G_{2,N_2}$ on $N_2$ vertices, with $N_1+N_2=N$ and $N_2=o(d_N)$. Then the average degree of the whole graph $G_N$ is $\tilde{d}_N=d_N(1+o(1))$. It is easy to check that if $G_{1,N_1}$ satisfies \eqref{eq:well_connect}, then $G_N$ satisfies \eqref{eq:well_connect}, even though $G_N$ is disconnected.
\end{example}
%if $A_N$ denotes the (scaled) adjacency matrix of a simple graph (see~\cref{def:ScaledAdj}) which has a ``smaller" disconnected component compared to a ``larger" connected component, even then the assumption continues to hold (see~\cref{prop:robseceg} for details). 
Under the assumption of a spectral gap, our next result shows universal fluctuations in the non-uniqueness regime.
\begin{theorem}\label{theo:nouniq}

	Suppose that  $(\beta,B)\in \Theta_2$. Assume further that the sequence of matrices $A_N$ satisfies \eqref{eq:A1},\eqref{eq:A2}, and \eqref{eq:well_connect}. %the following two conditions:
	%\item[\tz] $\sum_{i=1}^N R_i=N+o(v_N)$.
	%	\begin{align}\label{eq:A1}
	%	\max_{1\leq i\leq N}R_i\lesssim 1,\\
	%	\label{eq:A2}\lim_{N\rightarrow\infty}\lambda_1(A_N)=1.
	%	\end{align}	
	%\item[\tb] $\sum_{i} (R_i-1)^2=o(k_N)$ where $k_N$ will be specified depending on the theorem.
	%\item[\tc] $\lambda_1(A_N)\to 1$ as $N\to\infty$.
	%\item[\td] $\fa=o(r_N)$ where $r_N$ will be specified depending on the theorem.
	If $\ms$ is a random vector from the Ising model \eqref{eq:model}, then we have
	%Suppose $\ms$ is drawn according to~\eqref{eq:model} with $B=0$,  $\beta>1$ and $A_N$ satisfies~\ta,~\tc~and~\te. Then we have:
	\begin{equation}\label{eq:non_uniq}
	%\sup\limits_{x\in\mathbb{R}}\Big|\mmp \Big(\sqrt{N}(\bar{\sigma}_N-M(\ms))\leq x\Big)-\mmp \Big(\sqrt{\frac{1-t^2}{1-\beta(1-t^2)}}\cdot Z\leq x\Big)\Big
	d_{KS}\Big(\sqrt{N}(\oms-M(\ms)),Z_\tau\Big)\lesssim \frac{1}{\sqrt{N}}\left(\fa+\sum_{i=1}^N (R_i-1)^2+\bigg|\sum_{i=1}^N (R_i-1)\bigg|\right).
	\end{equation}
	where $M(\ms)$ and $Z_{\tau}$ are defined as in Lemma \ref{lem:cw_known}.
	
	%=\so\mmo(T_N\geq 0)-\so\mmo(T_N < 0)$. By symmetry $\mmp (\sqrt{N}(\Theta_N-M(\ms))\leq x)$ in the above display can be replaced by $\mmp (\sqrt{N}(\Theta_N-\ut|\Theta_N>0)\leq x)$.
\end{theorem}
%As before, it follows  from the above theorem that \eqref{eq:lows} holds for any sequence of regular graphs $G_N$ with $d_N\gg \sqrt{N}$ (c.f.~\ref{sec:expm}).
%Our next theorem deals with the critical regime $\Theta_3$. % and gives the convergence rates in terms of $\ell^\infty$ norm, as opposed to the $\ell^2$ norm used in the previous two theorems. %In both the theorems above, the error bounds are in terms of $\ell^2$ norm of the vector $R:=A_N{\bf 1}-{\bf 1}$. However, in the critical regime $\beta=1, B=0$, because of a cancellation we inherently need to work with $\ell^3$ norm of $R$, which we estimate by the $\ell^\infty$ norm. Thus the error bounds 
%A typical feature of this combination of parameters is that the $\sigma_i$'s are very strongly dependent thereby yielding non-Gaussian limits for $\Theta_N$ as demonstrated below. 
% we will work under the following extra assumptions:
%\begin{itemize}
%	\item[\tf] $\max_{1\leq i\leq N} |R_i-1|\lesssim o(\omega_N)$.
%	\item[\tg] $\max_{1\leq i\leq N}\sum_{j} A_N(i,j)^2\lesssim o(\pi_N)$.
%\end{itemize}
%Detailed comments on the above assumptions can be found in~\cref{rem:supctrl}. Now let us state the the corresponding theorem.
To prove universal fluctuations in the critical regime, we need a stronger notion of regularity on $A_N$, i.e.,
\begin{equation}\label{eq:A4}
\limsup\limits_{N\to\infty} N^{1/4}\max\limits_{1\leq i\leq N} |R_i-1|\lesssim 1.
\end{equation}
\begin{theorem}\label{theo:crit}
	
	Suppose that $(\beta,B)\in \Theta_3$. %Assume further that the sequence of matrices $A_N$ satisfies the following two conditions:
	%\item[\tz] $\sum_{i=1}^N R_i=N+o(v_N)$.
	%	\begin{align}\label{eq:A1}
	%	\max_{1\leq i\leq N}R_i\lesssim 1,\\
	%	\label{eq:A2}\lim_{N\rightarrow\infty}\lambda_1(A_N)=1.
	%	\end{align}
	%\item[\tb] $\sum_{i} (R_i-1)^2=o(k_N)$ where $k_N$ will be specified depending on the theorem.
	%\item[\tc] $\lambda_1(A_N)\to 1$ as $N\to\infty$.
	%\item[\td] $\fa=o(r_N)$ where $r_N$ will be specified depending on the theorem.
	If $\ms$ is a random vector from the Ising model \eqref{eq:model} where $A_N$ satisfies~\eqref{eq:well_connect}~and~\eqref{eq:A4}. Then we have %setting $\alpha_N:=\max_{1\le i\le N}\sum_{j=1}^NA_N(i,j)^2$we have
	%	Suppose $\ms$ is drawn according to~\eqref{eq:model} with $B=0$,  $\beta=1$ and $A_N$ satisfies~\ta,~\te,~\tf~and~\tg~with $\pi_N=N^{-1/2}$ and $\omega_N=N^{-1/4}$. Then we have:
	\begin{comment}
	\begin{equation}
	d_{KS}\Big(N^{1/4}\oms,W\Big)\lesssim \frac{1}{\sqrt{N}}+\frac{(\log N)^2}{N^{1/4}}\Big[\sum_{i=1}^N	
	%N^{1/2}\alpha_N\log N+\Big(\frac{\log N}{\sqrt{N}}\Big)^2\sqrt{N\max_{1\le i\le N}(R_i-1)^2+|\sum_{i=1}^N(R_i-1)|}.
	%\sup\limits_{x\in\mathbb{R}}\Bigg|\mmp \Bigg(N^{1/4}\Theta_N\leq x\Bigg)-\mmp (W\leq x)\Bigg|\lesssim \frac{\sqrt{\log{N}\log{\log{N}}}}{\sqrt{N}}\fa+\frac{(\log{N})^2}{N^{1/4}}\sqrt{\left(\sum_{i=1}^N (R_i-1)^2+\bigg|\sum_{i=1}^N (R_i-1)\bigg|\right)}
	\end{equation}
	\end{comment}
	\begin{equation}\label{eq:critmain}
	d_{KS}\Big(N^{1/4}\oms,W\Big)\lesssim \frac{\varepsilon_N}{\sqrt{N}}+\frac{\varepsilon_Nr_N}{{N^{1/4}}}+\frac{(\log N)^2}{N^{1/4}}\sqrt{\sum_{i=1}^N (R_i-1)^2+N^{-1/2}\Big[\sum_{i=1}^N(R_i-1)\Big]^2},
	%\sup\limits_{x\in\mathbb{R}}\Bigg|\mmp \Bigg(N^{1/4}\Theta_N\leq x\Bigg)-\mmp (W\leq x)\Bigg|\lesssim \frac{\sqrt{\log{N}\log{\log{N}}}}{\sqrt{N}}\fa+\frac{(\log{N})^2}{N^{1/4}}\sqrt{\left(\sum_{i=1}^N (R_i-1)^2+\bigg|\sum_{i=1}^N (R_i-1)\bigg|\right)}
	\end{equation}
	where
	\begin{small}
		\begin{align*}
		r_N:=&\sqrt{(\log{N})^3\max_{1\le i\le N}\sum_{j=1}^NA_N(i,j)^2}+\log N\max\limits_{1\leq i\leq N}|R_i-1|,\\
		\varepsilon_N:=&\fa+\frac{1}{N}\Big[\sum_{i=1}^N (R_i-1)\Big]^2+\frac{1}{N}\sum_{i=1}^N (R_i-1)^2+\log{N},
		\end{align*}
	\end{small}
	and $W$ is as in Lemma \ref{lem:cw_known}. %has density proportional to $\exp\left(-w^4/12\right)$. 	
\end{theorem}

\begin{remark}
	%It follows from the above theorems that for a sequence of well connected  regular matrices (i.e. $R_i\equiv 1$), the condition $\lVert A_N\rVert_F^2\ll \sqrt{N}$ is sufficient for universal fluctuations if $(\beta,B)\in \Theta_1\cup \Theta_2$, and $\max_{1\le i\le N}\sum_{j=1}^NA_N(i,j)^2\ll (\sqrt{N}\log N)^{-1}$ is sufficient for universal fluctuations of $\oms$ if $(\beta,B)\in \Theta_3$. 
	Using these results, in section \ref{sec:exmp} we will show that for any sequence of $d_N$ regular graphs satisfying the spectral gap condition (see~\eqref{eq:well_connect}), the fluctuation of $\oms$ is universal in $\Theta_1\cup \Theta_2$ if $d_N\gg \sqrt{N}$ ,  and in $\Theta_3$ if $d_N\gg \sqrt{N}\log N$. We now give an example to show that the above conditions are actually tight (up to $\log$ factor in the critical regime). The proof of this example will appear in an upcoming draft \cite{Manuel2020}.
\end{remark}

\begin{example}\label{eg:sparsity}
	Let $G_N$ denote the line graph of the complete graph $K_n$, so that $N={n\choose 2}=\frac{n^2}{2}(1+o(1))$. This is a regular graph with degree $d_N=2(n-2)=2\sqrt{2N}(1+o(1))$, and its top two eigenvalues are $\lambda_1(G_N)=2(n-2)$ and $\lambda_2(G_N)=n-2$ (see \cite[Lemma 2]{Chuang2009}). It follows that $A_N=\frac{1}{d_N}G_N$ does satisfy \eqref{eq:well_connect}, and 
	$$\lim_{N\rightarrow\infty}\frac{1}{\sqrt{N}}\lVert A_N\rVert_F^2=\sqrt{N}\max_{1\le i\le N}\sum_{j=1}^NA_N(i,j)^2=\frac{1}{2\sqrt{2}}\ne 0.$$
	%but every other term in the RHS of Theorems \ref{theo:uniq2}, \ref{theo:nonuniq}, \ref{theo:critic
	In this case we have the following limiting distributions across different regimes:
	%$G_{M}^{\mathrm{com}}$ denote the complete graph with $\sqrt{N}$ vertices. Denote the set of edges in $G_{M}^{\mathrm{com}}$ as $S_{M}^{\mathrm{com}}$ and fix an arbitrary enumeration of $S_{M}^{\mathrm{com}}$, say, $\{s_1,s_2,\ldots ,s_L\}$ where $L=M(M-1)/2$. Therefore, $L/N\to 1/2$. Next, define $A_N(i,j):= (2(M-2))^{-1}$ if $s_i$ and $s_j$ share exactly one common neighbor and $0$ otherwise. As defined above, $A_N$ is the \emph{scaled adjacency matrix of the line graph of a complete graph}. In this case, $R_i=1$ for $1\leq i\leq N$. Also $\lambda_2(A_N)\longrightarrow 1/2$; see e.g.,~\cite[Lemma 2]{Chuang2009}. In particular, all the required assumptions for Theorems~\ref{theo:uniq2},~\ref{theo:nouniq}~and~\ref{theo:crit} are satisfied, apart from~\td. In this example, it is easy to verify that $\fa/\sqrt{N}\to 1/2$. In an upcoming paper (\cite{Manuel2020}), the following theorems will be shown:
	\begin{eqnarray*}
		\sqrt{N}(\overline{\ms}_N-t)+\mu\overset{w}{\longrightarrow}&Z_\tau& \mbox{if }(\beta,B)\in \Theta_1,\\
		\sqrt{N}(\overline{\ms}_N-M(\ms))+\mathrm{sgn}(M(\ms))\mu\overset{w}{\longrightarrow}&Z_\tau&\mbox{if }(\beta,B)\in \Theta_2,\\
		N^{1/4}\overline{\ms}_N\overset{w}{\longrightarrow}&\tilde{W}& %\mbox{with density }\propto \exp\left(-c_1 w^4-c_2 w^2\right)\qquad 
		\mbox{ if } (\beta,B)\in \Theta_3,
	\end{eqnarray*}
	where $\mu:=\frac{\beta t}{\sqrt{2}(1-\beta(1-t^2))\cdot (2-\beta (1-t^2))}$ is strictly larger than $0$ if $(\beta,B)\in \Theta_{12}\cup \Theta_2\cup \Theta_3$, and $\tilde{W}$ has density proportional to $\exp(-\frac{w^4}{12}-\frac{w^2}{\sqrt{2}})$. Therefore, the fluctuations do not match that of the Curie-Weiss model unless  $(\beta,B)\in \Theta_{11}$.
	
\end{example}
%Note that the above theorems demonstrate the universal behavior and phase transitions of $T_N$ (the fluctuation) for $r_N=\sqrt{N}$ (up to logarithmic factors only if $B=0$ and $\beta=1$). In general, unless $B=0$ and $\beta<1$, this condition on $r_N$ is tight; see~\cref{eg:sparsity}. However, when $B=0$, $\beta<1$, the above condition is not tight and it is possible to come up with an improved bound, which proceeds under the following assumption:
%\begin{itemize}
%	\item[\thh] For any fixed $\epsilon>0$, we have $\#\{k:|R_k-1|\geq \epsilon\}=o\bigg(\big(\max_{1\leq i\leq N} \sum_j A_N(i,j)^2\big)^{-1}\bigg)$
%\end{itemize}
%which is again a necessary assumption (see~\cref{eg:CWweakreg}). 
Note that in the above example, $\oms$ has a different limit compared to the Curie-Weiss model in $\Theta_{12}\cup \Theta_2\cup \Theta_3$, but continues to have universal fluctuations in the high parameter regime $\Theta_{11}$. We now state a modified theorem for the regime $\Theta_{11}$, which shows that in this regime we can do better.
\\

\begin{theorem}\label{theo:uniq1}
	Suppose that  $(\beta,B)\in \Theta_{11}$, and $A_N$ satisfies
	\begin{align}\label{eq:A3}
	\lim_{N\rightarrow\infty}\max_{1\le i\le N}R_i=1.
	\end{align} %Assume further that the sequence of matrices $A_N$ satisfies \eqref{eq:A1},\eqref{eq:A2}, and \eqref{eq:well_connect}.
	If $\ms$ is a random vector from the Ising model \eqref{eq:model}, then setting $\alpha_N:=\max_{1\le i\le N}\sum_{j=1}^NA_N(i,j)^2$ we have
	% $A_N$ satisfies~\ta,~\tg~and~\thh~with $\pi_N=N^{-1/3}$, . Then we have:
	\begin{equation}\label{eq:uniq1main}
	d_{KS}\Big(\sqrt{N}\oms,Z_\tau\Big)\lesssim \frac{1}{\sqrt{N}}+\frac{\fa\sqrt{\alpha_N\log N}}{\sqrt{N}}+\Big[1+\lVert A_N\rVert_F \alpha_N\log N \Big]\sqrt{\frac{\sum_{i=1}^N (R_i-1)^2}{N}},
	% N^{1/3}\sqrt{\log{N}}\max_{1\le i\le N}\sum_{j=1}^NA_N(i,j)^2+(\log{N})^{3/4}\max_{1\le i\le N}|R_i-1|+N^{-1/2}.
	%\sup\limits_{x\in\mathbb{R}}\big|\mmp \big(\sqrt{N}\Theta_N\leq x\big)-\mmp \big((1-\beta)^{-1/2}\cdot Z\leq x\big)\big|\lesssim N^{-2/3}\sqrt{\log{N}}\fa+\sqrt{\frac{(\log{N})^{3/2}}{N}\sum_{i=1}^N (R_i-1)^2}+N^{-1/2} .
	\end{equation}
	where $Z_{\tau}$ is defined as in~\cref{lem:cw_known}.
\end{theorem}
\begin{remark}
	It follows from the above result that in the regime $\Theta_{11}$, $\oms$ has universal fluctuations on regular graphs of degree $d_N\gg (N\log N)^{1/3}$. We believe this is not tight, and universal fluctuations should hold on any sequence of regular graphs with $d_N\rightarrow\infty$. In \cite{Kabluchko2019fluctuations} the authors prove such a result when $G_N$ is a non symmetric Erd\H{o}s-R\'enyi graph in the regime $\Theta_{11}$ (details in example section below).
\end{remark}
Note that we only expect a similar behavior as in the Curie Weiss model, if the underlying graphs are approximately regular and have large degree. Quantifying this philosophy, the bounds in each of the theorems have two terms, the first term controls the sparsity of the underlying graph/matrix, and the second term controls the extent of regularity of the graph/matrix. Recall example \ref{eg:sparsity}, which suggests that the term controlling the sparsity is optimal. In a similar spirit, the following example suggests that the term controlling the extent of regularity is also optimal.
\\

\begin{example}\label{eg:regularity1}
	%First we will show that when $B>0$, the assumption~\tb~is necessary. Once again, assume that $\sqrt{N}$ is an integer. 

	\begin{enumerate}
		\item[(a)]
		Assume that $\sqrt{N}$ is an integer, and let $G_N$ be the disjoint union of two complete graphs of size $N-\sqrt{N}$ and $\sqrt{N}$ respectively.  Let $\overline{d}_N$ denote the average degree of $G_N$ and $A_N=(\overline{d}_N)^{-1}G_N$.  In this case $$\lim_{N\rightarrow\infty}\frac{1}{\sqrt{N}}\sum_{i=1}^N(R_i-1)^2>0,$$
		but every other term in the RHS of \eqref{eq:uniq2main} converges to $0$. If $\ms$ is a random vector from the Ising model \eqref{eq:model} with $B\ne 0$, then 
		$\sqrt{N}\big(\oms-t\big)\overset{w}{\longrightarrow}\mu+Z_\tau$, where $\mu:=\frac{\beta t(1-t^2)}{1-\beta(1-t^2)}+\tanh(B)-t\ne 0$.
		
		\item[(b)]
		With $G_N=K_N$, let $A_N=\frac{1}{N-\sqrt{N}}G_N$. In this case $$\lim_{N\rightarrow\infty}\frac{1}{\sqrt{N}}\sum_{i=1}^N(R_i-1)>0,$$ but every other term in the RHS of \eqref{eq:uniq2main} converges to $0$. If $\ms$ is a random vector from the Ising model \eqref{eq:model} with $B\ne 0$, then 
		$\sqrt{N}\big(\oms-t\big)\overset{w}{\longrightarrow}\mu+Z_\tau$, where $\mu:=\frac{\beta t(1-t^2)}{1-\beta(1-t^2)}\ne 0$.
	\end{enumerate}

\end{example}

The main ingredient of our proof technique is comparing the Ising model on an approximately regular graph to that of an i.i.d. model/Curie-Weiss model. As a byproduct of this approach, we also obtain quantitative bounds for the following asymptotics of the log partition function via the Mean-Field prediction formula, defined via the following lower bound (c.f. \cite{Basak2017}):
\begin{align*}
\log{Z_N(\beta,B)}\ge \sup_{\ms\in [-1,1]^N}\Big\{\frac{\beta}{2}\ms^{\top}A_N\ms+B\sum_{i=1}^N\sigma_i-\sum_{i=1}^NI(\sigma_i)\Big\},
\end{align*}
where $I(x):=\frac{1+x}{2}\log \frac{1+x}{2}+\frac{1-x}{2}\log \frac{1-x}{2}$ is the binary entropy function. 
By choosing $\ms=t{\bf 1}$ with $t$ as defined in Lemma \ref{lem:fixsol}, we get the further lower bound
\begin{align}\label{eq:mf_2}
\log{Z_N(\beta,B)}\ge N\Big\{\frac{\beta t^2}{2}+Bt-I(t)\Big\}+\frac{\beta t^2}{2}\sum_{i=1}^N (R_i-1)=:\mathcal{M}_N(\beta,B).
\end{align}
It follows from \cite[Theorem 2.1]{Basak2017} that $\log Z_N(\beta,B)-\mathcal{M}_N(\beta,B)=o(N)$, as soon as  $\fa+\sum_{i=1}^N(R_i-1)^2=o(N)$. Our next result gives a bound to the approximation error of the partition function $Z_N(\beta,B)$ by $\mathcal{M}_N(\beta,B)$, which we henceforth refer to as the Mean-Field prediction in this paper.
\\

%partition function of the Curie-Weiss model, and the matrix $A_N$.
%{\color{red} right hand side of the equation above.}
\begin{theorem}\label{thm:part}
	
	\begin{enumerate}
		
		Let $A_N$ satisfy \eqref{eq:A1} and \eqref{eq:A2}.  
		
		\item[(a)]
		If $(\beta,B)\in \Theta_1$  then we have
		\begin{align*}
		\log Z_N(\beta,B)-\mathcal{M}_N(\beta,B)\lesssim \fa+t^2\sum_{i=1}^N(R_i-1)^2.
		\end{align*}

		\item[(b)]
		If $(\beta,B)\in \Theta_2$, then the same conclusion as in part (a) holds under the extra assumption that $A_N$ satisfies \eqref{eq:well_connect}.
		%\begin{align*}
		% \log Z_N(\beta,B)-N\left\{\frac{\beta}{2}t^2+Bt-I(t)\right\}\lesssim \sum_{i=1}^N(R_i-1)^2+\fa+\Big|\sum_{i=1}^N(R_i-1)\Big|.
		%\end{align*}
		
		\item[(c)]
		If $(\beta,B)\in \Theta_3$, then under the extra assumption that $A_N$ satisfies \eqref{eq:well_connect} we have
		\begin{align*}
		\log Z_N(\beta,B)-\mathcal{M}_N(\beta,B)\lesssim \fa+\frac{1}{N}\Big[\sum_{i=1}^N(R_i-1)^2\Big]^2+\frac{1}{N}\Big[\sum_{i=1}^N(R_i-1)\Big]^2+\log N.
		\end{align*}
	\end{enumerate}
\end{theorem}
\begin{remark}\label{rem:impart}
	To see how the error bounds of the above theorem compare to existing error bounds for the Mean-Field prediction formula in the literature, let us take the example where $A_N$ is the (scaled) adjacency matrix of a $d_N$-regular graph $G_N$.
	%(see~\cref{def:ScaledAdj}). 
	In this case,  the above theorem gives the error bound $O(N/d_N)$ for the Mean-Field prediction formula. This immediately improves the bounds from~\cite[Theorem 1.1]{Basak2017} --- $o(N)$,~\cite[Theorem 1.1]{jain2019mean} --- ${O}(N/d_N^{1/3})$,~\cite[Example 3]{eldan2018taming} --- ${O}(N/d_N^{1/2-o(1)})$ under strong expander type conditions not needed here)~and~\cite[Corollary 2.9 and Example 2.10]{augeri2019transportation} --- ${O}(N/\sqrt{d_N})$. 
\end{remark}
For our next result, define an i.i.d. probability measure $\mmq$ on $\{-1,1\}^N$ by setting 
\begin{equation}\label{eq:realdist1}
\mmq(\sigma_1,\ldots ,\sigma_N):= \left(\exp(-\beta t-B)+\exp(\beta t+B)\right)^{-N}\exp\left((\beta t+B)\sum_{i=1}^N \sigma_i\right).
\end{equation}
Our next theorem shows that if an event is unlikely under the above i.i.d. measure/ the Curie Weiss model (depending on $(\beta,B)$), then it is also unlikely under an Ising model on an approximately regular graph with large degree.

\begin{theorem}\label{thm:conc}
	
	\begin{enumerate}
		
		Let $A_N$ satisfy \eqref{eq:A1} and \eqref{eq:A2}.  Also, let $\mE_N\subset \{-1,1\}^N$ be arbitrary.
		
		\item[(a)]

		If $(\beta,B)\in \Theta_1$, then we have
		\begin{align*}
		\log \P(\mE_N)\lesssim \log \mmq(\mE_N)+\fa+t^2\sum_{i=1}^N(R_i-1)^2.%+t^2\bigg|\sum_{i=1}^N (R_i-1)\bigg|.
		\end{align*}
		
		\item[(b)]
		If $(\beta,B)\in \Theta_2$, then 
		under the further assumption \eqref{eq:well_connect} we have
		\begin{align*}
		\log \P(\mE_N)\lesssim \log \P^{CW}(\mE_N)+\fa+\sum_{i=1}^N(R_i-1)^2.
		\end{align*}
		%the same conclusion as in part (a) holds under the further assumption \eqref{eq:well_connect}.
		
		\item[(c)]
		If $(\beta,B)\in \Theta_3$, then under the further assumption \eqref{eq:well_connect} we have
		\begin{align*}
		\log \P(\mE_N)\lesssim \log \P^{CW}(\mE_N)+ \fa+\frac{1}{N}\Big[\sum_{i=1}^N(R_i-1)^2\Big]^2+\frac{1}{N}\Big[\sum_{i=1}^N(R_i-1)\Big]^2+\log N.
		\end{align*}
	\end{enumerate}
\end{theorem}
As an application of the above theorem, we immediately get the following exponential concentration for $\oms$.

\begin{corollary}\label{cor:expconc}
	Suppose $A_N$ satisfies \eqref{eq:A1}, \eqref{eq:A2}, and 
	$$\lim_{N\rightarrow\infty}\frac{1}{N}\sum_{i=1}^N(R_i-1)^2=0,\qquad \lim_{N\rightarrow\infty}\frac{1}{N}\lVert A_N\rVert^2_F=0.$$
	\begin{itemize}
		\item
		If $(\beta,B)\in \Theta_1$, then for every $\delta>0$ we have
		$$\limsup_{n\rightarrow\infty}\frac{1}{N}\log \P(|\oms-t|>\delta)<0.$$
		The same conclusion holds for $(\beta,B)\in\Theta_3$, under the extra assumption that $A_N$ satisfies \eqref{eq:well_connect}.
		\item
		If $(\beta,B)\in \Theta_2$, then under the extra assumption that $A_N$ satisfies \eqref{eq:well_connect}, for every $\delta>0$ we have
		$$\limsup_{n\rightarrow\infty}\frac{1}{N}\log \P(|\oms-M(\ms)|>\delta)<0,$$
		where $M(\ms)$ is defined as in~\cref{lem:cw_known}.
	\end{itemize}
\end{corollary}
Similar concentration results can be obtained for other higher order polynomials of $\ms$, as studied in \cite{Rados2019, Cha2005,Gheissari2018} and the references therein. However, these papers focus exclusively on the high temperature regime $\Theta_{11}$ whereas our result applies to all temperatures. The references cited above can deal with non Ferromagnetic interactions and general external fields as well. We note in passing that it should be possible to extend our proof technique to general non constant magnetic fields.
%Note that the above theorem does not require assumption~\td~and applies to general $d_N$-regular graphs. %Similar conclusions also follow for the other parameter regimes from Lemmas~\ref{lem:partnonun}~and~\ref{lem:partcrit}, none of which require expander type conditions. 

%We do not need assumption~\tc~in~\cref{theo:uniq1} because assumptions~\ta~and~\thh~imply assumption~\tc~(see~\cref{prop:eigen}).
%We would like to point out that we do not expect~\cref{theo:uniq1} to be tight in the condition imposed on $r_N$. In particular, we believe that the coefficient of $\fa$ should be $N^{-1}$ (may be up to log factors). Note that the limiting distributions above coincide with those from the Curie-Weiss model (see~\cite{Ellis1978}).
%
%Observe that the left hand sides in the main theorems are bounded above by $1$, so it suffices to operate under the following assumptions:
%\begin{itemize}
%	\item[\tz] $\sum_{i=1}^N R_i=N+o(v_N)$.
%	%\item[\ta] $\max_{1\leq i\leq N}R_i\lesssim 1$.
%	\item[\tb] $\sum_{i} (R_i-1)^2=o(k_N)$ where $k_N$ will be specified depending on the theorem.
%	%\item[\tc] $\lambda_1(A_N)\to 1$ as $N\to\infty$.
%	\item[\td] $\fa=o(r_N)$ where $r_N$ will be specified depending on the theorem.
%\end{itemize}
%where $k_N$, $v_N$ and $r_N$ are chosen so that the right hand sides of our main theorems converge to $0$ (depending on the regime). The proofs of the above theorems can be found in~\cref{sec:mainthpfs}.
\subsection{Proof overview}\label{sec:pfideas}
For the sake of simplicity, we focus on the case where $A_N$ is the adjacency matrix of a $d_N$ regular graph scaled by $d_N$. For verifying Theorem \ref{theo:uniq2}, following~\cite[Theorem 2.1]{Cha2011}, form an exchangeable pair $(\ms,\ms')$ as follows:

Let $I$ denote a randomly sampled index from $\{1,2,\ldots ,N\}$. Given $I=i$, replace $\sigma_i$ with an independent $\pm 1$ valued random variable $\sigma'_i$ with mean $\E[\sigma_i|(\sigma_j,j\neq i)]=\tanh(\beta m_i(\ms)+B)$, where $m_i(\ms):=\sum_{j=1}^NA_N(i,j)\sigma_j.$ Then, setting $\ms':=(\sigma_1,\cdots,\sigma_{i-1},\sigma'_i,\sigma_{i+1},\cdots,\sigma_N)$, we have that $(\ms,\ms')$ is an exchangeable pair. With $T_N:=\sqrt{N}(\overline{\boldsymbol{\sigma}}-t)$ and $T_N':=\sqrt{N}(\overline{\boldsymbol{\sigma'}}-t),$ %the pair $(T_N,T_N')$ is exchangeable as we 
%Define $m_i(\ms):=\sum_{j=1}^N A_N(i,j)\sigma_j$. 
a simple computation using a Taylor's series expansion of $\tanh(\beta x+B)$ around $x=t$ gives
\begin{align*}
\E[T_N-T_N'|\ms]=&N^{-3/2}\sum_{i=1}^N\Big(\sigma_i-\tanh(\beta m_i(\ms)+B)\Big)\\
=&N^{-3/2}\bigg[\sum_{i=1}^N\Big(\sigma_i-\tanh(\beta t+B)\Big)+\sum_{i=1}^N(m_i(\ms)-t)\text{sech}^2(\beta t+B)\nonumber \\& \qquad +O_P\Big(\sum_{i=1}^N(m_i(\ms)-t)^2\Big)\bigg],
\end{align*}
Since $G_N$ is regular, we have $\sum_{i=1}^N\sigma_i=\sum_{i=1}^Nm_i(\ms)$. Also $t$ satisfies $t=\tanh(\beta t +B)$, and so the above display gives
\begin{align}\label{eq:non_linear}
\E[T_N-T_N'|\ms]=\frac{T_N}{N}(1-\beta(1-t^2))+N^{-3/2} O_P\Big(\sum_{i=1}^N(m_i(\ms)-t)^2\Big).
\end{align}
% By a two term Taylor expansion, $\tanh(\beta m_i(\ms)+B)=\tanh(\beta M(\ms)+B)+\beta\sech^2(\beta M(\ms)+B)(m_i(\ms)-M(\ms))+\xi_i(m_i-M(\ms))^2$ for bounded random variables $\xi_i$. As $\tanh(\beta M(\ms)+B)=M(\ms)$ and $M(\ms)^2=t^2$ from~\cref{lem:fixsol}, 
%\begin{align}\label{eq:pfeq1}
%\beta\sech^2(\beta M(\ms)+B)\sum_{i=1}^N (m_i-t)=\beta(1-t^2)\sum_{j=1}^N  \left(\sum_{i=1}^N A_N(i,j)\right)(\sigma_j-M(\ms))=\beta(1-t^2)\sqrt{N}T_N.
%\end{align}
%It can be shown that $M(\ms)=M(\ms')$ with high probability (see~\cref{lem:gentail}), using~\eqref{eq:pfeq1}, we then get:
%\begin{align}\label{eq:pfeq5}
%\E[T_N-T_N'|T_N]-\frac{T_N}{N}(1-\beta(1-t^2))\approx \frac{1}{N\sqrt{N}}\sum_{i=1}^N \E[\xi_i(m_i(\ms)-M(\ms))^2|T_N]
%\end{align}
By~\cite[Theorem 1.2]{Cha2011}, $T_N$ approximately satisfies Stein's equation if we show that the second term in the RHS above is negligible, i.e. $S_N:=\sum_{i=1}^N (m_i(\ms)-t)^2=o(\sqrt{N})$. This is the content of Lemma \ref{lem:qformuni}, which bounds the exponential moment of $S_N$ to show that $S_N=O_P\Big(\frac{N}{d_N}\Big)$. Thus we require $d_N\gg \sqrt{N}$ to ensure the linear term in \eqref{eq:non_linear} dominates the error term. The main ingredient for Lemma \ref{lem:qformuni} is a version of the Hanson-Wright inequality for $\{-1,+1\}$ valued random variables (c.f.~Lemma \ref{lem:tailsubG}). Justifying the above steps gives a proof of Theorem \ref{theo:uniq2}. The proof of Theorem \ref{theo:nouniq} follows on similar lines, after replacing $t$ above by $M(\ms)$, where $M(\ms)=t$ if $\bar{\ms}\ge 0$, and $M(\ms)=-t$ otherwise, as defined in~\cref{lem:cw_known}.

The above program does not work for Theorem \ref{theo:crit}, which deals with the critical regime $\Theta_3$. This is because $\beta(1-t^2)=1$, and so the linear term in \eqref{eq:non_linear} vanishes. With $T_N=N^{1/4}\overline{\ms}$, using a Taylor's series expansion of $\tanh(x)$ around $x=\bar{\boldsymbol{m}}(\ms):=N^{-1}\sum_{i=1}^N m_i(\ms)$ and following similar steps as the derivation of \eqref{eq:non_linear} we have
\begin{align}
\begin{split}\label{eq:non_linear2}
\E[T_N-T_N'|\ms]
= N^{-7/4}\bigg[&\sum_{i=1}^N\Big(\sigma_i-\tanh(\bar{\boldsymbol{m}}(\ms))\Big)+\frac{\tanh''(\bar{\boldsymbol{m}}(\ms))}{2}\sum_{i=1}^N(m_i(\ms)-\bar{\boldsymbol{m}}(\ms))^2\\
+&O_P\left(\sum_{i=1}^N\Big|m_i(\ms)-\bar{\boldsymbol{m}}(\ms)\Big|^3\right)\bigg].
\end{split}
\end{align}
Noting that $\bar{\boldsymbol{m}}(\ms)=\overline{\ms}$, the leading term in the RHS of \eqref{eq:non_linear2} equals $N^{-3/4}(\overline{\ms}-\tanh(\overline{\ms}))\approx \frac{1}{3N^{3/4}}\overline{\ms}^3.$
From here, provided one can ignore the two error terms in \eqref{eq:non_linear2}, we can use \cite[Theorem 1.2]{Cha2011} to show that $T_N$ converges in distribution to the non-normal limit $W$, as desired. The main obstacle is the non trivial step of bounding the error terms in \eqref{eq:non_linear2}. To this effect, note that the error terms in the RHS of \eqref{eq:non_linear2} can be bounded as follows:
\begin{align*}
\begin{split}%\label{eq:non_linear3}
&\;\;\frac{\tanh''(\bar{\boldsymbol{m}}(\ms))}{2}\sum_{i=1}^N(m_i(\ms)-\bar{\boldsymbol{m}}(\ms))^2=O_P(\overline{\ms}\widetilde{S}_N),\\ & 
\sum_{i=1}^N\Big|m_i(\ms)-\bar{\boldsymbol{m}}(\ms)\Big|^3=O_P\Big( \max_{1\le i\le N}|m_i(\ms)-\bar{\boldsymbol{m}}(\ms)| \widetilde{S}_N\Big),
\end{split}
\end{align*}
where $\widetilde{S}_N:=\sum_{i=1}^N (m_i(\ms)-\bar{\boldsymbol{m}}(\ms))^2$. Thus in contrast to what happened before, it no longer suffices to bound only the quadratic term $\widetilde{S}_N$, but instead we also need to bound $\max_{i\in [N]}|m_i(\ms)-\bar{\boldsymbol{m}}(\ms)|$. The estimate for $\widetilde{S}_N$ is done by introducing an auxiliary variable to express $\ms$ as a mixture of i.i.d.~distributions, and then using Lemma \ref{lem:tailsubG}. The more challenging task is to bound $\max_{i\in [N]}|m_i(\ms)-\bar{\boldsymbol{m}}(\ms)|$, which we achieve by using a novel recursive argument, as follows:

Using the method of concentration via exchangeable pairs, we first show the approximate fixed point equation
$$m_i(\ms)-\bar{\boldsymbol{m}}(\ms)\stackrel{P}{\approx } \sum_{j=1}^NA_N(i,j) (m_j(\ms)-\bar{\boldsymbol{m}}(\ms)).$$
Writing $\tbm(\ms):=(m_1(\ms)-\bar{\boldsymbol{m}}(\ms),\ldots ,m_N(\ms)-\bar{\boldsymbol{m}}(\ms))$, and recursing the above fixed point equation we get $\tbm \approx A_N^k \tbm$, which gives
\begin{align}\label{eq:non_linear4}
|\tilde{m}_i(\ms)|\approx \big|\sum_{j=1}^N (A_N^k)(i,j)\tilde{m}_j(\ms)\big|\lesssim \sqrt{\sum_{j=1}^N \tilde{m}_j(\ms)^2}\sqrt{(A_N^{2k})(i,i)}.
\end{align}
Since $A_N$ is the scaled adjacency matrix of a connected regular graph, there can be at most two eigenvalues with absolute value $1$, and so for $k$ large enough the contribution of all other eigenvalues to $A_N^{2k}$ should be negligible. Also the corresponding normalized eigenvectors for these two eigenvalues must have all entries equal to $\frac{1}{\sqrt{N}}$ in absolute value. This suggests the approximate inequality 
\begin{align}\label{eq:matrix_nl}
(A_N^{2k})(i,i)\le \frac{2}{N}
\end{align}
for $k$ large enough. 
In~\cref{lem:tracebd}, we show the above bound for general regular matrices with non-negative entries satisfying the spectral gap condition \eqref{eq:well_connect}, but no condition on the minimum eigenvalue (i.e. no expander type condition). Of course such a result is not correct if \eqref{eq:well_connect} does not hold, as then $G_N$ can be disconnected. Plugging the bound \eqref{eq:matrix_nl} in \eqref{eq:non_linear4} along with the estimate $\widetilde{S}_N=\sum_{i=1}^N\tilde{m}_i(\ms)^2=O_P\Big(\frac{N}{d_N}\Big)$ gives
$$\max_{1\le i\le N}\tilde{m}_i(\ms)=O_P\Big(\sqrt{\frac{N}{d_N} \times \frac{2}{N}}\Big)\lesssim \frac{1}{\sqrt{d_N}},$$
Because of standard union bounds, we incur a log factor and deduce the estimate
$$
\max_{1\le i\le N}|m_i(\ms)-\bar{\boldsymbol{m}}(\ms)|=\max_{1\le i\le N}|\tilde{m}_i(\ms)|=O_P\Big(\sqrt{\frac{\log N}{d_N}}\Big).
$$
Plugging this bound back into \eqref{eq:non_linear2} shows that both the error terms are negligible, and hence gives an approximate Stein's equation for $W$ thereby completing the proof of Theorem \ref{theo:crit}.
\\

For verifying Theorem \ref{theo:uniq1} in the regime $\Theta_{11}$, we use a modified version of \eqref{eq:non_linear} with $T_N=\sqrt{N}\bar{\ms}$, where we expand $\tanh(\beta x)$ around $x=0$:
\begin{align*}%\label{eq:non_linear3}
\notag\E[T_N-T_N'|\ms]&=\frac{1}{N^{3/2}}\sum_{i=1}^N(\sigma_i-\tanh(\beta m_i(\ms))\\ &=\frac{1}{N^{3/2}}\left[\sum_{i=1}^N(\sigma_i-\beta m_i(\ms))+O_P(\sum_{i=1}^N|m_i(\ms)|^3)\right]\\
&=\frac{(1-\beta)T_N}{\sqrt{N}}+O_P\Big(\max_{1\le i\le N}|m_i(\ms)| \sum_{i=1}^Nm_i(\ms)^2\Big).
\end{align*}
As before, to complete the proof one needs to show that the error term above is negligible. The quadratic term $S_N=\sum_{i=1}^Nm_i(\ms)^2$ is controlled  using Lemma \ref{lem:qformuni}, and the max term is controlled by setting up another fixed equation (see Lemma \ref{lem:unifbd} part (a)).
\\

The above sketch works for exactly regular graphs. To handle approximately regular graphs/matrices, we need to bound the moments of ${\bf c}^{\mathrm T}\ms$ where $c_i=R_i-1$, in the regimes $\Theta_{11}$ and $\Theta_3$. This requires another recursive argument, and  is carried out in Lemma \ref{lem:unifbd} part (b) and Lemma \ref{lem:linfbdgen} part (b) for regimes $\Theta_{11}$ and $\Theta_3$ respectively. In fact, the proof of Lemma \ref{lem:unifbd} applies to general vectors ${\bf c}$, and the proof of Lemma \ref{lem:linfbdgen} can be modified to handle this case. %This is because 
% This requires us to further bound the second moment of ${\bf p}'\ms$, where ${\bf p}$ is the eigenvector of the eigenvalues close to $-1$, if present. Using the fact that ${\bf p}$ is a contrast vector, we 
%
%we show $\max_{1\leq i\leq N} (A_N^{2k})(i,i)\lesssim N^{-1}$ provided $k\sim\log{N}$, under assumption~\eqref{eq:well_connect} and no assumptions on the minimum eigenvalue of $A_N$. Next, $\sum_{j=1}^N \tilde{m}_j^2=\sum_{j=1}^N (m_j(\ms)-\bar{\boldsymbol{m}}(\ms))^2=O(N/d_N)$ from before, and we therefore get $\max_{1\leq i\leq N}|m_i(\ms)-\bar{\boldsymbol{m}}(\ms)|=O(d_N^{-1/2})$ and consequently $\sum_{j=1}^N (m_j(\ms)-\bar{\boldsymbol{m}}(\ms))^3=O(d_N^{-3/2}N)=o(N^{1/4})$ for $d_N\gg\sqrt{N}$. A similar argument has also been used in the proof of~\cref{theo:uniq1}.

\subsection{Examples}\label{sec:exmp}
As mentioned before, the most common example of a coupling matrix $A_N$ in model~\eqref{eq:model} is the scaled adjacency matrix $\frac{1}{\bar{d}_N}G_N$, where $G_N$ is the adjacency matrix of a simple labeled graph on $N$ vertices with degree vector $(d_1,\cdots,d_N)$, and $\bar{d}_N:=\frac{1}{N}\sum_{i=1}^Nd_i$ is the average degree of $G_N$.
%\begin{definition}\label{def:ScaledAdj}
%	For a simple graph $G_N$ with vertices labeled $[n]:= \{1,2,\ldots ,n\}$, define the coupling matrix $A_N$ by setting $A_N(i,j):= (1/\overline{d}_N)\mathbbm{1}\{\textrm{vertices i and j are connected}\}$ where $\overline{d}_N$ denotes the average degree of $G_N$, i.e., $\overline{d}_N:= N^{-1}\sum_{i=1}^N d_i$ where $d_i$ denotes the degree of the $i^{\textrm{th}}$ vertex. When dealing with random graphs, the scaling by $\overline{d}_N$ is often replaced by its expectation.
The scaling discussed in the above definition ensures that the resulting Ising model has non-trivial phase transition properties (see e.g.,~\cite{Basak2017,Mukherjee2018}). Below we consider some specific examples of graphs to illustrate our theorems. %To hide log factors, we will sometimes use the notation $x_N\lesssim_{\log} y_N$ to imply the existence of an exponent $\gamma$ free of $N$ such that $x_N\le (\log N)^\gamma y_N$.%and also the main theorems in this paper) which is of great interest in Statistical Physics and Applied Probability. 
%\end{definition}
\begin{enumerate}
	\item[(a)]{\bf Regular graphs:}
	Let $G_N$ be a $d_N$ regular graph. Then $\lVert A_N\rVert_F^2=\frac{N}{d_N}$ and $R_i=1$, and so applying Theorems \ref{theo:uniq2}, \ref{theo:nouniq}, \ref{theo:crit} and \ref{theo:uniq1} give
	\begin{small}
	\begin{eqnarray*}
		d_{KS}\Big(\sqrt{N}(\oms-t),Z_\tau\Big)\lesssim & \sqrt{\frac{N \log N}{d_N^{3}}}+\frac{1}{\sqrt{N}}&\text{ if }(\beta,B)\in \Theta_{11},\\
		d_{KS}\Big(\sqrt{N}(\oms-t),Z_\tau\Big)\lesssim&\frac{\sqrt{N}}{d_N}&\text{ if }(\beta,B)\in \Theta_{12},\\
		d_{KS}\Big(\sqrt{N}(\oms-M(\ms)),Z_\tau\Big)\lesssim & \frac{\sqrt{N}}{d_N}&\text{ if }(\beta,B)\in \Theta_2\text{ and }G_N \text{ satisfies \eqref{eq:well_connect},}\\
		d_{KS}\Big(N^{1/4}\oms,W\Big)\lesssim & \Big(\frac{\sqrt{N}\log N}{d_N}\Big)^{3/2}+\frac{\sqrt{N}}{d_N}+\frac{\log N}{\sqrt{N}}&\text{ if }(\beta,B)\in \Theta_3\text{ and }G_N \text{ satisfies \eqref{eq:well_connect},}
	\end{eqnarray*}
    \end{small}
	where $Z_{\tau}$ and $W$ are defined as in~\cref{lem:cw_known}. In particular this means that $\oms$ has the same fluctuations as that of the Curie Weiss model as soon as 
	\begin{eqnarray}\label{eq:d_reg}
	\begin{split}
	d_N\gg& (N\log N)^{1/3}&\text{ if }(\beta, B)\in \Theta_{11},\\
	d_N\gg&  \sqrt{N}&\text{ if }(\beta,B)\in \Theta_{12},\\
	d_N\gg & \sqrt{N}&\text{ if }(\beta,B)\in	  \Theta_2 \text{ and \eqref{eq:well_connect} holds},\\
	d_N\gg & \sqrt{N}\log N&\text{ if }(\beta,B)\in \Theta_3\text{ and \eqref{eq:well_connect} holds}.
	\end{split}
	\end{eqnarray}
	%$G_N$ is well connected, and $d_N\gg_{\log} N^{1/3}$ if $(\beta,B)\in \Theta_{11}$, $d_N\gg \sqrt{N}$ if $(\beta,B)\in \Theta_{12}\cup \Theta_2$, and $d_N\gg_{\log} \sqrt{N}$ if $(\beta,B)\in \Theta_3$. 
	Further, as already shown in Example \ref{eg:sparsity}, the requirement $d_N\gg \sqrt{N}$ is sharp in the regimes $\Theta_{12}\cup \Theta_2\cup \Theta_3$. Note that for the particular case of the Curie-Weiss model at criticality we get the convergence rate of $N^{-1/2}\log N$, which matches the rate obtained in \cite{Cha2011} up to the log factor. In fact, it is easy to modify our argument in the special case of the Curie-Weiss model to get rid of the log factor. We observe that for the case of random $d_N$ regular graphs, condition \eqref{eq:well_connect} holds with high probability, as $\lambda_2(G_N)=O_P(\sqrt{d_N})\ll d_N$ (see~\cite{broder1987second}), and so our results apply directly to random regular graphs if $d_N$ satisfies \eqref{eq:d_reg}. We stress that our results apply to regular bipartite graphs as well, and does not need the graph to be an expander as in \cite{Bresler2019}.

	\item[(b)]{\bf Erd\H{o}s-R\'enyi graphs: }%\label{sec:appranreg}
	%Our main theorems also imply \emph{annealed} central limit theorems for Ising models on ``approximately" regular random graphs. In this case, note that $\overline{d}_N$, from~\cref{def:ScaledAdj} is itself a random variable and is often replaced by its expectation under the graph measure. Our general results apply under both scalings, however we only present the more commonly studied scaling with expectation of $\overline{d}_N$ in the examples below.
	%Let $G_N$ be an Erd\H{o}s-R\'enyi random graph with parameters $(N,p_N)$
	%see e.g.,~\cite{Bovier1993,kabluchko2019fluctuations}]\label{eg:erg}
	Suppose $G_N\sim \mathcal{G}(N,p_N)$ is the symmetric Erd\H{o}s R\'enyi random graph with $0<p_N\leq 1$. Define $A_N(i,j):= \frac{1}{(N-1)p_N}G_N(i,j)$, and note that
	\begin{small}
	\begin{align}\label{eq:gg}
	\begin{split}
	\max_{1\le i\le N}|R_i-1|=O_P\Big(\sqrt{\frac{\log N}{Np_N}}\Big),\quad
	\Big|\sum_{i=1}^N(R_i-1)\Big|=O_P\Big(\frac{1}{\sqrt{p_N}}\Big),\quad \sum_{i=1}^N(R_i-1)^2=O_P\Big(\frac{1}{p_N}\Big).
	\end{split}
	\end{align}
	\end{small}%\mathbbm{1}\{\textrm{vertices i and j are connected}\}$. 
	Since $\lambda_2(G_N)=O_P(\sqrt{Np_N})\ll Np_N$ (\cite[Theorem 1.1]{Feige2005}), \eqref{eq:well_connect} holds as well. Then our theorems conclude universal fluctuations for $\oms$ as soon as 
	\begin{eqnarray}\label{eq:er}
	\begin{split}
	p_N\gg& (\log N)^{1/3}N^{-2/3}&\text{ if }(\beta, B)\in \Theta_{11},\\
	p_N\gg&  N^{-1/2}&\text{ if }(\beta,B)\in \Theta_{12}\cup \Theta_2,\\
	p_N\gg &{(\log N)^4 N^{-1/2}}&\text{ if }(\beta,B)\in \Theta_3,
	\end{split}
	\end{eqnarray}
	both in the quenched and annealed setting. We note that our results also apply to the asymmetric Erd\H{o}s-R\'enyi random graph $\mathcal{\tilde{G}}(N,p_N)$, under the same regime of $p_N$ as in the symmetric case. This is because an Ising model on the asymmetric Erd\H{o}s-R\'enyi  graph is equivalent to an Ising model with the symmetric coupling matrix $A_N(i,j)=\frac{\mathcal{\tilde{G}}_N(i,j)+\mathcal{\tilde{G}}_N(j,i)}{2(N-1)p_N}$, which is approximately regular, as
	$$ R_i=\sum_{j=1}^NA_N(i,j)=\frac{1}{2(N-1)p_N}\sum_{j=1}^N (\mathcal{\tilde{G}}_N(i,j)+\mathcal{\tilde{G}}_N(j,i))\sim \frac{\mbox{Bin}(2(N-1),p_N)}{2(N-1)p_N}\stackrel{P}{\approx} 1,$$
	where the last approximation (in the sense of~\eqref{eq:gg}) follows by a standard application of Chernoff's inequality. The asymmetric case was studied recently in~\cite{Kabluchko2019fluctuations}, where the authors derive fluctuations as soon as $Np_N\to\infty$, but only in the sub parameter regime $\Theta_{11}\cup \Theta_3$. The authors conjecture similar results for the symmetric case, which we are able to verify partially in this paper. Moreover, our theorems apply simultaneously to both the symmetric and the asymmetric cases with explicit convergence rates. A few months after our paper was submitted,~\cite{kabluchko2020low} was uploaded where the authors obtain fluctuations for the magnetization in the asymmetric case for the parameter regime $\Theta_{12}\cup\Theta_2$ when $N^{1/3}p_N\to\infty$, in~\cite[Theorems 1.1 and 1.3]{kabluchko2020low}. In contrast, our results show universal fluctuations in the larger regime $N^{1/2}p_N\to\infty$, and apply to both the symmetric and asymmetric cases, from which the fluctuation results for the magnetization in \cite{kabluchko2020low} follow as corollaries. On the other hand, in~\cite[Theorem 1.4]{kabluchko2020low}, the authors derive a central limit theorem for the log partition function when $N^{1/3}p_N\to\infty$, a direction which is not explored in our paper.
	
	\item[(c)]{\bf Balanced stochastic block model: }% see e.g.,~\cite{liu2017log}]\label{eg:sbm}
	Suppose $G_N$ is a stochastic block model with $2$ communities of size $N/2$ (assume $N$ is even). Let the probability of an edge within the community be $a_N$, and across communities be $b_N$. This is the well known stochastic block model, which has received considerable attention in Probability, Statistics and Machine Learning (see \cite{deshpande2018contextual,liu2017log,mossel2012stochastic} and references within). If we take $A_N=\frac{2}{N(a_N+b_N)}G_N$,   universal asymptotics hold for $\oms$ as soon as $p_N:=\frac{a_N+b_N}{2}$ satisfies \eqref{eq:er}, and $\liminf_{N\rightarrow\infty}\frac{b_N}{a_N}>0$ (needed to ensure \eqref{eq:well_connect}). Similar results hold when the number of communities is larger than $2$. %Since the 

	\item[(d)]{\bf Sparse regular graphons:} 
	Suppose that $W$ be a symmetric measurable function from $[0,1]^2$ to $[0,1]$, such that $\int_{[0,1]}W(x,y)dy=a>0$ for all $x\in [0,1]$, and $\lambda_2(W)<a$, where $\{\lambda_i(W)\}_{i\ge 1}$ are the countable set of ordered eigenvalues. Also let $(U_1,\cdots,U_N)\stackrel{i.i.d.}{\sim}U(0,1)$. For $\gamma\in (0,1]$, let \[\{G_N(i,j)\}_{1\le i<j\le N}\stackrel{i.i.d.}{\sim}Bern\bigg(\frac{W(U_i,U_j)}{N^\gamma}\bigg).\] Such random graph models have been studied in the literature under the name $W$ random graphons (c.f.~\cite{BorgsdenseI,BorgsdenseII,BorgsLPI,BorgsLPII,Lovasz2012}). In this case for the choice $A_N=\frac{1}{Np_N}G_N$ with  $p_N=a N^{-\gamma}$, universal fluctuation holds as soon as $\gamma<1/2$. Indeed, note that $\E [R_i|U_1,\ldots ,U_N]=(aN)^{-1}\sum_{j=1}^N W(U_i,U_j)$ and write \[R_i-1=\left[R_i-\frac{\sum_{j=1}^NW(U_i,U_j)}{aN}\right]+\left[\frac{\sum_{j=1}^NW(U_i,U_j)}{aN}-1\right].\] By using Bernstein's inequality conditional on $(U_1,\ldots ,U_N)$, the first term is $O_P((Np_N)^{-1/2})$. Similarly, by applying Bernstein's inequality conditional on $U_i$, the second term is $O_P(N^{-1/2})$. An application of the union bound then implies that \eqref{eq:gg} holds. %Consequently we have universal fluctuations for $\bar{\ms}$ as soon as \eqref{eq:er} holds. 
	Also with $W_N$ denoting the $N\times N$ matrix with $W_N(i,j)=W(U_i,U_j)$, using \cite[Corollary 3.3]{Afonso2016} we have 
	$ \lVert A_N-(aN)^{-1}W_N\rVert_{\text{op}}=O_P\Big(\frac{\sqrt{N}}{Np_N}\Big)$. Since $W_N$ converges in cut norm to $W$, it follows using \cite[Section 11.6]{Lovasz2012} that
	\[\lim_{N\rightarrow\infty}\lambda_2(A_N)=a^{-1}\lim_{N\rightarrow\infty}\frac{\lambda_2(W_N)}{N}=a^{-1}\lambda_2(W)<1\]
	and so $A_N$ satisfies \eqref{eq:well_connect}. By our results, universal fluctuations hold for $\oms$ as soon as \eqref{eq:d_reg} holds.

	\item[(e)]{\bf Block spin Ising model:} 
	Suppose that $N$ is even, and
	\begin{align*}
	A_N(i,j)=&a_N\text{ if }i,j\le N/2\text{ or }i,j>N/2,\\
	=&b_N\text{ if }i\le N/2, j>N/2,\text{ or }i>N/2, j\le N/2.
	\end{align*}
	$A_N$ can be thought of as the expectation of a stochastic block model with 2 communities. In the particular case $a_N=\frac{\beta}{N}, b_N=\frac{\alpha}{N}$, this model has been studied in \cite{Berthet2019,Lowe2018} under the name block spin Ising model. Again in this case universal asymptotics holds for $\oms$ as soon as $d_N:=\frac{N(a_N+b_N)}{2}$ satisfies \eqref{eq:d_reg}, and $\liminf_{N\rightarrow\infty}\frac{b_N}{a_N}>0$. This in particular matches the results obtained from~\cite[Theorems 1.2, 1.4]{Lowe2018} which studies the sub parameter regime $\Theta_{11}\cup \Theta_3$. Our results apply to the whole parameter regime of $(\beta,B)$ and a wide regime of scalings of $(a_N,b_N)$,  providing explicit convergence rates. Similar extension holds when the matrix $A_N$ has more than 2 groups as well.

	%	\begin{align*}
	%	\max_{1\le i\le N}|R_i-1|=O_P\Big(\sqrt{\frac{\log N}{Np_N}}\Big),\quad
	%	\Big|\sum_{i=1}^N(R_i-1)\Big|=O_P\Big(\frac{1}{\sqrt{p_N}}\Big),\quad \sum_{i=1}^N(R_i-1)^2=O_P\Big(\frac{1}{p_N}\Big).
	%	\end{align*}

	\item[(f)]{\bf Wigner matrices:}
	To demonstrate that our techniques apply to examples well beyond scaled adjacency matrices, let  $A_N$ be a Wigner matrix with its entries $\{A_N(i,j),1\le i<j\le N\}$ i.i.d. from a distribution $F$ scaled by $N\mu$, where $F$ is a distribution on non-negative reals with finite exponential moment and mean $\mu>0$. In this case we have
	\begin{align*}
	\max_{1\le i\le N}|R_i-1|=O_P\Big(\sqrt{\frac{\log N}{N}}\Big),\quad
	\Big|\sum_{i=1}^N(R_i-1)\Big|=O_P(1),\quad \sum_{i=1}^N(R_i-1)^2=O_P(1).
	\end{align*}
	Also~\cite[Corollary 3.5]{Afonso2016} shows that $\lVert A_N-\frac{1}{N}{\bf 1}{\bf 1}^\top\rVert_{\text{op}}=N^{-1/2}$, and so \eqref{eq:well_connect} holds. Thus our theorems apply giving universal fluctuations for $\oms$. 

\end{enumerate}

\section{Main technical lemmas}\label{sec:maaintechlem}
In this section, we state our main technical lemmas which could be of independent interest. Our first result in this section is an exponential moment control lemma in all parameter regimes,  which is one of the main estimates of this paper, and is itself new. The proof of this is %based on known Hanson-Wright inequality, and is 
deferred to Section \ref{sec:maintechlempf}.

\begin{lemma}\label{lem:qformuni}
	Suppose $\ms$ is an observation from~\eqref{eq:model}, with $A_N$ satisfying \eqref{eq:A1} and \eqref{eq:A2}. 
	\begin{enumerate}
		\item[(a)]
		If $(\beta,B)\in \Theta_1$, then there exists a fixed positive number $\delta>0$  such that
		\begin{align}\label{eq:errbound}
		\log\E\left[\exp\left(\frac{\delta}{2}\sum_{i=1}^N (m_i(\ms)-t)^2\right)\right]\lesssim\fa+t^2\sum_{i=1}^N (R_i-1)^2.	\end{align}
		%where $\eta_N:=$.%+t^2|\sum_{i=1}^N (R_i-1)|.
		\item[(b)]
		If $(\beta,B)\in  \Theta_2$, then the conclusion of part (a) holds under the additional assumption that $A_N$ satisfies \eqref{eq:well_connect}. % there exists a fixed positive number $\delta>0$ such that
		%	\begin{align}\label{eq:qformuni}
		%	\log\E\left[\exp\left(\frac{\delta}{2}\sum_{i=1}^N (m_i(\ms)-M(\ms))^2\right)\right]\lesssim \fa+\sum_{i=1}^N (R_i-1)^2+\Big|\sum_{i=1}^N(R_i-1)\Big|.
		%	\end{align}
		\item[(c)]
		If $(\beta,B)\in  \Theta_3$, then under the additional assumption that $A_N$ satisfies \eqref{eq:well_connect} there exists a fixed positive number $\delta>0$ such that 
		\begin{align}\label{eq:critqformuni}
		&\;\;\;\;\log\E\left[\exp\left(\frac{\delta}{2}\sum_{i=1}^N (m_i(\ms)-\mm)^2\right)\right]\nonumber \\& \lesssim  \fa+\frac{1}{N}\Big[\sum_{i=1}^N(R_i-1)^2\Big]^2+\frac{1}{N}\Big[\sum_{i=1}^N(R_i-1)\Big]^2+\log N,
		\end{align}
		where $\mm:= N^{-1}\sum_{i=1}^N m_i(\ms)$.
	\end{enumerate}
	
	%where $m_i(\ms)$ is as defined in~\cref{sec:prelim}.
\end{lemma}

Our next lemma establishes a uniform control on the $m_i(\ms)$'s and a second moment bound on a linear statistic of interest, when $(\beta,B)\in\Theta_{11}$. The proof of this lemma is deferred to~\cref{sec:prooflem2}.

\begin{lemma}\label{lem:unifbd}
	Assume that $\ms$ is an observation from \eqref{eq:model} with $(\beta,B)\in\Theta_{11}$, and $A_N$ satisfies \eqref{eq:A3}. Setting $\alpha_N=\max_{1\le i\le N}\sum_{j=1}^NA_N(i,j)^2$ as in~\cref{theo:uniq1}, the following conclusions hold:
	\begin{enumerate}
		\item[(a)]
		$\log \mmp\big(\max_{1\leq i\leq N} |m_i(\ms)|\geq \lambda\sqrt{\alpha_N\log N}\big)\lesssim - \lambda^2,$ for any $\lambda>0$.
		
		\item[(b)] $\mme\left[\sum_{i=1}^N (R_i-1)\sigma_i\right]^2\lesssim  \left(\sum_{i=1}^N (R_i-1)^2\right)\Big[1+\fa \alpha_N^2(\log N)^2 \Big].$	
	\end{enumerate}
\end{lemma}

Our final lemma yields uniform control on $\max_{1\leq i\leq N}|m_i(\ms)-\bar{\boldsymbol{m}}(\ms)|$ and moment bounds on linear statistics of interest, when $(\beta,B)\in\Theta_3$. Its proof has been deferred to~\cref{sec:prooflem2}.

\begin{lemma}\label{lem:linfbdgen}
	Suppose $\ms$ is an observation from~\eqref{eq:model} with $(\beta,B)\in \Theta_3$,  such that $A_N$ satisfies \eqref{eq:A4} and \eqref{eq:well_connect}. Suppose further that the RHS of \eqref{eq:critmain} is bounded. Then the following conclusions hold:
	\begin{enumerate}
		\item[(a)]
		$\log  \mmp\bigg(\max\limits_{1\leq i\leq N}|m_i(\ms)-\mm|\geq \lambda\sqrt{\alpha_N( \log N)^3}+\log N\max_{1\le i\le N}|R_i-1|\bigg)\lesssim -\lambda^2$, for any $\lambda>0$.
		
		\item[(b)]
		$\mme\left[\sum_{i=1}^N (R_i-1)\sigma_i\right]^2\lesssim \bigg(\sum_{i=1}^N (R_i-1)^2+N^{-1/2}\Big[\sum_{i=1}^N(R_i-1)\Big]^2\bigg)(\log N)^{4}.$
		
		\item[(c)]
		$N^{3/2}\E(\oms^6)\lesssim 1$.
	\end{enumerate}
	
\end{lemma}

\begin{remark}[On the minimum eigenvalue of $A_N$]\label{rem:connected}
	Note that our results work even when $\lambda_N(A_N)\to -1$, as opposed to stronger spectral gap assumptions such as $\max_{2\leq i\leq N} |\lambda_i(A_N)|\to 0$. This has been achieved by a new matrix theoretic estimate (see~\cref{lem:tracebd}) which shows that under~\eqref{eq:well_connect}, $A_N$ can have at most one eigenvalue ``close" to $-1$ (see~\cref{rem:mineigen} for connections to graph theory).
\end{remark}
\section{Proof of main results}\label{sec:mainthpfs}

We first state a lemma which will be needed in all parameter regimes. 

\begin{lemma}\label{lem:auxtail}
	Suppose $\ms$ is an observation from~\eqref{eq:model} for some $A_N$ satisfying \eqref{eq:A1}, and $\beta>0,B\in \R$. 
	\begin{enumerate}
		\item[(a)]
		Recalling that $m_i(\ms)=\sum_{j=1}^NA_N(i,j)\sigma_j$, we have
		\begin{align*}
		\E \Big[\sum_{i=1}^N(\sigma_i-\tanh(\beta m_i(\ms)+B))\tanh(\beta m_i(\ms)+B)\Big]^2\lesssim N.
		\end{align*}
		\item[(b)]
		For any ${\bf c}=(c_1,\cdots,c_n)\in \R^n$ we have
		$$ \log \P\Big(|\sum_{i=1}^Nc_i(\sigma_i-\tanh(\beta m_i(\ms)+B))|>t\Big)\lesssim -\frac{t^2}{||{\bf c}||_2^2}.$$
	\end{enumerate}
\end{lemma}
Here, part (a) follows by invoking~\cite[Lemma 3.2]{ChaDembo2016} and (b) can be obtained by making minor adjustments in the proof of~\cite[Lemma 1]{mukherjee2019testing}.

\subsection{Proof of Theorems~\ref{theo:uniq2}~and~\ref{theo:nouniq}}\label{sec:pfmaintwo}
In this section, we will prove Theorems~\ref{theo:uniq2}~and~\ref{theo:nouniq} using~\cref{thm:conc}, Lemmas~\ref{lem:qformuni},~\ref{lem:auxtail}~and~\cref{prop:CWresmain}. The statement of~\cref{prop:CWresmain} is deferred to~\cref{sec:addresCW} as its scope is limited to the Curie-Weiss model introduced in~\eqref{eq:cw_new}. 

Without loss of generality we may assume that the RHS of \eqref{eq:uniq2main} and \eqref{eq:non_uniq} are bounded by $1$, because otherwise the bound is trivial. Recall the definition of $M(\ms)$ for $(\beta,B)\in\Theta_2$ from~\cref{lem:cw_known} and set $M(\ms)=t$ for $(\beta,B)\in\Theta_1$. We have not made the dependence of $M(\ms)$ on $(\beta,B)$ explicit for notational simplicity. From~\cref{sec:pfideas}, recall that $T_N=\sqrt{N}(\overline{\boldsymbol{\sigma}}-M(\ms))$ and $T_N'=\sqrt{N}(\overline{\boldsymbol{\sigma'}}-M(\ms'))$ where $\ms$ is an observation from the Ising model~\eqref{eq:model}, and $\ms'$ is generated as follows: Let $I$ denote a randomly sampled index from $\{1,2,\ldots ,N\}$. Given $I=i$, replace $\sigma_i$ with an independent $\pm 1$ valued random variable $\sigma'_i$ with mean $\tanh(\beta m_i(\ms)+B)=\E[\sigma_i|(\sigma_j,j\neq i)]$, and let $\ms':=(\sigma_1,\cdots,\sigma_{i-1},\sigma'_i,\sigma_{i+1},\cdots,\sigma_N)$.
%Setting $T_N:=\sqrt{N}(\overline{\boldsymbol{\sigma}}-t)$, note that %Recall that $M(\ms)=\ut$ if $\overline{\ms}\geq 0$ and $M(\ms)=-\ut$ if $\overline{\ms}< 0$. In other words, we are taking the first step of the Glauber dynamics. Set $T_N':= \sqrt{N}(\overline{\ms}'-M(\ms'))$. Note that,

With this setup, a direct computation gives
\begin{align}\label{eq:nonuniquetheo3}
\E[T_N-T_N'|T_N]=\frac{1}{N^{3/2}}\sum_{i=1}^N\E[\sigma_i-\tanh(\beta m_i(\ms)+B)|T_N]-\sqrt{N}\mathbb{E}[M(\boldsymbol{\sigma})-M(\boldsymbol{\sigma'})|T_N],
\end{align}
where the second term in the RHS above can be expanded as
\begin{align}\label{eq:nonuniquetheo4}
&\;\;\;\;\notag\sum_{i=1}^N \tanh(\beta m_i(\ms)+B)\nonumber \\ &=N\tanh(\beta M(\boldsymbol{\sigma})+B)+\beta(1-t^2)\sum_{i=1}^N(m_i(\ms)-M(\ms))+\sum_{i=1}^N\xi_i(m_i-M(\ms))^2\nonumber \\ & = NM(\ms)+\beta(1-t^2)\sum_{i=1}^N (\sigma_i-M(\boldsymbol{\sigma}))+\beta(1-t^2)M(\ms)\sum_{i=1}^N (R_i-1)\nonumber \\ &\qquad +\beta(1-t^2)\sum_{i=1}^N (R_i-1)(\sigma_i-M(\boldsymbol{\sigma}))+\sum_{i=1}^N \xi_i(m_i(\ms)-M(\boldsymbol{\sigma}))^2
\end{align}
for random variables $(\xi_i)_{1\le i\le N}$ satisfying $\max_{1\leq i\leq N}|\xi_i|\lesssim 1$, where the second line uses the identity $M(\ms)=\tanh(\beta M(\ms)+B)$. Setting $h_i=\beta(1-t^2)(R_i-1)$ and plugging~\eqref{eq:nonuniquetheo4} into~\eqref{eq:nonuniquetheo3} %, using the identity $\tanh(\beta M(\ms)+B)=M(\ms)$ 
we get
\begin{align}\label{eq:nonuniquetheo5}
\mathbb{E}[T_N-T_N'|T_N]&=\underbrace{\frac{T_N}{N}(1-\beta(1-t^2))}_{g(T_N)}-\underbrace{\frac{1}{N\sqrt{N}}\sum_{i=1}^N \mathbb{E}[\xi_i(m_i(\ms)-M(\boldsymbol{\sigma}))^2|T_N]}_{H_{1}(T_N)}\nonumber \\&-\underbrace{\frac{1}{N\sqrt{N}}\mathbb{E}\left[\sum_{i=1}^N h_i(\sigma_i-M(\boldsymbol{\sigma}))|T_N\right]}_{H_{2}(T_N)}\nonumber \\&-\underbrace{\sqrt{N}\mathbb{E}[M(\boldsymbol{\sigma})-M(\boldsymbol{\sigma}')|T_N]-N^{-3/2}\beta(1-t^2)M(\ms)\sum_{i=1}^N (R_i-1)}_{H_{3}(T_N)}.
\end{align}
Next, we observe that
\begin{small}
	\begin{align}\label{eq:nonuniquetheo7}
	\limsup_{N\rightarrow\infty} \frac{1}{N}\log{\P\left(|T_N-T_N'|\geq \frac{3}{\sqrt{N}} \right)}&\leq \limsup_{N\rightarrow\infty} \frac{1}{N}\log{ \P\left(\sqrt{N}|M(\ms)-M(\ms')|\geq \frac{1}{\sqrt{N}}\right)}\nonumber \\ &\leq \limsup_{N\rightarrow\infty}\frac{1}{N}\log \P(M(\ms)\ne M(\ms'))<0, 
	\end{align}
\end{small}
where the last inequality for $(\beta,B)\in \Theta_2$ follows on using part (b) of Theorem \ref{thm:conc} with $\mE_N:=\{\sum_{i=1}^N\sigma_i\in \{-2,-1,0,1,2\}\}$ along with part (c) of~\cref{prop:CWresmain} to note that
\begin{align}\label{eq:CWresmain}
\limsup_{N\rightarrow\infty}\frac{1}{N}\log \P(M(\ms)\ne M(\ms'))\le \limsup_{N\rightarrow\infty}\frac{1}{N}\log \P^{CW}(\mE_N)<0.
\end{align}
From~\eqref{eq:nonuniquetheo5}, we choose $g(x)=x(1-\beta(1-t^2))/N$. With this choice, observe that $G(x):= \int_0^x g(y)\,dy=(1-\beta(1-t^2))x^2/2N$. We now set  $c_0:= N/(1-t^2)$ and $c_1:= (2\pi \tau)^{-1/2}$, and note the existence of positive constants $c_2$ and $c_3$ free of $N$ such that assumptions (H1) and (H3) from \cite[Page 465]{Cha2011} are all satisfied.  By a slight variant of~\cite[Theorem 1.2]{Cha2011} (see~\cref{sec:revberry})~and~\eqref{eq:nonuniquetheo7}, we then have
\begin{align}\label{eq:nonuniquetheo6}
\;\;\;\ d_{KS}(T_N,Z_\tau)\lesssim  \mathbb{E}\Bigg|1-&\frac{N}{2(1-t^2)}\mathbb{E}\left[(T_N-T_N')^2|T_N\right]\Bigg|+\frac{c_1\max{(c_3,1)}}{\sqrt{N}}+\frac{\mme|T_N|+1}{\sqrt{N}}\nonumber \\ &+\frac{N}{c_1}\mathbb{E}\bigg[\sum_{a=1}^3|H_{a}(T_N)|\bigg]+\exp(-c_2N).
\end{align}
As we will see later in the proof, the $\exp(-c_2N)$ term is of a smaller order than the other terms in the RHS above. However, we choose to present it in this form so as to emphasize that~\eqref{eq:nonuniquetheo6} does not follow from a direct application of~\cite[Theorem 1.2]{Cha2011}, for $(\beta,B)\in\Theta_2$.

We will now estimate each term in the RHS of~\eqref{eq:nonuniquetheo6}.	Proceeding to control $\E| H_1(T_N)|$ we have
\begin{equation}\label{eq:moment_march}
N\sqrt{N}|H_{1}(T_N)|= \Bigg|\sum_{i=1}^N\mathbb{E}\left(\xi_i(m_i(\ms)-M(\boldsymbol{\sigma}))^2\right)|T_N\Bigg|\lesssim \mathbb{E}\Big(\sum_{i=1}^N (m_i(\ms)-M(\boldsymbol{\sigma}))^2\Big|T_N\Big),
\end{equation}
and so
\begin{align}\label{eq:nonuniquetheo9}
N\sqrt{N}\E|H_{1}(T_N)|\lesssim \E\sum_{i=1}^N\Big(m_i(\ms)-M(\ms)\Big)^2\le \eta_N
\end{align}
using Lemma \ref{lem:qformuni}, with $\eta_N:=\fa+t^2\sum_{i=1}^N (R_i-1)^2$. %denoting the RHS of \eqref{eq:errbound}. %~accordingly as we are in the non-uniqueness or the uniqueness regime. 
Next, we have
\begin{align}\label{eq:moment_march2}
\notag N\sqrt{N}\notag |H_{2}(T_N)|\le& \Bigg|\E\left(\sum_{i=1}^Nh_i\Big(\sigma_i-\tan(\beta m_i(\ms)+B)\Bigg|T_N\right)\Bigg|\\
+&\Bigg|\E\left(\sum_{i=1}^Nh_i\Big(\tanh(\beta m_i(\ms)+B)-\tan(\beta M(\ms)+B)\Bigg)\Bigg|T_N\right)\Bigg|
\end{align}
and so
\begin{align}\label{eq:nonuniquetheo10}
\notag N\sqrt{N}\E |H_{2}(T_N)|\notag	\lesssim & \sqrt{\sum_{i=1}^n h_i^2}+\sqrt{\sum_{i=1}^Nh_i^2}\sqrt{\E \sum_{i=1}^N\Big(m_i(\ms)-M(\ms)\Big)^2}\\
\lesssim & \sqrt{\sum_{i=1}^N(R_i-1)^2}(1+\sqrt{\eta_N})\lesssim \eta_N+\sum_{i=1}^N(R_i-1)^2,
\end{align}
where the penultimate line uses part (b) of Lemma \ref{lem:auxtail}, and the last line again uses \eqref{eq:nonuniquetheo9}.
%	\begin{align}\label{eq:nonuniquetheo11}
%	\notag\E H_{N,3}&:= \mathbb{E}\Bigg|\sum_{i=1}^N\mathbb{E}\left(h_i(\tanh{(\beta M(\boldsymbol{\sigma)+B)}}-\tanh(\beta m_i(\ms)+B))\right)|T_N\Bigg|\\ &\lesssim \sqrt{\sum_{i=1}^n h_i^2}\sqrt{\sum_{i=1}^N \mathbb{E}\left(m_i(\ms)-M(\boldsymbol{\sigma}))^2\right)}\lesssim \sqrt{\eta_N \sum_{i=1}^N(R_i-1)^2}\lesssim \eta_N+\sum_{i=1}^N(R_i-1)^2
%	\end{align}
%	where we have again used  Lemma \ref{lem:qformuni}. 
Also observe that,
\begin{equation}\label{eq:nonuniquetheo12}
N\sqrt{N}| H_{3}(T_N)|\lesssim N^2\mathbb{E}\Big(|M(\boldsymbol{\sigma})-M(\boldsymbol{\sigma}')|T_N\Big)+t\bigg|\sum_{i=1}^N (R_i-1)\bigg|,%\lesssim o(1)+\bigg|\sum_{i=1}^N (R_i-1)\bigg|
\end{equation}
where the first term has an expectation which is exponentially small in $N$ using  \eqref{eq:CWresmain}. Finally we have
\begin{align*}%\label{eq:nonuniquetheo13}
\Bigg|\mathbb{E}\left[1-\frac{N}{2(1-t^2)}(T_N-T_N')^2\bigg|T_N\right]\Bigg| \lesssim \mathbb{E}\Bigg|\mathbb{E}\Big[1-\frac{(\sigma_I-\sigma_I')^2}{2(1-t^2)}\bigg|T_N\Big]\Bigg|+N^2\mathbb{E}[|M(\boldsymbol{\sigma})-M(\boldsymbol{\sigma}')|].
\end{align*}
The second term on the RHS above  is exponentially small, by~\eqref{eq:CWresmain}. For the first term on the RHS, note that:
%$$\E\Big[\frac{(\sigma_I-\sigma_I')^2}{2(1-t^2)}\bigg|T_N\Big]=\frac{1-\sigma_i\tanh(\beta m_i(\ms))+B)}{1-t^2},$$
%$\mme_{\beta,B}^N|\mathbb{E}[1-(\sigma_I-\sigma_I')^2/2(1-t^2)|\boldsymbol{\sigma}]|$ . In this direction, observe that,
\begin{align*}
\mathbb{E}[1-(\sigma_I-\sigma_I')^2/2(1-t^2)|\boldsymbol{\sigma}]&=\frac{1}{N(1-t^2)}\sum_{i=1}^N (\E[\sigma_i\sigma_i'|\ms]-t^2)\nonumber \\&\lesssim N^{-1}\big|\sum_{i=1}^N (\sigma_i\tanh(\beta m_i(\ms)+B)-t^2)\big|.
\end{align*}
As a result we have
\begin{align}\label{eq:nonuniquetheo14}
&\;\;\;\;\mme|\mathbb{E}[1-(\sigma_I-\sigma_I')^2/2(1-t^2)| T_N]|\nonumber \\
&\lesssim\mme\Bigg|\sum_{i=1}^N (\sigma_i-\tanh(\beta m_i(\ms)+B))\tanh(\beta m_i(\ms)+B)\Bigg|+ \frac{1}{\sqrt{N}}\sqrt{\sum_{i=1}^N \mme(m_i(\ms)-M(\ms))^2}\nonumber \\
&\lesssim\frac{1}{\sqrt{N}}+\frac{1}{\sqrt{N}}\sqrt{\sum_{i=1}^N \mme (m_i(\ms)-M(\ms))^2}\le \frac{1+\sqrt{\eta_N}}{\sqrt{N}},
\end{align}
where we have used \eqref{eq:nonuniquetheo9}, and part (a) of Lemma \ref{lem:auxtail}.
We now claim that
\begin{align}\label{lem:secmomoth_2}
\E T_N^2\lesssim 1.
\end{align}
Given this claim, combining the estimates from~\eqref{eq:nonuniquetheo6},~\eqref{eq:nonuniquetheo7},%~\eqref{eq:nonuniquetheo8},
~\eqref{eq:nonuniquetheo9},~\eqref{eq:nonuniquetheo10},%\eqref{eq:nonuniquetheo11},
~\eqref{eq:nonuniquetheo12},%~\eqref{eq:nonuniquetheo13} 
and~\eqref{eq:nonuniquetheo14} we get
$$d_{KS}(T_N,Z_\tau)\lesssim \frac{1}{\sqrt{N}}+\frac{\eta_N}{\sqrt{N}}+\frac{1}{\sqrt{N}}\sum_{i=1}^N(R_i-1)^2+\frac{t}{\sqrt{N}}\Big|\sum_{i=1}^N(R_i-1)\Big|,$$
where $Z_{\tau}$ is defined as in~\cref{lem:cw_known}. The desired bound follows on noting that $\eta_N\gtrsim \fa\gtrsim1$.
\\

It thus suffices to prove \eqref{lem:secmomoth_2}. To this effect, using~\eqref{eq:nonuniquetheo5} we get:
$$\Bigg|\mme[T_N-T_N'|T_N]-\frac{T_N}{N}(1-\beta(1-t^2))\Bigg| \lesssim \sum_{a=1}^3|H_a(T_N)|. $$
By multiplying both sides of the above display by $N T_N$ and taking expectation gives
$$\mme[T_N^2]\lesssim\bigg|\mme[N(T_N-T_N')T_N]\bigg|+ \mme\left[N|T_N|\bigg(\sum_{a=1}^3|H_a(T_N)|\bigg)\right],$$
where we have used the fact that $\beta(1-t^2) < 1$. This follows from~\cref{lem:fixsol}, parts (b) and (c), on noting that $\phi'(t)=1-\beta(1-t^2)$ where $\phi(\cdot)$ is defined as in~\cref{lem:fixsol}. By the exchangeability of $T_N$ and $T_N'$ we have
\begin{align*}
\mme [N(T_N-T_N')T_N]=\mme[N(T_N'-T_N)T_N']=\frac{1}{2}\mme[N(T_N-T_N')^2]\lesssim 1.
\end{align*}
Also, from~\eqref{eq:moment_march},~\eqref{eq:moment_march2}~and~\eqref{eq:nonuniquetheo12}~we have \begin{align*}
N\sum_{a=1}^3\mme [H_a(T_N)]^2\lesssim \frac{\eta_N+\sum_{i=1}^N(R_i-1)^2}{\sqrt{N}}\lesssim 1,
\end{align*}
where the last bound uses the fact that the RHS of \eqref{eq:uniq2main} and \eqref{eq:non_uniq} are bounded.
Using Chebyshev's inequality then gives
\begin{align*}
\mme(T_N^2)\lesssim 1+\sqrt{\E(T_N^2)} \sqrt{\sum_{a=1}\E(N H_a(T_N))^2}\lesssim 1+\sqrt{\E(T_N^2)}
\end{align*}
which implies $\mme(T_N^2)\lesssim1$.
This verifies \eqref{lem:secmomoth_2}, and hence completes the proof of the theorem.

\subsection{Proof of Theorem \ref{theo:uniq1}}
\noindent We will now prove Theorem \ref{theo:uniq1} using Lemmas~\ref{lem:qformuni}~and~\ref{lem:unifbd} whose proofs have been deferred to~\cref{sec:prooflem2}. %The first lemma gives a control on $\max_{1\le i\le N}|m_i(\sigma)|$.  

%\begin{remark}
%We note that in part (a) of the above Lemma, for exactly $d_N$ regular graphs we can actually get the term $d_N^{-1/2}$ in place of $N^{-1/6}$, which is optimal in this case. For part (b) of the above lemma, the extra term $(\log N)^{3/2}$ seems to be a proof artifact.
%\end{remark}

\begin{proof}
	Without loss of generality we can assume that the RHS of \eqref{eq:uniq1main} is bounded as before. As in the proof of the previous theorems, it suffices to bound the RHS of \eqref{eq:nonuniquetheo6}, but with $t=M(\ms)=0$ which implies $H_3(T_N)=0$. To begin, use \eqref{eq:nonuniquetheo14} to get
	\begin{align}\label{eq:nonuniquetheo15}
	\E \Bigg|\mathbb{E}\left[1-\frac{N}{2}(T_N-T_N')^2\big|T_N\right]\Bigg|\lesssim \frac{\fa}{N}+\frac{1}{\sqrt{N}},
	\end{align}
	using \eqref{eq:errbound}, which allows us to replace $\eta_N$ in the previous proof by $\fa$. Proceeding to bound $\E|H_1(T_N)|$, use the first equality of \eqref{eq:moment_march} along with Cauchy-Schwarz inequality to note that
	%Recall the $C_i$'s from~\eqref{eq:nonuniquetheo5}. The crucial difference in this setting is to use the fact that $\mme_{\beta}^N[\max_{1\leq i\leq N} |C_i|^2]\lesssim \mme_{\beta}^N[\max_{1\leq i\leq N} |m_i(\ms)|^2]\lesssim N^{-1/3}\log{N}$ (which is a consequence of~\cref{lem:unifbd}). The bound from~\eqref{eq:nonuniquetheo7} holds as it is. Also, in this case, $\ut=0$ which implies $r_2(T_N)=0$. By~\cref{lem:secmomoth}, $\mme_{\beta}^N[T_N^2]=O(1)$ and by the same calculation as in~\eqref{eq:nonuniquetheo13}, we have:
	%	\begin{align}\label{eq:nonuniquetheo15}
	%	&\;\;\;\;
	%	\mathbb{E}\Bigg|\mathbb{E}\left[\mathbb{E}\left[1-\frac{N}{2}(T_N-T_N')^2|\tilde{\sigma}\right]\big|T_N\right]\Bigg| \lesssim \frac{1}{N}\mathbb{E}_{\beta}^N\Bigg|\sum_{i=1}^N \sigma_i\tanh(\beta m_i(\ms))\Bigg|\nonumber \\& \lesssim\frac{1}{N}\left(\sum_{i=1}^N \mme_{\beta}^N[m_i(\ms)^2]+\mme_{\beta}^N\Bigg|\sum_{i=1}^N (\sigma_i-\tanh(\beta m_i(\ms)))\tanh(\beta m_i(\ms))\Bigg|\right)\lesssim N^{-1}\fa+N^{-1/2}
	%	\end{align}
	%where the last line follows from Lemmas~\ref{lem:auxtail}~and~\ref{lem:qformuni}. 
	%It only remains to control $r_1(T_N)$. Towards this direction, observe that:
	\begin{align}\label{eq:nouniquetheo16}
	N\sqrt{N}\mme|H_1(T_N)|&\lesssim \E \max_{1\le i\le N}|m_i(\ms)| \sum_{i=1}^Nm_i(\ms)^2\nonumber \\ &\le \sqrt{\E\max_{1\le i\le N}m_i(\ms)^2}\sqrt{\E\Big(\sum_{i=1}^Nm_i(\ms)^2\Big)^2}\lesssim \fa\sqrt{\alpha_N\log N},
	\end{align}
	where the last inequality uses part (a) of Lemma \ref{lem:unifbd}.	Finally, for $\E|H_2(T_N)|$ we have
	\begin{align}
	N\sqrt{N}\E|H_2(T_N)|\le  \mme\Bigg|\sum_{i=1}^N (R_i-1)\sigma_i\Bigg|\nonumber  \lesssim  \sqrt{(\sum_{i=1}^N (R_i-1)^2)}\Big[1+\lVert A_N\rVert \alpha_N\log N \Big],
	\end{align}
	where we use part (b) of Lemma~\ref{lem:unifbd}. Plugging in the above bounds in~\eqref{eq:nonuniquetheo6}, we have
	\begin{align*}
	d_{KS}(T_N,Z_\tau)\lesssim \frac{1+\E(T_N^2)}{\sqrt{N}}+\frac{\fa\sqrt{\alpha_N\log N}}{\sqrt{N}}+\Big[1+\lVert A_N\rVert \alpha_N\log N \Big]\sqrt{\frac{\sum_{i=1}^N (R_i-1)^2}{N}},
	\end{align*}
	with $Z_{\tau}$ defined as in~\cref{lem:cw_known}. The claimed bound follows immediately, if we can verify \eqref{lem:secmomoth_2}. But the proof of this is the same as in the previous theorem, and so we are done.
\end{proof}

\subsection{Proof of Theorem \ref{theo:crit}}
In this section, we will use Lemmas~\ref{lem:qformuni}~and~\ref{lem:linfbdgen} to prove~\cref{theo:crit}. The proofs of the aforementioned lemmas are presented in~\cref{sec:prooflem2}.

\begin{proof}
	With $(\ms,\ms')$ the usual exchangeable pair, setting $T_N:= N^{1/4}\overline{\ms}$ and $T_N:= N^{1/4}\overline{\ms}'$ we have
	\begin{align*}
	\mme[T_N-T_N'|\ms]&=N^{-3/4}(\overline{\ms}-\tanh(\overline{\ms}))+N^{-3/4}(\tanh(\overline{\ms})-\tanh(\mm))\\ &\qquad +N^{-7/4}\sum_{i=1}^N (\tanh(m_i(\ms))-\tanh(\mm)).
	\end{align*} 
	Using Taylor's expansion, this gives
	%By a standard Taylor series expansion of $\tanh(x)$, the above equality immediately implies that there exists universal constants $C_1$, $C_2$ and $C_3$ such that:
	\begin{align}\label{eq:reflat}
	&\;\;\;|\mme[T_N-T_N'|\ms]-N^{-3/4}(\overline{\ms}-\tanh(\overline{\ms}))|\nonumber \\&\lesssim N^{-3/4}|\overline{\ms}-\mm|+N^{-7/4}|\overline{\ms}|\sum\limits_{i=1}^N (m_i(\ms)-\mm)^2+N^{-7/4}\Bigg|\sum\limits_{i=1}^N (m_i(\ms)-\mm)^3\Bigg|,
	\end{align} 
	and so we have
	%We will borrow the notation from~\cref{lem:mombound} (please refer back to that proof for details). Let us start with~\eqref{eq:mombdlateq}. Note that it implies 
	$\mme[T_N-T_N'|T_N]=g(T_N)+H(T_N)$ where $g(x)=N^{-3/2}x^3/3$, and $H(T_N)$ satisfies  
	\begin{align*}
	\mme[|H(T_N)|]&\lesssim N^{-2}\mme\big[|T_N|^5\big]+N^{-3/4}\mme\big[|\overline{\ms}-\mm|\big]+N^{-2}\mme\left[|T_N|\sum_{i=1}^N (m_i(\ms)-\mm)^2\right]\\ &+N^{-7/4}\mme\left[\bigg|\sum_{i=1}^N (m_i(\ms)-\mm)^3\bigg|\right].
	\end{align*}
	Invoking~\cite[Theorem 1.2]{Cha2011} with $G(x):=\int_0^x g(t)\,dt=N^{-3/2}x^4/12$ we have
	\begin{align}\label{eq:theocrit}
	&\;\;\;\ d_{KS}(T_N,W)\lesssim  \mathbb{E}\Bigg|1-\frac{N^{3/2}}{2}\mathbb{E}\left[(T_N-T_N')^2|T_N\right]\Bigg|+N^{3/2}\mme[|H(T_N)|]+N^{-3/4}\mme|T_N|^3.
	\end{align} By part (c) of~\cref{lem:linfbdgen} we have $\mme[|T_N|^5]\lesssim 1$. Set 
	$$\delta_N:= \sum_{i=1}^N (R_i-1)^2+N^{-1/2}\Big[\sum_{i=1}^N(R_i-1)\Big]^2,$$%\quad \varepsilon_N:= \fa+\frac{1}{N}\sum_{i=1}^N(R_i-1)^2+\frac{1}{N}\Big[\sum_{i=1}^N(R_i-1)\Big]^2+\log{N}.\]
	and use part (b) of~\cref{lem:linfbdgen} and the Cauchy-Schwarz inequality to get
	\begin{align*}
	\mme\big[|\overline{\ms}-\mm|\big]\lesssim \sqrt{\mme(\overline{\ms}-\mm)^2}\lesssim N^{-1}(\log N)^{2}\sqrt{\delta_N}.
	\end{align*} 
	Similarly, by the Cauchy-Schwarz inequality and part (c) of Lemmas~\ref{lem:qformuni} along with part (a) of Lemma~\ref{lem:linfbdgen}, we get:
	\begin{small}
	\begin{align*}
	&\mme\left[|T_N|\sum_{i=1}^N(m_i(\ms)-\mm)^2\right]\le \sqrt{\E(T_N)^2}\sqrt{\E\Big(\sum_{i=1}^n(m_i(\ms)-\mm)^2\Big)^2}\lesssim \varepsilon_N,\\
	& \mme\left[\sum_{i=1}^N |m_i(\ms)-\mm|^3\right]\le \sqrt{\E\max_{1\le i\le N}(m_i(\ms)-\bar{\boldsymbol{m}}(\ms))^2}\sqrt{ \E \Big(\sum_{i=1}^N(m_i(\ms)-\bar{\boldsymbol{m}}(\ms))^2\Big)^2}\lesssim r_N\varepsilon_N,
	\end{align*}
	\end{small}
	where $\varepsilon_N$ is as in the statement of Theorem \ref{theo:crit}.
	Combining the above observations, we get
	\begin{align}\label{eq:combine}
	N^{3/2}\mme[|H(T_N)|]\lesssim N^{-1/2}+N^{-1/4}(\log N)^2\sqrt{\delta_N}+N^{-1/4}r_N\varepsilon_N.
	\end{align} 
	%\emph{Note that the second term on the right hand side of the above display would not arise if $R_i=1$ for all $1\leq i\leq N$}. Next, by another application of~\cref{lem:secmomoth}, $(N^{3/2}/N^{9/4})\mme|T_N|^3\lesssim N^{-3/4}$. We also have by the tower property,
	Finally, we have
	\begin{align}\label{eq:crit1}
	&\;\;\;\;\notag\mme\bigg|1-\frac{N^{3/2}}{2}\mme\left[(T_N-T_N')^2|T_N\right]\bigg|\nonumber \\ &\lesssim \frac{1}{N}\mme\bigg|\sum_{i=1}^N \sigma_i\tanh m_i(\ms)\bigg|\nonumber \\ 
	\notag&\lesssim \frac{1}{N}\mme\bigg|\sum_{i=1}^N (\sigma_i-\tanh m_i(\ms))\tanh m_i(\ms)\bigg|+\frac{1}{N}\E\sum_{i=1}^N\Big(m_i(\ms)-\bar{\boldsymbol{m}}(\ms)\Big)^2+\E\oms^2\nonumber \\ &\lesssim \frac{1}{\sqrt{N}}+\frac{\varepsilon_N}{N}+\frac{1}{\sqrt{N}}
	\end{align}
	where the last inequality follows from part (a) of~\cref{lem:auxtail}, part (c) of~\cref{lem:qformuni},  and part  (c) of~\cref{lem:linfbdgen}. Combing \eqref{eq:combine} and \eqref{eq:crit1} along with \eqref{eq:theocrit} gives 
	$$ d_{KS}(T_N,W)\lesssim \frac{1}{\sqrt{N}}+\frac{\sqrt{\delta_N}(\log{N})^2}{N^{1/4}}+\frac{r_N\varepsilon_N}{N^{1/4}},$$%\lesssim \frac{\sqrt{\delta_N}(\log{N})^2}{N^{1/4}}+\frac{\log N}{\sqrt{N}}\lVert A_N\rVert_F^2,$$
	as desired, with $W$ defined as in~\cref{lem:cw_known}. %where the last inequality uses the fact that $\fa\gtrsim 1$, and $N^{-1/2}\delta_N\lesssim \sqrt{N^{-1/2}\delta_N}$. The claimed bound in the theorem is immediate from this.
	%we have that the quantity on the right hand side of~\eqref{eq:crit1} is bounded by $N^{-1/2}$. Finally plugging all the obtained bounds in~\eqref{eq:theocrit} completes the proof. 
\end{proof}

\section{Proofs of Theorems~\ref{thm:part},~\ref{thm:conc}~and Lemma~\ref{lem:qformuni}}\label{sec:maintechlempf}
We will  need the following proposition which expresses the Curie-Weiss model as a mixture of i.i.d. random variables, first shown in~\cite[Lemma 3]{Mukherjee2018}.
\begin{prop}\label{prop:CWbasic}
	Let $\ms$ be an observation from the Curie-Weiss model in  \eqref{eq:cw_new}.
	% whose p.m.f. can be written as $$\P^{CW}(\ms)=\frac{1}{Z_N^{CW}(\beta,B)}\exp\Big\{N\beta \bar{\sigma}_N^2+B\sum_{i=1}^N\sigma_i\Big\}.$$% with normalizing constant~$Z_N^{CW}(\beta,B)$, 
	Given $\ms$, let $W_N$ be a Gaussian random variable with mean $\bar{\ms}$ and variance $(N\beta)^{-1}$. Then the following conclusions hold:
	\begin{itemize}
		\item[(a)] Given $W_N$, the random variables $(\sigma_1,\sigma_2,\ldots ,\sigma_N)$ are i.i.d. with mean $\tw_N:= \tanh(\beta W_N+B)$.
		\item[(b)] The marginal density of $W_N$ is proportional to $\exp(-Nf(w))$, where $f(w)=\frac{\beta w^2}{2}-\log{\cosh(\beta w+B)}.$
	\end{itemize}
\end{prop}
We state two more lemmas necessary for proving the results of this section, the proofs of which we defer to~\cref{sec:supp}.
The first lemma is a version of the Hanson-Wright inequality, which controls exponential moment of quadratic forms of binary random variables.
\\

\begin{lemma}\label{lem:tailsubG}
	Suppose $X_1,X_2,\ldots ,X_N$, $N\geq 1$ are i.i.d. $\pm 1$ valued random variables such that $\mme[X_1]=\mu$ where $\mu\in (-1,1)$. Define $s_{\mu}:= 2\mu/(\log{(1+\mu)}-\log{(1-\mu)})$ with $s_0$ being $1$. Also assume that $D_N$ is a $N\times N$ symmetric matrix such that $s_\mu \limsup_{N\rightarrow\infty}\lambda_1(D_N) < 1$. Then, given any vector $\mc^{\top}:=(c_1,c_2,\ldots ,c_N)$, we get:
	\begin{equation*}\label{eq:techlem1}
	\log\left\{\mme\left[\exp\left(\frac{1}{2}\sum_{i,j=1}^ND_N(i,j)\tilde{X}_i\tilde{X}_j+\sum_{i=1}^N c_i\tilde{X}_i\right)\right]\right\}\lesssim  \mathrm{Tr}^+(D_N)+\fd+\sum_{i=1}^N c_i^2
	\end{equation*}
	where $\tilde{X}_i=X_i-\mu$ for $1\leq i\leq N$, and $\mathrm{Tr}^+(D_N)=\sum_{i=1}^N\max(D_N(i,i),0)$. %and the hidden constants on the right hand side above only depends on $\rho$.
\end{lemma}
\begin{comment}
Suppose $\ms$ is an observation from the Curie Weiss model $\P^{CW}$. Then for any $N\times N$ symmetric matrix  $V_N$ such that $s_{t}\limsup_{N\rightarrow\infty}\lambda_1(V_N) \lesssim 1$ and  $\mc^{\top}:=(c_1,c_2,\ldots ,c_N)\in \R^N$ we have
\begin{equation*}
\log\left\{\mme\exp\Big[\frac{1}{2}\sum_{i,j=1}^N(\BN(i,j)+\delta V_N(i,j))\sigma_i\sigma_j+\delta\sum_i c_i\sigma_i+\delta g_N(\tw_N) \Big]\right\}\lesssim  \mathrm{Tr}^+(D_n)+\fd+\sum_{i=1}^n c_i^2
\end{equation*}
where  $\mathrm{Tr}^+(D_n):=\sum_{i=1}^N|D_N(i,i)|$, and $s_{t}:= \frac{t}{\tanh^{-1}(t)}$ if $t\ne 0$, and $s_0=1$.
\end{lemma}
\end{comment}
The second lemma gives a quantitative estimate which allows us to neglect the region where $\tw_N$ is not close to $t$.

\begin{lemma}\label{lem:gentail}
	Suppose~\eqref{eq:A1},~\eqref{eq:A2}~and~\eqref{eq:well_connect} holds, and further assume that~$\fa=o(N)$, $\sum_{i=1} (R_i-1)=o(N)$. Also, let $V_N$ be any random variable such that $V_N\leq cN$ for some fixed $c>0$, and $\varepsilon>0$ be fixed. Recalling $\BN:=A_N-\vo\vo^{\top}/N$, for any
	$(\beta,B)\in \Theta_2\cup \Theta_3$ there exists $\delta=\delta(\varepsilon,c,\beta)>0$ such that,
	\begin{align}\label{eq:gentailmain}
	\limsup_{N\rightarrow\infty}\frac{1}{N}\log\mme^{CW}\left[\exp\left(\delta  V_N+\frac{\beta}{2}\ms^{\top}\BN\ms\right)\mathbbm{1}(|\tw_N-M(\ms)|\geq\epsilon)\right]<0.
	\end{align}
\end{lemma}

Additionally the proofs of Theorems~\ref{thm:part},~\ref{thm:conc}, and~\cref{lem:qformuni}, require~\cref{prop:CWresmain} which is stated in~\cref{sec:addresCW}.
\begin{proof}[Proof of~\cref{thm:part}]
	%We split the proof into cases depending on the value of $(\beta,B)$.
	\begin{enumerate}
		\item[(a)]
		To begin, note that
		\begin{align*}
		\frac{\beta}{2}\ms^{\top}A_N\ms+B\sum_{i=1}^N \sigma_i=\frac{\beta}{2}(\ms-t)^{\top}A_N(\ms-t)+\sum_{i=1}^N (\beta t R_i+B)\sigma_i-(\beta t^2/2)\vo^{\top}A_N\vo.
		\end{align*}
		Recall that $\mathcal{M}_N(\beta,B)=N\Big\{\frac{\beta t^2}{2}+Bt-I(t)\Big\}+\frac{\beta t^2}{2}\sum_{i=1}^N (R_i-1)$ as in~\eqref{eq:mf_2}. The above display then gives %~\eqref{eq:model} implies
		\begin{small}
			\begin{align}\label{eq:mainsimplify1}
			&\frac{Z_N(\beta,B)}{\exp(\mathcal{M}_N(\beta,B))}=
			%e^{\frac{\beta t^2}{2}(\sum_{i=1}^N (R_i-1)}}{e^{N(\beta t^2/2+Bt-I(t))}}=
			\mme^{\mmq}\exp\left(\frac{\beta}{2}\sum_{i,j=1}^N (\sigma_i-t)A_N(i,j)(\sigma_j-t)+\beta t \sum_{i=1}^N (R_i-1)(\sigma_i-t)\right)
			\end{align}
		\end{small}
		where $\mmq$ is the measure induced by $N$ independent $\pm 1$ valued random variables with mean $t$ (as defined in~\eqref{eq:realdist1}). In this case with $D_N=\beta A_N$ we have $$s_t\limsup_{N\rightarrow\infty}\lambda_1(D_N)=\beta s_t\limsup_{N\rightarrow\infty}\lambda_1(A_N)\le \beta s_t=\frac{\beta t}{\beta t+B}<1,$$
		for $(\beta,B)\in\Theta_{12}$. If $(\beta,B)\in\Theta_{11}$, then we have $t=0$, and $s_0=1$, and so with $D_N=\beta A_N$ as before, we have $s_t\limsup_{N\rightarrow\infty}\lambda_1(D_N) =\beta<1$. Thus in both cases~\cref{lem:tailsubG} is applicable with $D_N=\beta A_N, c_i=\beta t(R_i-1)$,  which using \eqref{eq:mainsimplify1} gives
		\begin{align}\label{eq:carry_d1}
		&\;\;\;\;\log{\mme^{\mmq}\exp\left(\frac{\beta}{2}\sum_{i,j=1}^N (\sigma_i-t)A_N(i,j)(\sigma_j-t)+\beta t \sum_{i=1}^N (R_i-1)(\sigma_i-t)\right)}\nonumber \\ & \lesssim \fa +t^2\sum_{i=1}^N(R_i-1)^2.
		\end{align}
		The conclusion of part (a) follows from this combined with \eqref{eq:mainsimplify1}.
		%Invoking proposition \ref{prop:CWbasic} we have that $\sigma_1,\cdots,\sigma_N$ are i.i.d. given $\tw_N$. Recall the definition of $s_\mu$ from Lemma \ref{lem:tailsubG}, and use the equation $t=\tanh(\beta t+B)$ to note that $s_t=\frac{t}{\beta t+B}$, and so $\beta s_t=\frac{\beta t}{\beta t+B}<1$. Since $\mu\mapsto s_\mu$ is continuous, we can choose $\varepsilon>0$ such that $\rho:=\sup_{\mu\in [t\pm \varepsilon]}\beta s_\mu<1$. Thus, on the set $|\tw_N-t|\le \varepsilon$ we have $\beta s_{\tw_N}\lambda_1(A_N)\le \beta \rho\lambda_1(A_N)=\beta\rho(1+o(1))$, and so invoking Lemma \ref{lem:tailsubG} with $\tilde{W}_N\in [t\pm \varepsilon]$ gives
		%Since $B=0$, using symmetry we get
		%\begin{align*}
		%\frac{Z_N(\beta,B)}{Z_N^{CW}(\beta,B)}=&2\E^{CW} \exp(\frac{\beta}{2}\ms'\BN\ms)\mathbbm{1}(\tw_N>0)\\
		%=&2\E^{CW} \exp(\frac{\beta}{2}\ms'\BN\ms)\mathbbm{1}(\tw_N>0,|\tw_N-t|>\varepsilon)+2\E^{CW} \exp(\frac{\beta}{2}\ms'\BN\ms)\mathbbm{1}(\tw_N>0,|\tw_N-t|\le \varepsilon),
		%\end{align*}
		%where the first term in the RHS above converges to $0$ by part (ii) of Lemma \ref{lem:gentail}. 
		
		\item[(b)]%{$(\beta,B)\in \Theta_2$ :} 
		Define \begin{align}\label{eq:y}
		Y_N:=(\ms-\tw_N)^{\top}\BN(\ms-\tw_N)+2\tw_N\sum_{i=1}^N (R_i-1)(\sigma_i-\tw_N)+(\tw_N^2-t^2)\sum_{i=1}^N(R_i-1),
		\end{align} and note that
		$\ms^{\top}\BN\ms=Y_N+t^2\sum_{i=1}^N (R_i-1).$
		Using this, with  $J_{N,\epsilon}:=\{ |t|-\varepsilon\le |\tw_N|\le |t|+\varepsilon\}$ for some $\epsilon>0$, by a similar calculation as in part (a) we have:
		\begin{align}\label{eq:hw_20}
		&\;\;\;\;\notag\frac{Z_N(\beta,B)}{Z_N^{CW}(\beta,B)}\nonumber \\ &=\E^{CW} \exp\left(\frac{\beta}{2}\ms^{\top}\BN\ms\right)\nonumber \\
		&=\E^{CW} \bigg[\exp\left(\frac{\beta}{2}\ms^{\top}\BN\ms\right)\mathbbm{1}(J_{N,\epsilon}^c)\bigg]+ \exp\Big(\frac{\beta}{2}t^2\sum_{i=1}^N(R_i-1)\Big)\E^{CW}(e^{\frac{\beta}{2}Y_N}\mathbbm{1}(J_{N,\epsilon}))\big].
		\end{align}
		The first term in the right hand side of~\eqref{eq:hw_20} is $o(1)$ by invoking~\cref{lem:gentail} with $\delta=0$. For the second term, by~\cref{prop:CWbasic}, the inner (conditional) expectation is taken with respect to i.i.d. $\pm 1$ valued random variables with mean $\tw_N$. In this regime $\beta s_t=\beta t/\beta t=1$. But since $\limsup_{N\rightarrow\infty}\lambda_1(\BN)<1$ by \eqref{eq:well_connect}, on the set $J_{N,\varepsilon}$ we have $$\limsup_{N\rightarrow\infty}s_{\tw_N}\lambda_1(\beta \BN)\le \limsup_{N\rightarrow\infty}\sup_{\mu\in J_{N,\varepsilon}}s_\mu \lambda_1(\beta \BN)<1$$ for $\varepsilon$ small enough. Therefore, Lemma \ref{lem:tailsubG} is applicable with $D_N=\beta \BN$ and $c_i=2\tw_N(R_i-1)$ to give	
		\begin{align*}%\label{eq:carry_1}
		\log \E^{CW}(e^{\frac{\beta}{2}Y_N}\mathbbm{1}(J_{N,\epsilon})|\tw_N)\le C\Big\{  \fa+\sum_{i=1}^N(R_i-1)^2\Big\}+\frac{\beta }{2}\Big|(\tw_N^2-t^2)\sum_{i=1}^N(R_i-1)\Big|
		\end{align*}
		for some $C<\infty$, which on taking another expectation gives
		\begin{small}
		\begin{align}\label{eq:carry_2}
		\notag\log \E^{CW}(e^{\frac{\beta}{2}Y_N}\mathbbm{1}(J_{N,\epsilon}))\le &C\Big\{  \fa+\sum_{i=1}^N(R_i-1)^2\Big\}+\log \E e^{\frac{\beta }{2}\Big|(\tw_N^2-t^2)\sum_{i=1}^N(R_i-1)\Big|}\\
		\lesssim &\fa+\sum_{i=1}^N(R_i-1)^2+\frac{1}{N}\Big[\sum_{i=1}^N(R_i-1)\Big]^2,
		\end{align}
		\end{small}
		where the last step uses part (b) of~\cref{prop:CWresmain}. This along with \eqref{eq:hw_20} gives
		%\begin{align*}
		%\notag&\log \E e^{\frac{\beta}{2}\sigma'\BN\sigma}-\frac{\beta t^2}{2}\sum_{i=1}^N(R_i-1)\lesssim 
		%\tw_N^2\sum_{i=1}^N(R_i-1)\E^{CW}(e^{\frac{\beta}{2}Y_N}|\tw_N)\mathbbm{1}(J_{N,\epsilon})\big].
		%\notag	\le &C\{\fa+\sum_{i=1}^N(R_i-1)^2\}+\log \E e^{\frac{\beta}{2}(\tw_N^2-t^2)\sum_{i=1}^N(R_i-1)}\\
		%	\lesssim &\fa+\sum_{i=1}^N(R_i-1)^2+\frac{1}{N}\Big[\sum_{i=1}^N(R_i-1)\Big]^2,
		%\end{align*}
		%where the last line uses~\cref{prop:CWresmain}. Using \eqref{eq:hw_20}, this gives
		%On taking 
		\begin{align*}
		\log{Z_N(\beta,B)}-\log{Z_N^{CW}(\beta,B)} -\frac{\beta t^2}{2}\sum_{i=1}^N(R_i-1)\lesssim& \fa+\sum_{i=1}^N(R_i-1)^2,
		\end{align*}
		from which the desired conclusion follows by another application of part (a) of \cref{prop:CWresmain} to note that $\log Z_N^{CW}(\beta,B)- N\left[ \frac{\beta t}{2}+Bt-I(t)\right]\lesssim1$.
		
		%where the first term in the RHS above goes to $0$ exponentially fast, on invoking part (i) of Lemma \ref{lem:gentail}. For dealing with the second term, note that
		%We now split the proof into cases depending on the value of $(\beta,B)$.
		%Now, $\limsup_{N\rightarrow\infty}\lambda_1(
		%\begin{align}\label{eq:hw_20}
		%\exp\left(\frac{\beta}{2}\ms'A_n\ms+B\sum_{i=1}^N\sigma_i\right)$
		%\E^{CW}\exp\left(\frac{\beta}{2}\ms'\BN\ms\right)\mathbbm{1}(|\tw_N-t|\le \epsilon)=&\E^{CW}\Big[\E^{CW}(e^{\frac{\beta}{2}Y_N}|\tw_N)\mathbbm{1}(|\tw_N-t|\le \epsilon)e^{\frac{\beta}{2}\tw_N^2\sum_{i=1}^N(R_i-1)}\Big]
		%(\ms-\tw_N)'\BN(\ms-\tw_N)+\beta \tw_N\sum_{i=1}^N(R_i-1)(\sigma_i-\tw_N)+\frac{\beta\tw_N^2}{2}\sum_{i=1}^N(R_i-1)\right)
		%\end{align}
		%We further split the proof into sub-cases based on whether $(\beta,B)\in \Theta_{11}$ or $\Theta_{12}$.
		\item[(c)]%{$(\beta,B)\in \Theta_3$}
		In this regime we have $t=0$, and so $s_t=s_0=1$, and $\beta s_0=1$. As in the proof of part (b), the first term in the RHS of \eqref{eq:hw_20} is $o(1)$ invoking~\cref{lem:gentail} with $\delta=0$. For handling the second term,
		invoking \eqref{eq:well_connect} gives $$\limsup_{N\rightarrow\infty}s_{\tw_N}\lambda_1( \BN)\le \limsup_{N\rightarrow\infty}\sup_{\mu\in J_{N,\varepsilon}}s_\mu \lambda_1( \BN)<1$$ for $\varepsilon$ small enough. Also 
		Lemma \ref{lem:tailsubG} with $D_N=\BN$ and $c_i=2\tw_N(R_i-1)$  gives
		\begin{align*}
		\log \E^{CW}(e^{\frac{\beta}{2}Y_N}\mathbbm{1}(J_{N,\epsilon})|\tw_N)\le C\Big\{\fa+\tw_N^2\sum_{i=1}^N(R_i-1)^2\Big\}+\frac{\beta}{2}\tw_N^2\sum_{i=1}^N(R_i-1)
		\end{align*}
		for some $C<\infty $ free of $N$. This, on taking another expectation
		along with \eqref{eq:hw_20} gives
		\begin{small}
		\begin{align}\label{eq:carry_3}
		\notag\log \E^{CW}(e^{\frac{\beta}{2}Y_N}\mathbbm{1}(J_{N,\epsilon}))\le &C\fa+\log \E \exp\left(C\tw_N^2\sum_{i=1}^N(R_i-1)^2+\frac{\beta}{2}\tw_N^2\sum_{i=1}^N(R_i-1)\right)\\
		\lesssim &\fa+\frac{1}{N}\Big[\sum_{i=1}^N(R_i-1)+\sum_{i=1}^N(R_i-1)^2\Big]^2,
		\end{align}
		\end{small}
		%\begin{align*}
		% \log{Z_N(\beta,B)}-\log{Z_N^{CW}(\beta,B)}
		% \le& C\fa+\log\E \exp\left(\frac{\beta}{2}\tw_N^2\sum_{i=1}^N(R_i-1)+C\tw_N^2\sum_{i=1}^N(R_i-1)^2\right)\\
		% \lesssim&\fa+\frac{1}{N}\Big[\sum_{i=1}^N(R_i-1)+\sum_{i=1}^N(R_i-1)^2\Big]^2
		%\end{align*}
		where the last bound uses part (b) of~\cref{prop:CWresmain}.
		Combining \eqref{eq:hw_20} and \eqref{eq:carry_3} gives
		\[\log Z_N(\beta,B)-\log Z_N^{CW}(\beta,B)\lesssim \fa+\frac{1}{N}\Big[\sum_{i=1}^N(R_i-1)]^2+\frac{1}{N}[\sum_{i=1}^N(R_i-1)^2\Big]^2.\]
		%A similar proof as in $(\beta,B)\in \Theta_{2}$ then completes the proof. 
		We incur an additional log factor in the final answer because $\log Z_N^{CW}(\beta,B)- N\left[ \frac{\beta t}{2}+Bt-I(t)\right]\lesssim \log N$ by part (a) of~\cref{prop:CWresmain}.
		
	\end{enumerate}
\end{proof}
\begin{proof}[Proof of Theorem \ref{thm:conc}]
	\begin{enumerate}
		\item[(a)]
		Using a similar calculation as in~\eqref{eq:hw_20}, we get:
		\begin{align}\label{eq:hw_21}
		&\;\;\;\;(c(N))^{-1}\mmp(\ms\in \mE_N)\nonumber \\&=\mme^{\mmq}\left[\exp\left(\frac{\beta}{2}\sum_{i,j=1}^N(\sigma_i-\ut)A_N(\sigma_j-\ut)+\beta\ut\sum_{i=1}^N (R_i-1)(\sigma_i-\ut)\right)\mathbbm{1}(\ms\in \mE_N)\right]
		\end{align}
		where the deterministic sequence $c(N)$ satisfies
		\begin{align*}
		c(N)=\frac{\exp(\beta \ut^2(\vo^{\top}A_N\vo-N))\left(\exp(\beta \ut+B)+\exp(-\beta \ut-B)\right)^N}{Z_N(\beta,B)\exp\left((\beta \ut^2/2)\vo^{\top}A_N\vo\right)} \leq 1,
		\end{align*}
		on invoking the Mean-Field lower bound \eqref{eq:mf_2}. %The last inequality follows by using the Mean-Field approximation (see~\cite[Theorem 1.6]{ChaDembo2016}) to note that
		%\begin{align}\label{eq:mf_2}
		%\log Z_N(\beta,B)\ge N\left[\frac{\beta t^2}{2}+Bt-I(t)\right]+\frac{\beta t^2}{2}\sum_{i=1}^N (R_i-1).
		%\end{align}
		\noindent Next, by using H\"older's inequality with exponent $p$ (to be chosen later), the left hand side of~\eqref{eq:hw_21} can be bounded above by,
		\begin{small}
		\begin{align}\label{eq:holder}
		\left\{\mme^{\mmq}\exp\left(\frac{\beta(1+p)}{2}\sum_{i,j=1}^N(\sigma_i-\ut)A_N(\sigma_j-\ut)+\beta\ut(1+p)\sum_{i=1}^N (R_i-1)(\sigma_i-\ut)\right)\right\}^{\frac{1}{1+p}} (\mmq(\mE_N))^{\frac{p}{1+p}}.
		\end{align}
		\end{small}
		Using arguments similar to the derivation of \eqref{eq:carry_d1} shows that for $p$ small enough we have
		%We once again have $s_t\limsup\limits\limits_{N\to\infty}\lambda_1(\beta(1+p)A_N)<1$ for a small enough $p>0$ and consequently, an application of~\cref{lem:tailsubG} with $D_N=\beta(1+p)A_N$ and $c_i=\beta(1+p)(R_i-1)$ gives
		\begin{align*}%\label{eq:holder_2}
		&\;\;\;\log \mme^{\mmq}\left[\exp\left(\frac{\beta(1+p)}{2}\sum_{i,j}(\sigma_i-\ut)A_N(\sigma_j-\ut)+\beta\ut(1+p)\sum_{i=1}^N (R_i-1)(\sigma_i-\ut)\right)\right]\\ &\lesssim \fa+t^2\sum_{i=1}^N(R_i-1)^2.
		\end{align*}
		Combining this along with \eqref{eq:hw_21} and \eqref{eq:holder} gives the desired conclusion.
		\item[(b)]%{$(\beta,B)\in \Theta_2$:}
		With $Y_N$ as in \eqref{eq:y}, using a similar calculation as in the derivation of \eqref{eq:hw_20} we can bound $P(\ms\in \mE_N)$ by
		\begin{align}\label{eq:part_2}
		\notag&\frac{Z_N^{CW}(\beta,B)}{Z_N(\beta,B)}\E^{CW}e^{\frac{\beta}{2}{\ms^\top\BN\ms}}\mathbbm{1}(\ms\in \mE_N)\\
		\le &\frac{Z_N^{CW}(\beta,B)}{Z_N(\beta,B)}\left[\E^{CW}e^{\frac{\beta}{2}{\ms^\top\BN\ms}}\mathbbm{1}(\ms\in J_{N,\varepsilon}^c)+e^{\frac{\beta t^2}{2}\sum_{i=1}^N(R_i-1)}\E^{CW}e^{\frac{\beta}{2}Y_N}\mathbbm{1}(\ms\in \mE_N)\mathbbm{1}( J_{N,\varepsilon})\right].
		\end{align}
		For controlling the ratio of partition functions in the RHS of \eqref{eq:part_2}, use the Mean-Field approximation \eqref{eq:mf_2} to get a lower bound for $\log{Z_N(\beta,B)}$, whereas part (a) of Proposition \ref{prop:CWresmain} gives
		%\label{eq:mf_3}
		$\log Z_N^{CW}(\beta,B)- N\left[ \frac{\beta t}{2}+Bt-I(t)\right]\lesssim 1.$
		Combining these two observations, we get:
		\begin{align}\label{eq:mf_4}
		\log Z_N^{CW}(\beta,B)-\log Z_N(\beta,B)+\frac{\beta t^2}{2}\sum_{i=1}^N(R_i-1)\lesssim 1.
		\end{align}
		Also, the first term inside the parenthesis in the RHS of \eqref{eq:part_2} is exponentially small in $N$, by invoking Lemma \ref{lem:gentail} with $\delta=0$. Proceeding to control the second term in the RHS of \eqref{eq:part_2} we have 
		\begin{align}\label{eq:part_3}
		\E^{CW}e^{\frac{\beta}{2}Y_N}\mathbbm{1}(\ms\in \mE_N)\mathbbm{1}( J_{N,\varepsilon})\le &\Big[\E^{CW}e^{\frac{\beta(1+p)}{2}Y_N}\mathbbm{1}( J_{N,\varepsilon})\Big]^{\frac{1}{1+p}}\left[\P^{CW}(\ms\in \mE_N)\right]^{\frac{p}{1+p}},
		\end{align}
		where the last step uses Holder's inequality for any $p>0$. For controlling the first term inside the bracket in the RHS of \eqref{eq:part_3}, by choosing $p>0$ small enough and repeating the same argument as in the derivation of \eqref{eq:carry_2}, we get:
		%As before, we choose $\varepsilon>0$ small enough such that $\beta\sup_{\mu\in [t\pm \varepsilon]}s_\mu<1$, and now we choose $p>0$ small enough such that $\beta (1+p)\sup_{\mu\in [t\pm \varepsilon]}s_\mu<1$ as well. After this, repeating the proof of the previous theorem we get the bound
		\begin{align}\label{eq:mf_5}
		\log \E^{CW}(e^{\frac{\beta(1+p)}{2}Y_N}) \mathbbm{1}(J_{N,\varepsilon})\lesssim \fa+\sum_{i=1}^N(R_i-1)^2.%\frac{\beta(1+p)}{2}\Big|(\tw_N^2-t^2)\sum_{i=1}^N(R_i-1)\big|%+\Big|\sum_{i=1}^N(R_i-1)\Big|.
		\end{align}
		%for some $C_p<\infty$ free of $N$.
		%where the first term is exponentially small by part (i) of Lemma \ref{lem:gentail}, and using Proposition \ref{prop:CWresmain}	
		%Noting that the first term in the RHS of \eqref{eq:part_2} is exponentially small using Lemma \ref{lem:gentail}, it suffices to control the second term. For this, fixing $p>0$ we use Holder's inequality to get 
		%	\begin{align}\label{eq:holder}
		%	\E^{CW}e^{\frac{\beta}{2}\ms'\BN\ms}\mathbbm{1}(\ms\in \mE_N,|\tw_N-t|\le \epsilon)
		%	\le &\left[\E^{CW}e^{\frac{\beta(1+p)}{2}\ms'\BN\ms}\mathbbm{1}(|\tw_N-t|\le \epsilon)\right]^{\frac{1}{1+p}}\left[\P^{CW}(\ms\in \mE_N)\right]^{\frac{p}{1+p}}.
		%\end{align}
		Combining \eqref{eq:part_2}, \eqref{eq:mf_4}, \eqref{eq:part_3} and \eqref{eq:mf_5}, the desired conclusion follows.
		%\begin{align*}\log \mmp(\ms\in \mE_N)\lesssim \fa+\sum_{i=1}^N(R_i-1)^2+\log \mmp^{CW}(\ms\in \mE_N)+\log \E e^{\frac{\beta(1+p)}{2}|(\tw_N^2-t^2)\sum_{i=1}^N(R_i-1)|}\\
		%\le 
		\item[(c)]%{$(\beta,B)\in \Theta_3$:}
		%Proceeding as in the proof of part (b) above, we
		All steps of part (b) above go through verbatim, except the RHS of \eqref{eq:mf_4} gets replaced by $\log N$ (by part (a) of~\cref{prop:CWresmain}), and \eqref{eq:mf_5} is replaced by (c.f.~\eqref{eq:carry_3}) %As in part (b) above, the main step is to estimate the second term in the RHS of \eqref{eq:part_2}. To this effect, repeating the same argument as in the derivation of \eqref{eq:carry_3} with $p>0$ small enough yields: 
		%As before, we choose $\varepsilon>0$ small enough such that $\beta\sup_{\mu\in [t\pm \varepsilon]}s_\mu<1$, and now we choose $p>0$ small enough such that $\beta (1+p)\sup_{\mu\in [t\pm \varepsilon]}s_\mu<1$ as well. After this, repeating the proof of the previous theorem we get the bound
		\begin{align}\label{eq:mf_5n}
		\log \E^{CW}e^{\frac{\beta(1+p)}{2}Y_N}\mathbbm{1}(J_{N,\varepsilon})\lesssim \fa+\frac{1}{N}\Big[\sum_{i=1}^N(R_i-1)^2\Big]^2+\frac{1}{N}\Big[\sum_{i=1}^N(R_i-1)\Big]^2.
		\end{align}
		%Also, in this case using the Mean-Field approximation \eqref{eq:mf_2} along with  Proposition \ref{prop:CWresmain} gives
		%\begin{align*}%\label{eq:mf_4n}
		%\log Z_N^{CW}(\beta,B)-\log Z_N(\beta,B)\lesssim \log N.
		%\end{align*}
		Combining this with \eqref{eq:part_2} and \eqref{eq:mf_5n} gives the desired conclusion.
	\end{enumerate}
\end{proof}
\begin{proof}[Proof of~\cref{lem:qformuni}]
	\begin{enumerate}
		\item[(a)]%{$(\beta,B)\in\Theta_1$ :}
		Invoking Theorem \ref{thm:conc} and changing $\delta$ if necessary, it suffices to show the desired conclusion under $\mmq$, where $\mmq$ is the i.i.d. measure induced by $N$ $\pm 1$ valued random variables with mean $t$, as defined in~\eqref{eq:realdist1}. A direct calculation shows that $m_i(\ms)-t$ equals $\sum_{j=1}^NA_N(i,j)(\sigma_j-t)+t(R_i-1)$, and so
		\begin{align}
		\notag\sum_{i=1}^N\Big(m_i(\ms)-t\Big)^2 \le &2\sum_{i=1}^N\Big[\sum_{j=1}^N A_N(i,j)(\sigma_j-t)\Big]^2+2t^2\sum_{i=1}^N(R_i-1)^2\\
		=&2\sum_{i=1}^N\sum_{j=1}^N (A_N^2)(i,j)(\sigma_i-t)(\sigma_j-t)+2t^2\sum_{i=1}^N(R_i-1)^2\label{eq:m_21}.
		\end{align}
		It therefore suffices to control the exponential moment of the first term in the RHS of \eqref{eq:m_21}. %Since Proposition \ref{prop:CWresmain} gives the existence of $\delta>0$ such that $\E^{CW} e^{N\delta (\tw_N-t)^2}<\infty$, it suffices to the existence of $\delta>0$ such that
		Since $\limsup_{N\rightarrow\infty}\lambda_1(A_N^2)\le 1$, for any $\delta\in (0,1/2)$, using Lemma \ref{lem:tailsubG} with $D_N=\delta A_N^2$ and $c_i=0$ we have
		\begin{align*}
		\log \E^{\mmq}\exp\left(\delta (\ms-t)^\top A_N^2(\ms-t)\right)\lesssim {\lVert A_N^2\rVert_F^2}=\sum_{i=1}^N\lambda_i^4\lesssim \sum_{i=1}^N\lambda_i^2=\fa.
		\end{align*}
		This gives the desired conclusion.
		
		\item[(c)]%{$(\beta,B)\in \Theta_2$ :}
		By invoking Theorem \ref{thm:conc}, it suffices to show the desired conclusion under the Curie-Weiss model. Start by noting that $\bar{\boldsymbol{m}}(\ms)=\frac{1}{N}\sum_{i=1}^NR_i\sigma_i$, and so
		\begin{align*}%\label{eq:april_13}
		m_i(\ms)-\bar{\boldsymbol{m}}(\ms)=\sum_{j=1}^NA_N(i,j)(\sigma_j-\tw_N)+\frac{1}{N}\sum_{i=1}^NR_i(\sigma_i-\tw_N)
		+\tw_N(R_i-\bar{R}).
		\end{align*}
		This shows that $\sum_{i=1}^N\Big(m_i(\ms)-\bar{\boldsymbol{m}}(\ms)\Big)^2$ is bounded by 
		\begin{align}\label{eq:april_13.1}
		\notag&3 \sum_{i=1}^N\Big[\sum_{j=1}^NA_N(i,j)(\sigma_j-\tw_N)\Big]^2+\frac{3}{N}\Big[\sum_{i=1}^NR_i(\sigma_i-\tw_N)\Big]^2+3\tw_N^2\sum_{i=1}^N(R_i-\bar{R})^2\\
		\le &3 \sum_{i,j=1}^N\left((A_N^2)(i,j)+\frac{3}{N}R_iR_j\right)(\sigma_i-\tw_N)(\sigma_j-\tw_N)+3\tw_N^2\sum_{i=1}^N(R_i-1)^2.
		\end{align}
		Conditioning on $\tw_N$, we now control the exponential moment of the first term in the RHS of the above display under the Curie-Weiss model. By Proposition \ref{prop:CWbasic}, under the Curie Weiss model, given $\tw_N$, the random vector $(\sigma_1,\cdots,\sigma_N)$ are i.i.d. with mean $\tw_N$. Note that $$\limsup_{N\rightarrow\infty}\lambda_1\Big( A_N^2+\frac{3}{N}{\bf R}{\bf R}^\top\Big)\le \limsup_{N\rightarrow\infty}\lambda_1(A_N^2)+\limsup_{N\rightarrow\infty}\frac{3}{N}\lambda_1({\bf R}{\bf R}^\top)\lesssim 1,$$ 
		$${\lVert \frac{1}{N}{\bf R}{\bf R}^\top\rVert_F^2}=\frac{1}{N^2}(\sum_{i=1}^NR_i^2)^2\lesssim 1$$ where the last display follows from the assumption that	$\max_{1\leq i\leq N} R_i\lesssim 1$ by~\eqref{eq:A2}. Based on these observations, on invoking Lemma \ref{lem:tailsubG} with $c_i=0$, $D_N=\delta \Big(A_N^2+\frac{3}{N}{\bf R}{\bf R}^\top\Big)$ for $\delta$ small enough, we get
		\begin{align*}%\label{eq:covid}
		\log \E^{CW} e^{\delta \sum_{i=1}^N(m_i(\ms)-\bar{\boldsymbol{m}}(\ms))^2}-\log \E^{CW}e^{3\delta\tw_N^2\sum_{i=1}^N(R_i-1)^2} \lesssim\lVert A_N^2\rVert_F^2+\text{tr}(A_N^2) \lesssim \fa,%\frac{1}{N}\Big[\sum_{i=1}^N(R_i-1)^2\Big]^2,
		\end{align*}
		from which the desired conclusion follows on noting that \begin{align*}
		\log \E^{CW}e^{3\delta\tw_N^2\sum_{i=1}^N(R_i-1)^2}\lesssim \frac{1}{N}\big[\sum_{i=1}^N(R_i-1)^2\big]^2,
		\end{align*} which follows from part (b) of Proposition \ref{prop:CWresmain}.
		%The conclusion for $(\beta,B)\in\Theta_{3}$ is immediate from \eqref{eq:covid}.
		% For $(\beta,B)\in\Theta_2$, note that by similar calculations as before,
		\item[(b)]
		To begin, note that
		\begin{align}\label{eq:holder_b}
		\sum_{i=1}^N (m_i-M(\ms))^2&\lesssim \sum_{i=1}^N(m_i(\ms)-\bar{\boldsymbol{m}}(\ms))^2+\frac{1}{N}\Big[\sum_{i=1}^NR_i(\sigma_i-\tw_N)\Big]^2\nonumber \\ &\qquad +(\tw_N-M(\ms))^2\bigg|\sum_{i=1}^N (R_i-1)\bigg|.
		\end{align}
		By H\"older's inequality, it suffices to bound the exponential moments of the three terms of the above display at some $\delta>0$. Exponential moment of the third term in the RHS of \eqref{eq:holder_b} is  bounded by part (b) of~\cref{prop:CWresmain}, as $\sum_{i=1}^N|R_i-1|=o(N)$. Proceeding to bound the sum of the first two terms, use  \eqref{eq:april_13.1} to get
		\begin{align*}
		&\;\;\;\sum_{i=1}^N(m_i(\ms)-\bar{\boldsymbol{m}}(\ms))^2+\frac{1}{N}\Big[\sum_{i=1}^NR_i(\sigma_i-\tw_N)\Big]^2\\ &\le  \sum_{i,j=1}^N\left((A_N^2)(i,j)+\frac{4}{N}R_iR_j\right)(\sigma_i-\tw_N)(\sigma_j-\tw_N)+3\sum_{i=1}^N(R_i-1)^2,\end{align*}
		and so it suffices to bound $$\log \E^{CW}\exp\left(\delta\sum_{i,j=1}^N\left((A_N^2)(i,j)+\frac{4}{N}R_iR_j\right)(\sigma_i-\tw_N)(\sigma_j-\tw_N)\right)$$
		for $\delta$ small enough. But this follows on invoking~\cref{lem:tailsubG} with $D_N=\delta(A_N^2+\frac{4}{N}{\bf R}{\bf R}^\top)$ and $c_i=0$ to get
		\begin{align*}&\;\;\;\;\log \E^{CW}\exp\left(\delta\sum_{i,j=1}^N\left((A_N^2)(i,j)+\frac{4}{N}R_iR_j\right)(\sigma_i-\tw_N)(\sigma_j-\tw_N)\right)\\ &\lesssim \lVert A_N^2\rVert_F^2+\text{tr}(A_N^2)\lesssim \fa,\end{align*}
		which completes the proof of part (b).
		%The second term can be bounded similarly as before by combining~\cref{lem:tailsubG} with~\cref{prop:CWbasic},
		%as $\mme^{CW}[\exp(\delta N(\tw_N-M(\ms))^2)]\lesssim 1$ for small enough positive $, by~\cref{prop:CWresmain}.

	\end{enumerate}
\end{proof}
\section{Proof of Lemmas \ref{lem:unifbd} and \ref{lem:linfbdgen}}\label{sec:prooflem2}
\begin{proof}[Proof of~\cref{lem:unifbd}]
	\begin{enumerate}
		\item[(a)]
		To begin, note that  it suffices to prove the bound for $\lambda$ large enough. To this effect, using part (b) of Lemma \ref{lem:auxtail} we have the existence of a constant $M$ free of $N$, such that for all $\lambda>0$ we have
		\begin{align*}
		\P\bigg(|m_i(\ms)-\sum_{j=1}^NA_N(i,j)\tanh(\beta m_j(\ms))|>\lambda\sqrt{\log N\sum_{j=1}^NA_N(i,j)^2}\bigg)\le 2 e^{-\frac{\lambda^2\log N}{M}},
		\end{align*}
		which on using a union bound with $\alpha_N=\max_{1\le i\le N}\sum_{j=1}^NA_N(i,j)^2$ (as in~\cref{theo:uniq1}) gives
		\begin{align*}
		\P\bigg(\max_{1\le i\le N}|m_i(\ms)-\sum_{j=1}^NA_N(i,j)\tanh(\beta m_j(\ms))|>\lambda \sqrt{\alpha_N\log N}\bigg)\le 2N e^{-\frac{\lambda^2\log N}{M}}.
		\end{align*}
		On the set $\Big\{\max_{1\le i\le N}|m_i(\ms)-\sum_{j=1}^NA_N(i,j)\tanh(\beta m_j(\ms))|\le \lambda\sqrt{ \alpha_N\log N}\Big\}$ using the bound $|\tanh(x)|\le |x|$  we have
		\begin{align*}
		\max_{1\le i\le N}|m_i(\ms)|\le  \sqrt{\alpha_N\log N }+\beta \max_{1\le i\le N}R_i  \max_{1\le i\le N}|m_i(\ms)|,
		\end{align*}
		which on using the fact that $\max_{1\leq i\leq N} R_i\to 1$ (see~\eqref{eq:A3}) gives $
		\max_{1\le i\le N}|m_i(\ms)|\lesssim \sqrt{\alpha_N\log N}$.
		Thus there exists a constant $c'$ such that
		\begin{align*}
		\P(\max_{1\le i\le N}|m_i(\ms)|>c'\lambda \sqrt{\alpha_N\log N})\le 2Ne ^{-\frac{\lambda^2\log N}{M}},
		\end{align*}
		from which the desired conclusion follows for all $\lambda$ large enough.
		
		\item[(b)]
		More generally, we will show that for any vector ${\bf c}\in \R^N$ we have
		\begin{align}\label{eq:claim_20}
		\E \left(\sum_{i=1}^Nc_i\sigma_i\right)^2\lesssim (\log N)^{3/2} \sum_{i=1}^Nc_i^2.
		\end{align}
		To this effect, for every non-negative integer $\ell$ set
		%The proof proceeds along similar lines as~\cref{lem:contrcrit} but in this case, we don't require any assumption on $\sum_{i=1}^N c_i$. Set 
		$\mathbf{c}^{(\ell)}:= \beta^{\ell}A_N^\ell{\mathbf c}$, and $x_\ell:= \mathbb{E}[(\sum_{i} c^{(\ell)}_i\sigma_i)^2]$, and note that ${\mathbf c}^{(0)}={\mathbf c}$, and the LHS of \eqref{eq:claim_20} is just $x_0$.
		%Also, let $\beta^*:= (1+\beta)/2$. Note that for all large enough $N$, we have $\max\{\lambda_1(A_N),-\lambda_N(A_N)\}\leq \beta^*$. Throughout this proof we will use $D$ to denote absolute constants which may change from one line to another. 
		Now, for any $\ell\ge 0$ we can write 
		\begin{align}\label{eq:gather}
		x_\ell=T_{1,\ell}+T_{2,\ell}+T_{3,\ell},
		\end{align}
		where
		\begin{gather*}%\label{eq:gather_1}
		T_{1,\ell}:= \mathbb{E}\left[\left(\sum_{i=1}^N c^{(\ell)}_i(\sigma_i-\tanh(\beta m_i(\ms)))\right)^2\right]\;\; ,\;\; T_{2,\ell}:= \mathbb{E}\left[\left(\sum_{i=1}^N c^{(\ell)}_i\tanh(\beta m_i(\ms))\right)^2\right]\\ T_{3,\ell}\:= 2\mathbb{E}\left[\left(\sum_{i\ne j}c^{(\ell)}_i c^{(\ell)}_j (\sigma_i-\tanh(\beta m_i(\ms)))\tanh(\beta m_j(\ms))\right)\right].
		\end{gather*}
		For controlling $T_{3,\ell}$, setting $m_i^j(\ms):=\sum_{k=1,k\neq j}^N A_N(i,k)\sigma_k\sigma_j$ we have
		\begin{align}\label{eq:gather_1}
		%&2\Bigg|\mathbb{E}\Big[\Big(\sum_{i,j}c^{(k)}_i c^{(k)}_j (\sigma_i-\tanh(\beta m_i(\ms)))\tanh(\beta m_j(\ms))\Big)\Big]\Bigg|\\ 
		\notag |T_{3,\ell}|=&  2\Bigg|\sum_{i\ne j}^Nc^{(\ell)}_i c^{(\ell)}_j \mathbb{E}\Big[(\sigma_i-\tanh(\beta m_i(\ms)))(\tanh(\beta m_j(\ms))-\tanh(\beta m_j^i(\ms)))\Big]\Bigg|\\ {\lesssim} &  \sum_{i\ne j}^N\big|c^{(\ell)}_i\big|\big|c^{(\ell)}_j\big|A_N(i,j)\lesssim \lVert \mc^{(\ell)}\rVert_2^2 
		\end{align}
		where, in the first line, we use $\mme[\sigma_i-\tanh(\beta m_i(\ms))|(\sigma_j,j\neq i)]=0$ and consequently $\mathbb{E}\Big[(\sigma_i-\tanh(\beta m_i(\ms)))\tanh(\beta m_j^i(\ms))\Big]=0$ for $i\neq j$. The bound $|\tanh(\beta m_i(\ms))-\tanh(\beta m_i^j(\ms))|\lesssim A_N(i,j)$ is used in the second line.
		
		Proceeding to bound $T_{2,\ell}$, use %For $B=0$ and $\beta<1$, note that there exists $\xi_i$'s (uniformly bounded random variables) 
		a Taylor's series expansion to get  $\tanh(\beta m_i(\ms))=\beta m_i(\ms)+\xi_i m_i(\ms)^3$ for random variables $\{\xi_i\}_{1\le i\le N}$ uniformly bounded by $1$ in absolute value. Also note that
		$$x_{\ell+1}=\E\left[\left(\left(\mc^{(\ell+1)}\right)^{\top}\ms\right)^2\right]=\E\left[\left(\beta \left(\mc^{(\ell)}\right)^{\top}A_N\ms\right)^2\right]=\E\left[\left(\beta\sum_{i=1}^N c_i^{\ell}m_i(\ms)\right)^2\right].$$ %Moreover, for all large enough $N$, $\mathbb{E}[\sum_i m_i(\ms)^6]\lesssim N^{-2/3}(\log{N})^2\fa\lesssim (\log{N})^{3/2}$ and $\lambda_1(H_N)\leq \beta^*$. As a result,
		Consequently, 
		\begin{align}\label{eq:rate1}
		\notag T_{2,\ell} -x_{\ell+1}&=\mathbb{E}\left[\left(\sum_{i=1}^N c^{(\ell)}_i\left\{m_i(\ms)\beta +\xi_i m_i(\ms)^3\right\}\right)^2\right]-\mathbb{E}\left[\left(\beta\sum_{i=1}^N c^{(\ell)}_im_i(\ms)\right)^2\right]\\ 
		&\le 2\sqrt{ x_{\ell+1}}\lVert \mc^{(\ell)}\rVert_2\sqrt{\mathbb{E}\left[\sum_i m_i(\ms)^6\right]}+\lVert \mc^{(\ell)}\rVert_2^2\mathbb{E}\left[\sum_i m_i(\ms)^6\right]. 
		\end{align}
		Finally, using Cauchy-Schwarz inequality gives
		\begin{align}\label{eq:rate_0}
		\E\Big(\sum_{i=1}^Nm_i(\ms)^6\Big)\le \sqrt{\E(\sum_{i=1}^Nm_i^2)^2}\sqrt{\E \max_{1\le i\le N}|m_i(\ms)|^8}\le C^2 \fa \alpha_N^2(\log N)^2
		\end{align}
		for some $C$ free of $N$, where the last inequality uses part (a) of this lemma and part (b) of Lemma \ref{lem:qformuni}.
		%&\le 2\sqrt{x_{k+1}}\lVert \mc \rVert_2\beta^k\lVert A_N\rVert_2\varepsilon_N+\lVert \mc \rVert_2^2\beta_N^{2k}\lVert A_N\rVert_2^2\varepsilon_N^2	\end{align}
		%with $\beta_N:=\beta \max_{1\le i\le N}R_i$, where the last line uses \eqref{eq:rate_0}. 
		Noting that  $T_{1,\ell}\lesssim \lVert {\mathbf c}^{(\ell)}\rVert_2^2$ by part (b) of~\cref{lem:auxtail}, combining \eqref{eq:gather_1}, \eqref{eq:rate1} and \eqref{eq:rate_0}  along with \eqref{eq:gather} gives the existence of a constant $D$ free of $N,\ell$ such that
		\begin{align}\label{eq:recursion}
		x_\ell\le x_{\ell+1}+2\sqrt{x_{\ell+1}}\lVert \mc \rVert_2\beta_N^\ell\delta_N+\lVert \mc \rVert_2^2\beta_N^{2\ell}\delta_N^2+D\beta_N^{2\ell}\lVert \mc \rVert_2^2,
		\end{align}
		where we have also used the bound $\Vert \mc^{(\ell)} \rVert_2\le \beta_N^\ell \Vert \mc \rVert_2$ with $\beta_N:=\beta \lVert A_N\rVert_2$, and we set $\delta_N:=\max(1,C \lVert A_N\rVert\alpha_N\log N)$. Since $\beta_N\rightarrow \beta<1$, for all $N$ large we have $\beta_N\le \beta_0$ for some $\beta_0<1$. Given constants $\beta_0\in (0,1),D>0$, there exists $M$ large enough such that $M>(\beta_0\sqrt{M}+1)^2+D$. With this $M,\beta_0$ we claim that for all $\ell$, we have 
		\begin{align}\label{eq:recursion_2}
		x_\ell\le M \lVert \mc \rVert_2^2 \beta_0^{2\ell}\delta_N^2,
		\end{align}
		from which \eqref{eq:claim_20} is immediate on setting $\ell=0$. For proving \eqref{eq:recursion_2} we use backwards induction on $\ell$. Using Cauchy-Schwarz inequality gives
		\begin{align*}
		x_\ell\le N\lVert \mc^{(\ell)} \rVert_2^2\le N \beta_N^\ell \lVert \mc \rVert_2^2,
		\end{align*}
		and so \eqref{eq:recursion_2} holds for all $\ell$ large enough, as $\beta_N<\beta_0$. Assume that the result holds for $x_{\ell+1}$ for some $\ell$, i.e. $x_{\ell+1}\le M \lVert \mc \rVert_2^2 \beta_0^{2\ell+2}\delta_N^2$. Using \eqref{eq:recursion} gives
		\begin{align*}
		x_\ell\le  \lVert \mc \rVert_2^2 \beta_0^{2\ell}\delta_N^2\Big(M\beta_0^2+2\sqrt{M}\beta_0+1+D\Big)\le M \lVert \mc \rVert_2^2 \beta_0^{2\ell}\delta_N^2,
		\end{align*}
		where the last step uses the choice of $M$. This verifies the claim for $\ell$, and hence proves \eqref{eq:recursion_2} by backward induction, for all $\ell\ge 0$. 
	\end{enumerate}  
\end{proof}

\begin{proof}[Proof of~\cref{lem:linfbdgen}]
	\noindent
	%First, recall the definition of $r_N$ from~\cref{theo:crit}.
	%where $\alpha_N=\max_{1\leq i\leq N}\sum_{j}A_N(i,j)^2$ is as defined in~\cref{theo:crit}.
	(a)	As in the proof of part (a) of Lemma \ref{lem:unifbd}, it suffices to prove the result for $\lambda$ large.
	To this effect, define an $N\times N$ matrix $\tilde{A}_N$ by setting $\tilde{A}_N(i,j):= A_N(i,j)/R_{\mathrm{max}}$ for $i\neq j$ and $\tilde{A}_N(i,i):= 1-R_i/R_{\mathrm{max}}$ where $R_{\mathrm{max}}=\max_{1\leq i\leq N}R_i$. Observe that $\mathbf{1}^{\top}\tilde{A}_N=\mathbf{1}^{\top}$, and so
	%and $\max_{1\leq i\leq j} |\tilde{A}_N(i,j)-A_N(i,j)|\lesssim r_N/\log{N}$. Also $\tilde{A}_N$ is a symmetric matrix with non-negative entries and $\limsup_{N\to\infty}\lambda_2(\tilde{A}_N)<1$. This gives
	\begin{align}\label{eq:covid_2}
	\notag&|(m_i(\ms)-\mm)-\sum_{j=1}^N\tilde{A}_N(i,j)(m_j(\ms)-\mm)|\\
	\notag=&|m_i(\ms)-\sum_{j=1}^N\tilde{A}_N(i,j)m_j(\ms)|
	\\
	\notag\le &|m_i-\sum_{j=1}^NA_N(i,j)m_j(\ms)|+\sum_{j=1}^N|A_N(i,j)-\tilde{A}_N(i,j)|\\
	\lesssim &|m_i-\sum_{j=1}^NA_N(i,j)\tanh(m_j(\ms))|+\max_{1\le i\le N}|m_i(\ms)|^3+\max_{1\le i\le N}|R_i-1|.
	\end{align}
	Using part (b) of~\cref{lem:auxtail}, a union bound as in the proof of part (a) of Lemma \ref{lem:unifbd} shows that for all $\lambda>0$ we have $\mmp(E_N^c)\le 2e^{-c\lambda^2}$ for some constant $c>0$ free of $N$,\text{ where }  
	\begin{equation}\label{eq:prebd}
	E_N:=\bigg\{\max\limits_{1\leq i\leq N}\big|m_i(\ms)-\sum_{j=1}^N A_N(i,j)\tanh( m_j(\ms))\big|\leq \lambda\sqrt{\alpha_N\log N} \bigg\}
	\end{equation}
	for some constant $c$ free of $N$, with $\alpha_N=\max_{1\le i\le N}\sum_{j=1}^NA_N(i,j)^2$ as in Theorem \ref{theo:uniq1}. Proceeding to bound the second term in the RHS of \eqref{eq:covid_2}, 
	note that, with $K:=\argmax_{1\leq i\leq N}|m_i(\ms)|$ and assuming $m_K(\ms)\geq 0$ without loss of generality, we have:
	\begin{small}
	\begin{align*}
	m_K(\ms)^3\lesssim m_K(\ms)-\tanh(m_K(\ms))\le m_K(\ms)-\sum_{j=1}^NA_N(K,j)\tanh(m_j(\ms))+\max_{1\le i\le N}|R_i-1|.
	\end{align*}
	\end{small}
	By a symmetric argument, we get:
	\begin{align}\label{eq:covid_3}
	\max_{1\leq i\leq N} |m_i(\ms)|^3\lesssim \max_{1\leq i\leq N}|m_i(\ms)-\sum_{j=1}^N A_N(i,j)\tanh(m_j(\ms))|+\max_{1\le i\le N}|R_i-1|.
	\end{align}
	Thus, combining \eqref{eq:covid_2} and \eqref{eq:covid_3}, on the set $E_N$ we have
	\begin{small}
	\begin{align}\label{eq:covid_4}
	\max_{1\le i\le N}|(m_i(\ms)-\mm)-\sum_{j=1}^N\tilde{A}_N(i,j)(m_j(\ms)-\mm)|\le C\Big[ \lambda\sqrt{\alpha_N\log N}+\max_{1\le i\le N}|R_i-1|\Big]
	\end{align}
	\end{small}
	for some $C<\infty$ free of $N$.
	Now, for any integer $\ell\ge 2$ we have
	\begin{align*}
	&|(m_i(\ms)-\mm)-\sum_{j=1}^N\tilde{A}^{\ell}_N(i,j)(m_j(\ms)-\mm)|\\
	\le&|(m_i(\ms)-\mm)-\sum_{j=1}^N\tilde{A}^{\ell-1}_N(i,j)(m_j(\ms)-\mm)|+\Big|\sum_{j=1}^N\tilde{A}_N^{\ell-1}(i,j)\Big\{(m_j(\ms)-\mm)\\ &\quad\qquad -\sum_{k=1}^N\tilde{A}_N(j,k)(m_k(\ms)-\mm)\Big\}\Big|\\
	\le&|(m_i(\ms)-\mm)-\sum_{j=1}^N\tilde{A}^{\ell-1}_N(i,j)(m_j(\ms)-\mm)|+\max_{1\le j\le N}\Big|(m_j(\ms)-\mm)\\ &\qquad\qquad -\sum_{k=1}^N\tilde{A}_N(j,k)(m_k(\ms)-\mm)\Big|,
	\end{align*}
	which, via a recursive argument gives
	\begin{align}\label{eq:covid_5}
	\notag&\max_{1\le i\le N}\Big|(m_i(\ms)-\mm)-\sum_{j=1}^N\tilde{A}^{\ell}_N(i,j)(m_j(\ms)-\mm)\Big|\nonumber \\
	\le& \ell \max_{1\le i\le N}|(m_i(\ms)-\mm)-\sum_{j=1}^N\tilde{A}_N(i,j)(m_j(\ms)-\mm)| \nonumber \\ \le & C\ell \Big(\lambda\sqrt{\alpha_N\log N}+\max_{1\le i\le N}|R_i-1|\Big),
	\end{align}
	where the last line uses \eqref{eq:covid_4} on the set $E_N$. 
	Using part (a) of~\cref{lem:tracebd}, we note the existence of $D$ free of $N$ such that for the choice $\ell=D\log N$ we have $\max_{1\le i\le N}A^\ell(i,i)\le \frac{3}{N}$. With this choice of $\ell$, we have
	\begin{align*}
	&\;\;\;\;\;\;\mmp\bigg(\max\limits_{1\leq i\leq N}|m_i(\ms)-\mm|\geq 2C\ell\Big[ \lambda\sqrt{\alpha_N\log N}+\max_{1\le i\le N}|R_i-1|\Big], E_N\bigg)\\ \le &
	%\overset{(a)}{\lesssim} 2\exp\left(-\overline{\theta}C^2\right)+
	\P\bigg(\max_{1\leq i\leq N}|\sum_{j=1}^N \tilde{A}_N^{\ell}(i,j)(m_j(\ms)-\mm)|\ge C\ell\Big[ \lambda\sqrt{\alpha_N\log N}+\max_{1\le i\le N}|R_i-1|\Big]
	\bigg)\\
	&\le \mmp\bigg(\sum_{j=1}^N (m_j(\ms)-\mm)^2\ge  \frac{C^2\ell^2N}{2}\Big[ \lambda\sqrt{\alpha_N\log N}+\max_{1\le i\le N}|R_i-1|\Big]^2\bigg),
	% \le &e^{-\delta C^2\ell^2 N\Big[ \lambda\sqrt{\alpha_N\log N}+\max_{1\le i\le N}|R_i-1|\Big]^2} e^{\varepsilon_N},
	\end{align*}
	where the last line uses Cauchy-Schwarz inequality. Fixing $\delta$ small enough and using part (c) of~\cref{lem:qformuni}, this gives
	\begin{align*}
	&\log \P\bigg(\max\limits_{1\leq i\leq N}|m_i(\ms)-\mm|\geq 2C\ell\Big[ \lambda\sqrt{\alpha_N\log N}+\max_{1\le i\le N}|R_i-1|\Big],E_N\bigg)\\
	\lesssim& -N\alpha_N (\log N)^3\lambda^2-N(\log N)^2\max_{1\le i\le N}|R_i-1|^2+\log \E e^{\delta \sum_{i=1}^N(m_i(\ms)-\mm)^2}\\
	\lesssim & -N\alpha_N (\log N)^3\lambda^2-N(\log N)^2\max_{1\le i\le N}|R_i-1|^2+\fa+\frac{1}{N}\Big[\sum_{i=1}^N(R_i-1)^2\Big]^2\\ &\qquad\qquad +\frac{1}{N}\Big[\sum_{i=1}^N(R_i-1)\Big]^2+\log N,
	\end{align*}
	from which the desired conclusion follows for $\lambda$ large enough on noting the inequality $N\alpha_N\ge \fa\gtrsim 1$.
\end{proof}
In order to prove~\cref{lem:linfbdgen}, parts (b) and (c), we need the following lemma whose proof we defer to the end of this section. 
\begin{lemma}\label{lem:mombound}
	
	Assume that~\eqref{eq:A1},~\eqref{eq:A2},~\eqref{eq:A4}~holds, and the RHS of~\eqref{eq:critmain} is bounded. Then, setting $\nu_N:=\mathbbm{E}_1^N[(N^{1/4}\overline{\ms})^6]$ the following conclusions hold:
	\begin{align}
	\nu_N\lesssim &\nu_N^{2/3}+\nu_N^{1/3}+\nu_N^{1/2}+\nu_N^{1/2}\sqrt{\frac{\mathbbm{E}\left[\sum_{i=1}^N (R_i-1)\sigma_i\right]^2}{N^{1/2}}},\label{eq:eigvecpr_1}\\
	%\end{align*}
	%
	%
	%\begin{lemma}\label{lem:eigvecpr}
	%	Let $\mq_N:= (q_{N1},\ldots ,q_{NN})$ denote an eigenvector corresponding to a eigenvalue of $A_N$ other than the maximal eigenvalue $\lambda_1(A_N)$. Then we have that $\mme[\sum_{i=1}^N q_{Ni}\sigma_i]^2\lesssim 1$. %Further, let $\tilde{A}_N$ denote the $N\times N$ matrix defined in the proof of~\cref{lem:linfbdgen} and $\tilde{\mq}_N:= (\tilde{q}_{N1},\ldots ,\tilde{q}_{NN})$, an eigenvector of $\tilde{A}_N$ other than $1$. Then $\mme[\sum_{i=1}^N \tilde{q}_{Ni}\sigma_i]^2\lesssim 1+\mme[N^{1/4}\mm]^2$.
	%\end{lemma}
	%\begin{lemma}\label{lem:contrcrit}
	%With $\mc:= (R_1-1,R_2-1,\ldots ,R_N-1)$. we have
	\mme\left[\sum\limits_{i=1}^N (R_i-1)\sigma_i\right]^2\lesssim &(\log N)^4\left(\sum_{i=1}^N (R_i-1)^2+N^{-1/2}\Big[\sum_{i=1}^N(R_i-1)\Big]^2\right)\bigg(1+\mme[(N^{1/4}\overline{\ms})^2]\bigg).\label{eq:contrcrit}
	\end{align}
\end{lemma}

\begin{proof}[Proof of~\cref{lem:linfbdgen}, parts (b) and (c)]
	Use~\eqref{eq:contrcrit} and the fact that the RHS of \eqref{eq:critmain} is bounded to get \[\mme\left[\sum_{i=1}^N (R_i-1)\sigma_i\right]^2\lesssim \sqrt{N}(1+\mme[(N^{1/4}\overline{\ms})^2])\lesssim \sqrt{N}(1+\nu_N^{1/3}).\]
	Along with \eqref{eq:eigvecpr_1}, this gives %$\nu_N$ from~\cref{lem:mombound} and assume that $\liminf \nu_N=\infty$. Therefore, there exists a subsequence along which $\nu_N$ converges to $\infty$. By an abuse of notation,~\cref{lem:mombound} then implies,
	$\nu_N\lesssim \nu_N^{2/3}+\nu_N^{1/3}+\nu_N^{1/2}(1+\nu_N^{1/3})+1$, and so $\nu_N$ must be bounded, thereby proving part (b). Now, part (c) is an immediate consequence of part (b) and~\eqref{eq:contrcrit}.
\end{proof}

\begin{proof}[Proof of~\cref{lem:mombound}]
	\begin{enumerate}
		\item[(a)] Proof of \eqref{eq:eigvecpr_1}.
		
		%We begin with the usual exchangeable pair and $T_N=N^{1/4}\overline{\ms}$ and $T_N'=N^{1/4}\overline{\ms}'$, to get~\eqref{eq:reflat} as in the proof of~\cref{theo:crit}. Next, look at the real valued function, $f(x)=\tanh(x)-x+x^3/3$. It is easy to see from standard calculus considerations that $|f(x)|\leq 2|x|^5/15$ for all $x\in\mathbb{R}$. Plugging the above bound and $T_N=N^{1/4}\overline{\ms}$ in~\eqref{eq:reflat}, we get:
		To begin, borrowing notation from the proof of Theorem \ref{theo:crit} and using \eqref{eq:reflat} gives the existence of $C<\infty$ such that
		\begin{align*}
		&\;\;\;\big|\mme[T_N-T_N'|\ms]-N^{-3/2}T_N^3/3\big|\nonumber \\ &\le \frac{2}{15} N^{-2}|T_N|^5+C\bigg\{N^{-3/4}|\overline{\ms}-\mm|+N^{-2}|T_N|\sum_{i=1}^N (m_i(\ms)-\mm)^2\\ &\qquad\qquad+N^{-7/4}\bigg|\sum_{i=1}^N (m_i(\ms)-\mm)^3\bigg|\bigg\}.
		\end{align*}
		On multiplying both sides of the above inequality by $N^{3/2}|T_N|^3$ and taking expectation gives
		\begin{align}\label{eq:mombound1}
		\mme[T_N^6]&\leq  (2/5)N^{-1/2}\mme|T_N|^8+3C\bigg\{N^{3/4}\mme\left[|T_N|^3|\overline{\ms}-\mm|\right]\nonumber \\&\qquad\qquad +N^{-1/2}\mme\left[|T_N|^4\sum_{i=1}^N (m_i(\ms)-\mm)^2\right]\nonumber \\ &+N^{-1/4}\mme\left[|T_N|^3\big|\sum_{i=1}^N (m_i(\ms)-\mm)^3\big|\right]\bigg\}+3N^{3/2}\big|\mme(T_N-T_N')T_N^3\big|.
		\end{align}
		We will now bound each of the terms in the RHS of \eqref{eq:mombound1}. 
		To begin, note that  that  $|T_N-T_N'|\leq 2N^{-3/4}$ and $\mme[T_N]=\mme[T_N']$. This, along with the fact that $(T_N,T_N')$ is an exchangeable pair gives
		\begin{align}\label{eq:5}
		\notag\mme(T_N-T_N')T_N^3&=(1/2)\mme(T_N-T_N')T_N^3-(1/2)\mme(T_N-T_N')(T_N')^3\\ &=(1/2)\mme\notag\left[(T_N-T_N')^2(T_N^2+T_NT_N'+(T_N')^2)\right]\\
		&\leq 6N^{-3/2}\mme[T_N^2]\le 6 N^{-3/2} \nu_N^{1/6},
		\end{align}
		where $\nu_N=\E[(N^{1/4}\bar{\mathbf{\sigma}})^6]$ as in the statement of the lemma. Also with $\varepsilon_N,r_N$ as in the statement of Theorem \ref{theo:crit}, use part (c) of Lemma \ref{lem:qformuni}, and part (a) of Lemma \ref{lem:linfbdgen} to get that for any positive integer $p$, we have
		\begin{align}\label{eq:lem:qformuni}
		\E \Big[\sum_{i=1}^N(m_i(\ms)-\mm)^2\Big]^p \lesssim \varepsilon_N^p,\quad \E \max_{1\le i\le N}|m_i(\ms)-\mm|^p \lesssim r_N^p.
		\end{align}
		Finally, since the RHS of \eqref{eq:critmain} is bounded, we have 
		\begin{align}\label{eq:bound_conclusion}
		\varepsilon_N\lesssim \sqrt{N},\quad \varepsilon_N r_N\lesssim N^{1/4}, \quad \lVert\mathbf{c}\rVert_2^2+N^{-1/2}\Big[\sum_{i=1}^Nc_i\Big]^2\lesssim \sqrt{N}.
		\end{align}
		Armed with these estimates and proceeding to bound the second, third and fourth terms in~\eqref{eq:mombound1}, use H\"older's inequality  to get
		\begin{align}\label{eq:2}
		N^{3/4}\mme[|T_N|^3|\overline{\ms}-\mm|]\leq &N^{-1/4}\sqrt{\nu_N}\sqrt{\mme\left[\sum_{i=1}^N (R_i-1)\sigma_i\right]^2}\\
		\label{eq:3}\mme\left[T_N^4\sum_{i=1}^N (m_i(\ms)-\mm)^2\right]\leq &\nu_N^{2/3}\left(\mme\left[\sum_{i=1}^N (m_i(\ms)-\bar{\boldsymbol{m}})^2\right]^3\right)^{1/3}\nonumber \\ &\lesssim \nu_N^{2/3}\varepsilon_N\lesssim \nu_N^{2/3}\sqrt{N}
		\end{align}
		\begin{align}\label{eq:4}
		&\;\;\;\;\notag\mme\left[|T_N|^3\Bigg|\sum_{i=1}^N (m_i(\ms)-\mm)^3\Bigg|\right]&\nonumber \\ &\leq\sqrt{\nu_N}\bigg(\E \Big[\sum_{i=1}^N(m_i(\ms)-\mm)^2\Big]^4\bigg)^{1/4}\bigg(\E \Big[\max_{1\le i\le N}(m_i(\ms)-\mm)^4\Big]\bigg)^{1/4}\nonumber \\
	    &\lesssim \sqrt{\nu_N}\varepsilon_Nr_N\lesssim \sqrt{\nu_N} N^{1/4}\end{align}	%\qquad ,\qquad \mme[T_N^2]\leq \big(\mme[T_N^6]\big)^{1/3} \qquad \mbox{ and}
		where the last bounds in \eqref{eq:3} and \eqref{eq:4} use \eqref{eq:lem:qformuni} and  \eqref{eq:bound_conclusion}. %Using part (c) of Lemma~\ref{lem:qformuni} we have
		%$\bigg(\E\Big[ \sum_{i=1}^N(m_i(\ms)-\mm)^2\Big]^3\bigg)^{1/3}\lesssim \varepsilon_N$,
		%whereas using part (a) of Lemma \ref{lem:linfbdgen} gives
		%\begin{align*}
		%\bigg(\E \Big[\sum_{i=1}^N(m_i(\ms)-\mm)^3\Big]^2\bigg)^{1/2}
		%\le \bigg(\E \max_{1\le i\le N}|m_i(\ms)-\mm|^4\bigg)^{1/4} \bigg(\E \Big[\sum_{i=1}^N(m_i(\ms)-\mm)^2\Big]^4\bigg)^{1/4}
		% \lesssim\epsilon_N r_N,
		%\end{align*}
		%where $\varepsilon_N,r_N$ are as defined in Theorem \ref{theo:crit}. 
		%Since the RHS of \eqref{eq:critmain} is bounded, we have $\varepsilon_N\lesssim \sqrt{N}$, and $\varepsilon_N r_N\lesssim N^{1/4}$. 
		Finally, for the fifth term in the RHS of \eqref{eq:mombound1}, note that $|T_N|\leq N^{1/4}$, and so the first term in the RHS of \eqref{eq:mombound1} is bounded by $(2/5)\E[T_N^6]$. Combining this along with \eqref{eq:mombound1}, \eqref{eq:5}, \eqref{eq:2}, \eqref{eq:3} and \eqref{eq:4} gives
		%	Plugging in the above observations in~\eqref{eq:mombound1}, we get
		%	\begin{align}\label{eq:mombound2}
		%	\mme[T_N^6]&\leq (2/5)\mme[T_N^6]+3CN^{3/4}\mme\left[|T_N|^3|\overline{\ms}-\mm|\right]+3CN^{-1/2}\mme\left[|T_N|^4\sum_{i=1}^N (m_i(\ms)-\mm)^2\right]\nonumber \\ &+3CN^{-1/4}\mme\left[|T_N|^3\big|\sum_{i=1}^N (m_i(\ms)-\mm)^3\big|\right]+18\mme[T_N^2]
		%	\end{align}
		%	 and and the assumptions stated in the statement of the lemma, we automatically have polynomial moment control of the form: 
		%$$\left(\mme\left[\Bigg|\sum_{i=1}^N (m_i(\ms)-\bar{\boldsymbol{m}})^3\Bigg|\right]^4\right)^{1/4}\lesssim o\left(N^{1/4}\right)\qquad \mbox{and} \qquad \left(\mme\left[\sum_{i=1}^N (m_i(\ms)-\bar{\boldsymbol{m}})^2\right]^5\right)^{1/5}\lesssim o\left(N^{1/2}\right).$$
		%Now, plugging in $\nu_N=\mme[(N^{1/4}\overline{\ms})^6]$ and the above bounds in~\eqref{eq:mombound2}, we get:
		\begin{align*}
		\nu_N\lesssim \nu_N^{1/3}+\nu_N^{1/2}+\nu_N^{2/3}+\nu_N^{1/2}\sqrt{\frac{\mathbbm{E}\left[\sum_{i=1}^N (R_i-1)\sigma_i\right]^2}{N^{1/2}}}
		\end{align*}
		which completes the proof of \eqref{eq:eigvecpr_1}

		\item[(b)] Proof of~\eqref{eq:contrcrit}.
		
		To begin, for any vector ${\bf h}:=(h_1,\cdots,h_N)$ write \begin{align*}\sum_{i=1}^N h_i\sigma_i=\sum_{i=1}^Nh_i(\sigma_i-& \tanh(m_i(\ms)))+\sum_{i=1}^Nh_i(\tanh(m_i(\ms))-\tanh(\mm))\\ &+\tanh(\mm)\sum_{i=1}^Nh_i,\end{align*} which using part (b) of Lemma \ref{lem:auxtail} gives
		\begin{align}\label{eq:2-6'}
		\E\Big[\sum_{i=1}^Nh_i\sigma_i\Big]^2 \lesssim &\lVert\mathbf{h}\rVert_2^2+\lVert\mathbf{h}\rVert_2^2\varepsilon_N+\Big[\sum_{i=1}^Nh_i\Big]^2\E \mm^2\lesssim \lVert\mathbf{h}\rVert_2^2\varepsilon_N+\Big[\sum_{i=1}^Nh_i\Big]^2\E (\mm^2),
		%\lesssim &\epsilon_N\lVert\mathbf{c}\rVert_2^2+\Big[\sum_{i=1}^Nc_i\Big]^2,
		\end{align}
		where the second line uses part (c) of Lemma \ref{lem:qformuni}, and $\varepsilon_N$  equals the RHS of \eqref{eq:critqformuni}. Setting ${\bf c}={\bf R}-{\bf 1}$ and using \eqref{eq:2-6'} with ${\bf h}={\bf c}$ gives 
		%let ${\bf c}:={\bf R}-{\bf 1}$, and write \[\sum_{i=1}^Nc_i\sigma_i=\sum_{i=1}^Nc_i(\sigma_i-\tanh(m_i(\ms)))+\sum_{i=1}^Nc_i(\tanh(m_i(\ms))-\tanh(\mm))+\tanh(\mm)\sum_{i=1}^Nc_i.\]
		%From this, using part (b) of Lemma \ref{lem:auxtail} gives
		\begin{align}\label{eq:2-6}
		\E\Big[\sum_{i=1}^Nc_i\sigma_i\Big]^2 \lesssim &\lVert\mathbf{c}\rVert_2^2+\lVert\mathbf{c}\rVert_2^2\varepsilon_N+\Big[\sum_{i=1}^Nc_i\Big]^2\E \mm^2\lesssim N,
		%\lesssim &\epsilon_N\lVert\mathbf{c}\rVert_2^2+\Big[\sum_{i=1}^Nc_i\Big]^2,
		\end{align}
		where the last line uses \eqref{eq:bound_conclusion}.
		% Since the RHS of \eqref{eq:critmain} is bounded, we have \begin{align}\label{eq:bound_conclusion}
		%\varepsilon_N\lesssim \sqrt{N},\quad \lVert\mathbf{c}\rVert_2^2+\Big[\sum_{i=1}^Nc_i\Big]^2\lesssim \sqrt{N},\
		%\end{align}giving the crude bound
		Along with \eqref{eq:eigvecpr_1} this gives
		$\nu_N\lesssim  \nu_N^{1/3}+\nu_N^{1/2}+\nu_N^{2/3}+\nu_N^{1/2}\sqrt{N}$, and so 
		\begin{align}\label{eq:m1}
		\nu_N\lesssim \sqrt{N}\Rightarrow \E \bar{\ms}^6\lesssim N^{-1}.
		\end{align}
		Also, an argument similar to the derivation of \eqref{eq:2-6} shows that for any positive integer $p$, we have
		\begin{align}\label{eq:m2}
		\E(\bar{\ms}-\mm)^{2p}&=N^{-2p}\E\Big[\sum_{i=1}^Nc_i\sigma_i\Big]^{2p}
		\\ &\lesssim N^{-2p}\bigg(\lVert\mathbf{c}\rVert_2^{2p}+\lVert\mathbf{c}\rVert_2^{2p}\varepsilon_N^p+\left(\sum_{i=1}^Nc_i\right)^{2p}\E\mm^2\bigg)\lesssim N^{-p},
		\end{align}
		where the last bound uses \eqref{eq:bound_conclusion}. Combining we have the following conclusions:
		\begin{align}
		\label{eq:m3}\E {\mm}^{6}\lesssim& \E(\bar{\ms})^{6}+\E(\bar{\ms}-\mm)^{6}\lesssim \frac{1}{N},\\
		\notag \E \Big(\sum_{i=1}^Nm_i(\ms)^6\Big)\lesssim& N\E {\mm}^6+\sqrt{\E\max_{1\le i\le N}(m_i(\ms)-\mm)^8}\sqrt{\E \Big[\sum_{i=1}^N(m_i(\ms)-\mm)^2\Big]^2}\\
		\label{eq:m4}\lesssim &1+r_N^4\varepsilon_N\lesssim 1,
		\end{align}
		where \eqref{eq:m3} uses
		\eqref{eq:m1} and \eqref{eq:m2} with $p=3$, and \eqref{eq:m4} uses \eqref{eq:m3} along with \eqref{eq:lem:qformuni} and \eqref{eq:bound_conclusion}. Armed with these estimates, we now focus on deriving \eqref{eq:contrcrit}. 
		
		Let $\tilde{A}_N$ be as defined in the proof of part (a) of  Lemma \ref{lem:linfbdgen}, %and set $H_N:= A_N-\tilde{A}_N$.%a $N\times N$ matrix such $\tilde{A}_N(i,j):= A_N(i,j)$ for $i\neq j$ and $\tilde{A}_N(i,i)=R_M-R_i$ where $R_M$ denotes the maximum row sum of $A_N$. Set $H_N:= A_N-\tilde{A}_N$. 
		and set  $\mc^{(\ell)}:= \mc^{\top}\tilde{A}_N^\ell$ and $x_\ell:=\mme\left[\sum\limits_{i=1}^N c^{(\ell)}_i\sigma_i\right]^2$ for $\ell\ge 0$. As in the proof of part (b) of Lemma \ref{lem:unifbd}, we can write $x_\ell=T_{1,\ell}+T_{2,\ell}+T_{3,\ell}$, where
		\begin{align*}
		T_{1,\ell}:= &\mme\bigg[\sum\limits_{i=1}^N  c^{(\ell)}_i(\sigma_i-\tanh{m_i(\ms)})\bigg]^2\qquad,\qquad T_{2,\ell}:= \mme\bigg[\sum\limits_{i=1}^N c^{(\ell)}_i\tanh{m_i(\ms)}\bigg]^2,\\ & T_{3,\ell}:= 2\mme\left[\sum_{i\ne j} c^{(\ell)}_ic^{(\ell)}_j(\sigma_i-\tanh{m_i(\ms)})\tanh{m_j(\ms)}\right].
		\end{align*}
		%Throughout the course of this proof, we will use $M>0$ to denote a  universal constant and consider $\ell\lesssim (\log{N})^2$. The value of $M$ may change from line to line. Also note that $\|\mc^{(\ell)}\|\leq \|\mc\|$. 
		By the argument presented in the proof of part  (b) of Lemma \ref{lem:unifbd} we have $T_{1,\ell}\lesssim \lVert \mc^{(\ell)}\rVert_2^2\le  \lVert \mc\rVert_2^2  $, and $T_{3,\ell}\lesssim \lVert \mc\rVert_2^2 $.
		%	define $m_i^j(\ms):= \sum_{k\ne j}^N A_N(i,k)\sigma_k$, and note that
		%	\begin{align*}
		%	&\;\;\;\;\sum_{i,j,i\neq j} c^{(\ell)}_ic^{(\ell)}_j\mme[(\sigma_i-\tanh{m_i(\ms)})\tanh{m_j(\ms)}]\\ &\overset{(a)}{=} \sum_{i,j,i\neq j} c^{(\ell)}_ic^{(\ell)}_j\mme[(\sigma_i-\tanh{m_i(\ms)})(\tanh{m_j(\ms)}-\tanh{m_j^i(\ms)})]\\&\leq \sum_{i,j,i\neq j} |c^{(\ell)}_i||c^{(\ell)}_j|A_N(i,j)\leq M\sum_{i=1}^N (c^{(\ell)}_i)^2\leq M\|\mc\|^2.
		%	\end{align*}
		%	Here (a) follows from the fact that $\mme[(\sigma_i-\tanh{m_i(\ms)})\tanh{m_j^i(\ms)}]=0$ for $i\neq j$. 
		Next, using Taylor Series expansion, we can write $\tanh(m_i(\ms))=m_i(\ms)+\xi_im_i(\sigma)^3$ for random variables $\{\xi_i\}_{1\le i\le N}$ which are uniformly bounded by $1$ in absolute value. Consequently, 
		\begin{align}\label{eq:t2l}
		\notag T_{2,\ell}-x_{\ell+1}=&\mme\bigg[\ms^\top A_N\mathbf{c}^{(\ell)}+\sum_{i=1}^N c^{(\ell)}_i \xi_i m_i(\ms)^3\bigg]^2-\mme\bigg[ \ms^\top \tilde{A}_N {\mathbf c}^{(\ell)}\bigg]^2\\
		\notag\le &2\sqrt{x_{\ell+1}}\sqrt{\E\Big[\sum_{i=1}^N|c_i^{(\ell)}m_i(\ms)^3|\Big]^2}+2\sqrt{x_{\ell+1}}\sqrt{\E\Big[{\bf c}^{(\ell)}H_N \ms\Big]^2}+\E\Big[{\bf c}^{(\ell)}H_N \ms\Big]^2\\
		\notag&+\E\Big[\sum_{i=1}^N|c_i^{(\ell)}m_i(\ms)^3|\Big]^2
		+2\sqrt{\E\Big[{\bf c}^{(\ell)}H_N \ms\Big]^2}\sqrt{\E\Big[ \sum_{i=1}^N|c_i^{(\ell)}m_i(\ms)^3|\Big]^2}\\%+\E\Big[\sum_{i=1}^N|c_i^{(\ell)}m_i(\ms)^3|\Big]^2\\
		\notag\le &2\sqrt{x_{\ell+1}} \lVert \mc\rVert_2^2 \sqrt{\E \sum_{i=1}^Nm_i(\ms)^6}+2\sqrt{x_{\ell+1}}\sqrt{E\Big[{\bf c}^{(\ell)}H_N \ms\Big]^2}+\E\Big[{\bf c}^{(\ell)}H_N \ms\Big]^2\\
		+& \lVert \mc\rVert_2^2 \E\Big[ \sum_{i=1}^Nm_i(\ms)^6\Big]+2 \lVert \mc\rVert_2^2 \sqrt{\E\Big[{\bf c}^{(\ell)}H_N \ms\Big]^2}\sqrt{\E \sum_{i=1}^Nm_i(\ms)^6}.
		%\mme\bigg[\sum_{i=1}^N c^{(\ell+1)}_i\sigma_i+(\mc^{(\ell)})^{\top}H_N\ms+\sum_{i=1}^N c^{(\ell)}_i \xi_i m_i(\ms)^3\bigg]^2
		\end{align}
		%To simplify notation, define $\GN_1:=\big|\sum_{1=1}^N c_i\big|+\sum_{i=1}^N c_i^2$ and recall the definition of $\varepsilon_N$ from~\cref{theo:crit}. Observe the following chain of inequalities:
		Proceeding to bound the RHS of \eqref{eq:t2l}, %use \eqref{eq:eigvecpr_1} to note that
		use \eqref{eq:A4} and \eqref{eq:m2} respectively to note that $\lVert H_N\rVert_{\text{op}}\lesssim N^{-1/4}$, and $N\E (\mm)^2\lesssim N \E (\bar{\ms})^2+ 1$, and an application of \eqref{eq:2-6'} with ${\bf h}=H_N\mc^{(\ell)})$ gives
		%Using this, we first bound $\E\Big[{\bf c}^{(\ell)}H_N \ms\Big]^2$ below:
		\begin{align}
		\;\;\;\;\mme\left[(\mc^{(\ell)})^{\top}H_N\ms\right]^2
		%\notag&\le  3\mme\left[\left(\sum_{i,j=1}^N c^{(\ell)}_iH_N(i,j)(\sigma_j-\tanh(m_j(\ms)))\right)^2+3\left(\sum_{i,j=1}^N c^{(\ell)}_iH_N(i,j)\big(\tanh(m_j(\ms))-\tanh(\mm)\big)\right)^2\right]\\
		%\notag &+3\left(\sum_{i,j=1}^N c^{(\ell)}_i H_N(i,j)\tanh(\mm)\right)^2 \\
		&\lesssim \lVert \mc^{(\ell)}\rVert_2^2  \lVert H_N\rVert_\text{op}^2\bigg(\varepsilon_N+N \E(\mm^2)\bigg)
		\lesssim  \lVert \mc\rVert_2^2 \mu_N,\label{eq:chm}
		%\lesssim (R_M-1)^2\GN_1\varepsilon_N+N^{-1/2}\GN_1\varepsilon_N\mme[(N^{1/4}\mm)^2]\lesssim \GN_1\big(1+\mme[(N^{1/4}\mm)^2]\big)
		\end{align}
		with $\mu_N:=1+\E(N^{1/4}\bar{\ms})^2$, where the second inequality uses \eqref{eq:bound_conclusion}~and~\eqref{eq:m2}. We now claim that there exists a constant $D>0$ such that 
		\begin{align}\label{eq:claim_recur}
		x_{D(\log N)^2}\lesssim \mu_N\bigg( \lVert \mc\rVert_2^2+N^{-1/2}\bigg[\sum_{i=1}^Nc_i\bigg]^2\bigg)^2.
		\end{align} 
		Given this claim, we have the existence of a constant $C$ free of $N$ such that 
		\begin{align}\label{eq:induct_1}
		x_{D(\log N)^2}&\le  C^2\mu_N\bigg( \lVert \mc\rVert_2^2+N^{-1/2}\bigg[\sum_{i=1}^Nc_i\bigg]^2\bigg)^2.
		\end{align}
		Also, 
		%	Next we have
		%	\begin{align}\label{eq:chm2}
		%	\notag\E \sum_{i=1}^Nm_i(\ms)^6\lesssim& \E \sum_{i=1}^N(m_i(\ms)-\mm)^6+N^{-1/2} \E (N^{1/4}\mm)^6\\
		%	\lesssim &\varepsilon_N r_N^4+N^{-1/2}\E(N^{1/4}\mm)^2\lesssim \mu_N,
		%	\end{align}
		%	where the last bound uses the fact that $r_N^4\lesssim N^{-1/2}$, as the RHS of \eqref{eq:critmain} is bounded. 
		using \eqref{eq:chm} and \eqref{eq:m4}, and making $C$ bigger if needed, for all $\ell\ge 0$ we have
		%\[\mme\left[(\mc^{(\ell)})^{\top}H_N\ms\right]^2\le C^2\lVert \mathbf{c}\rVert|^2 \mu_N,\quad \E \sum_{i=1}^Nm_i(\ms)^6\le C^2\mu_N\]
		%which along with \eqref{eq:t2l}, \eqref{eq:chm} and \eqref{eq:m4} gives
		%	The last line above follows from a combination of Lemmas~\ref{lem:auxtail} and~\ref{lem:qformuni}. Next, observe that,
		%	\begin{align}\label{eq:contrcrit1}
		%	\mme\left[\sum_{i=1}^N c^{(\ell)}_i\xi_im_i(\ms)^3\right]^2\lesssim N\|\mc^{(\ell)}\|\mme\left[\max_{1\leq i\leq N}|m_i(\ms)|^6\right]\lesssim \|\mc\|^2\lesssim \GN_1
		%	\end{align}
		%	where the above display follows from the fact that $\mme[\max_{1\leq i\leq N}|m_i(\ms)|^6]\lesssim N^{-1}$ as was a part of the proof of~\cref{lem:eigvecpr} (see~\eqref{eq:prelpf}). Now plugging in these bounds for $x_\ell$, we get:
		\begin{align}\label{eq:contrcrit2}
		x_\ell\leq &x_{\ell+1}+2C\sqrt{x_{\ell+1}}\lVert \mathbf{c}\rVert_2\sqrt{\mu_N}+C^2\lVert \mathbf{c}\rVert_2^2\mu_N.
		%+&C^2\lVert \mathbf{c}^{(\ell)}\rVert_2^2\mu_N+2C^2\lVert \mathbf{c}^{(\ell)}\rVert_2\mu_N	%\sqrt{x_{\ell+1}\GN_1\big(1+\mme[(N^{1/4}\mm)^2]\big)}+\GN_1\big(1+\mme[(N^{1/4}\mm)^2]\big)
		\end{align}
		%We now claim that there exists a constant $\lambda>0$ such that with $\ell=\lambda(\log N)^2$ we have $x_{\ell}\le  C^2\lVert \mc\rVert_2^2\mu_N$. 
		With $L=D (\log N)^2$, we will now show that the bound 
		\begin{align}\label{eq:bd_recur}
		x_{\ell}\le (L-\ell+1)^2C^2 \left[\lVert \mc\rVert_2^2 +N^{-1/2}\left(\sum_{i=1}^Nc_i\right)^2\right]
		\end{align} holds for all $\ell \in [0,L]$ by backwards induction. By \eqref{eq:induct_1} we have that \eqref{eq:bd_recur} holds for $\ell=L$. Suppose \eqref{eq:bd_recur} holds for $\ell+1$ for some $\ell\in [0,L-1]$. Using \eqref{eq:contrcrit2} gives 
		\[x_\ell \le C^2\mu_N  \lVert \mc\rVert_2^2 \Big[(L-\ell)^2+2(L-\ell)+1\Big]= (L-\ell+1)^2C^2 \mu_N  \lVert \mc\rVert_2^2, \]
		verifying \eqref{eq:bd_recur} for $\ell$, and thus verifying \eqref{eq:bd_recur} for all $\ell\in [0,L]$ by induction. Setting $\ell=0$ in \eqref{eq:bd_recur} we get the bound
		\begin{align*}
		\E\Big(\sum_{i=1}^Nc_i\sigma_i\Big)^2&\le L^2C^2\mu_N\bigg[\sum_{i=1}^Nc_i^2+N^{-1/2}\Big(\sum_{i=1}^Nc_i\Big)^2\bigg]\\ &\le C^2D^2\mu_N(\log N)^4\bigg[\sum_{i=1}^Nc_i^2+N^{-1/2}\Big(\sum_{i=1}^Nc_i\Big)^2\bigg],
		\end{align*}
		which verifies \eqref{eq:contrcrit}, as desired.
		\\
		
		It thus remains to verify \eqref{eq:claim_recur}, for which using spectral decomposition write $\tilde{A}_N=\sum_{i=1}^N\tilde{\lambda}_i\tilde{\bf q}_i\tilde{\bf q}_i^\top$, where we set $\tilde{\lambda}_i:=\lambda_i(\tilde{A}_N)$ for convenience of notation . With $L=D(\log N)^2$, this gives
		%Now suppose that $x_{\ell}\le M$
		%by suitably changing $M$. Now set $l=C(\log{N})^2$ for some fixed but large $C>0$ (free of $N$, to be determined later). Let $\tilde{\lambda}_1\geq \ldots \geq\tilde{\lambda}_N$ and $\tilde{\mq}_1,\ldots ,\tilde{\mq}_N$ be the eigenvalues and eigenvectors from the standard spectral decomposition of $\tilde{A}_N$. It is easy to check that $\tilde{\lambda}_1=1$ and $\tilde{\mq}_1=\mathbf{1}/\sqrt{N}$. Then,
		\begin{align*}
		\mc^{\top}\tilde{A}_N^L\ms&=\overline{\ms}\sum_{i=1}^N c_i+\tilde{\lambda}_N^L \mc^{\top}\tilde{\mq}_N\tilde{\mq}_N^{\top}\ms+\sum_{i=2}^{N-1}\tilde{\lambda}_i^L\mc^{\top}\tilde{\mq}_i\tilde{\mq}_i^{\top}\ms\\ &=\overline{\ms}\sum_{i=1}^N c_i+\tilde{\lambda}_N^L \mc^{\top}\tilde{\mq}_N\tilde{\mq}_N^{\top}\ms+O(N^{-cD+2}), 
		\end{align*}
		where the last equality uses~\cref{lem:tracebd}  to get  \[\max_{2\le i\le N-1}|\tilde{\lambda}_i|^{L}\le \Big(1-\frac{c}{\log N}\Big)^{\ell}\le N^{-cD}\] for some $c>0$. Consequently for $D$ large enough we have
		%The above display follows from~\cref{lem:tracebd} by choosing $C$ fixed, but large. As a result by application of~\cref{lem:eigvecpr}, we have:
		\begin{align}\label{eq:contrcritrev1}
		\mme\left[\mc^{\top}\tilde{A}_N^L\ms\right]^2\lesssim \Big[\sum_{i=1}^Nc_i\Big]^2\mme[\overline{\ms}^2]+\lVert \mathbf{c}\rVert_2^2\mme[(\tilde{\bf q}_N^{\top}{\ms})^2].
		\end{align}
		Since $\tilde{\mq}_N^{\top}\tilde{A}_N=\lambda_N \tilde{\mq}_N^{\top}$ where $\tilde{\lambda}_N$ is bounded away from $1$ by~\eqref{eq:well_connect}, we have
		\begin{align*}%\label{eq:eig_contr}
		&\;\;\;\;\notag(1-\tilde{\lambda}_N)\sum_{i=1}^N \tilde{q}_N(i)\sigma_i\\ &=\sum_{i=1}^N \tilde{q}_{N}(i)(\sigma_i-m_i(\ms))+\tilde{\mathbf q}_N^{\top}H_N\sigma\\
		&=\sum_{i=1}^N \tilde{q}_{N}(i)(\sigma_i-\tanh(m_i(\ms)))+\sum_{i=1}^N \tilde{q}_{N}(i)(\tanh(m_i(\ms))-m_i(\ms))+\tilde{\mathbf q}_N^{\top}H_N\sigma.
		\end{align*}
		This immediately gives
		\begin{align*}
		&\;\;\;\;\notag(1-\tilde{\lambda}_N)^2\mme\left[\sum_{i=1}^N \tilde{q}_{N}(i)\sigma_i\right]^2\\ &\lesssim \mme\left[\sum_{i=1}^N \tilde{q}_{N}(i)(\sigma_i-\tanh(m_i(\ms)))\right]^2+\mme\left[\sum_{i=1}^N |\tilde{q}_{N}(i)||m_i(\ms)|^3\right]^2+\E \Big[\tilde{\mathbf q}_N^\top H_N\ms\Big]^2\\
		\notag &\leq \sum_{i=1}^N \tilde{q}_{N}(i)^2+\sqrt{\sum_{i=1}^N \tilde{q}_{N}(i)^2}\sqrt{\mme\left[\sum_{i=1}^N m_i(\ms)^6\right]}+\E \Big[\tilde{\mathbf q}_N^\top H_N\ms\Big]^2\\
		&\lesssim 1+ \lVert H_N\rVert_\text{op}^2\Big[\varepsilon_N+N\E( \mm^2)\Big].
		\end{align*}
		where the last bound uses  \eqref{eq:2-6'} with ${\bf h}=\tilde{\mathbf q}_N$. Since $N\E(\mm^2)\lesssim N\E(\bar{\ms}^2)+1\lesssim \sqrt{N}\mu_N$,
		using the last bound along with \eqref{eq:contrcritrev1} gives \[\E(\mc^{\top}\tilde{A}_N^L\ms)^2\lesssim \mu_N\bigg(N^{-1/2}\Big[\sum_{i=1}^Nc_i\Big]^2+\sum_{i=1}^Nc_i^2 \Big),\]
		thus verifying \eqref{eq:claim_recur}, and hence completing the proof of the lemma.
	\end{enumerate}
\end{proof}
\begin{remark}
	As in the proofs of part (b) of Lemmas~\ref{lem:unifbd}~and~\ref{lem:linfbdgen}, the above argument can be modified to bound the moments of general linear combinations $\sum_{i=1}^Nc_i\sigma_i$ for any ${\mathbf c}\in \R^N$.
\end{remark}

\section{Supplementary lemmas and proofs}\label{sec:supp}
\subsection{Proof of~\cref{lem:tailsubG} and~\cref{lem:gentail}}
\begin{proof}[Proof of~\cref{lem:tailsubG}]
	Noting the presence of $ \mathrm{Tr}^+(D_N)$ in the RHS of the bound, it suffices to prove the result for $D_N$ with all diagonal entries set to $0$. %$\mz^{\top}:= 
	Let $(Z_1,Z_2,\ldots ,Z_N)$ be i.i.d. $\mathcal{N}(0,1)$ random variables. We claim that %Set $\mtx^{(0)}:= (\tilde{X}_1,\ldots ,\tilde{X}_N)$, $T_N:=\sum_{i,j=1}^ND_N(i,j)\tilde{X}_i\tilde{X}_j$ and $S_n=\sum_{i,j=1}^ND_N(i,j)Z_iZ_j$
	\begin{small}
	\begin{equation}\label{eq:keysteptailsubG}
	\mme\left[\exp\left(\frac{1}{2}\sum_{i,j=1}^ND_N(i,j)\tilde{X}_i\tilde{X}_j+\sum_{i=1}^N c_i\tilde{X}_i\right)\right]\le \mme\left[\exp\left(\frac{s_\mu}{2}\sum_{i,j=1}^ND_N(i,j)Z_iZ_j+\sqrt{s_\mu}\sum_{i=1}^N c_iZ_i\right)\right].
	%\mme\left[\exp\left(\frac{1}{2}T_n+\lambda_n\mc^{\top}\mtx^{(0)} \right)\right] \leq \mme\left[\exp\left(\frac{s_{\mu}}{2}S_n+\lambda_n\sqrt{s_{\mu}}\mc^{\top}\mz\right)\right].
	\end{equation}
	\end{small}
	Indeed, to see this, recall that the sub-Gaussian norm of $\tilde{X}_i$ is given by $s_{\mu}$ for $1\leq i\leq n$, (see e.g.,~\cite[Theorem 2.1]{ostrovsky2014exact}). Consequently, for every $\theta\in\R$ we have
	%\begin{equation}\label{eq:subG}
	$\mme\left[\exp\left(\theta\tilde{X}_i\right)\right]\leq \mme\left[\exp\left(\theta \sqrt{s_{\mu}} Z_i\right)\right].$
	%\end{equation}
	Using this,~\eqref{eq:keysteptailsubG} can be obtained by inductively replacing each $\tilde{X}_i$ on the left hand side of~\eqref{eq:keysteptailsubG} with $\sqrt{s_{\mu}}Z_i$.% using~\eqref{eq:subG}.
	 The RHS of \eqref{eq:keysteptailsubG} can be computed directly to get
	%\noindent Note that $s_{\mu}S_n=\sum_{i,j,i\neq j} s_{\mu}D_n(i,j)Z_iZ_j$ is a quadratic form of standard i.i.d. Gaussian random variables. By a simple change of measure, it is easy to see that,
	\begin{align*}
	&\;\;\;\;\log\left\{\mme\left[\exp\left(\frac{1}{2}\sum_{i,j=1}^Ns_{\mu}D_N(i,j)Z_iZ_j+\sqrt{s_{\mu}}\sum_{i=1}^N c_iZ_i\right)\right]\right\}\\ &= -(1/2)\log{\mathrm{det}(I_N-s_{\mu}D_N)}+(1/2)s_{\mu}\sum_{i=1}^N c_i^2,
	\end{align*}
	from which the desired bound follows on noting the existence of  $\rho \in (s_\mu\limsup_{N\rightarrow\infty}\lambda_1(D_N),1)$, and using the bound $-\log(1-x)\lesssim x$ for $x\in [0,1-\rho]$.
\end{proof}

\begin{proof}[Proof of~\cref{lem:gentail}]
	By H\"{o}lder's inequality, for any $p>0$ the left hand side of~\eqref{eq:gentailmain} can be bounded by
	\begin{align*}%\label{eq:errbound5}
	\left(\mathbb{E}^{CW}\left[\exp\left(\frac{\beta(1+p)}{2}\boldsymbol{\sigma}^{\top}\BN\boldsymbol{\sigma}\right)\right]\right)^{1/(1+p)}\P(|\tilde{W}_N-M(\ms)|\geq\varepsilon)^{\frac{p}{1+p}}.
	\end{align*}
	Since $\limsup_{N\rightarrow\infty}\frac{1}{N}\log \P(\tilde{W}_N-M(\ms)|>\varepsilon)<0$ by part (b) of~\cref{prop:CWresmain}, it suffices to show the existence of $p>0$ such that
	\begin{align}\label{eq:suffice_covid}
	\limsup\limits_{N\to\infty}\frac{1}{N}\log \mathbb{E}^{CW}\left[\exp\left(\frac{\beta(1+p)}{2}\boldsymbol{\sigma}^{\top}\BN\boldsymbol{\sigma}\right)\right]\leq 0.
	\end{align}
	To this effect, setting $$g_p(\ms):=\frac{\beta}{2}\boldsymbol{\sigma}^{\top}A_N\boldsymbol{\sigma}+\frac{\beta p}{2}\boldsymbol{\sigma}^{\top}\BN\boldsymbol{\sigma}$$ note that
	\begin{align}\label{eq:errbound6}
	\log \mathbb{E}^{CW}\left[\exp\left(\frac{\beta(1+p)}{2}\boldsymbol{\sigma}^{\top}\BN\boldsymbol{\sigma}\right)\right]=%&\log \bigg[\frac{\sum\limits_{\boldsymbol{\sigma}\in \{-1,1\}^N}e^{g_p(\ms)}}{Z_N^{CW}(\beta,B)}\bigg]
	\sup_{\ms\in [-1,1]^N}\{g_p(\ms)-\sum_{i=1}^NI(\sigma_i)\}-\log Z_N^{CW}(\beta,B)+o(N),
	\end{align}
	where the last line uses \cite[Theorem 1.1]{Basak2017} along with the observation  $\mathrm{Tr}((A_N+\BN)^2)=o(N)$.
	Using spectral theorem we have $A_N=\sum_{i=1}^N\lambda_i\mathbf{q}_i\mathbf{q}_i^\top$ with $\lambda_i=\lambda_i(A_N)$, and so
	\begin{align*}
	&\;\;\;\;\;\sup\limits_{\boldsymbol{\sigma}\in [-1,1]^N} \left(g_p(\boldsymbol{\sigma})-\frac{\beta}{2}\sum_{i=1}^N \sigma_i^2\right)\nonumber\\ 
	&{=} \sup\limits_{\boldsymbol{\sigma}\in[-1,1]^N} \Bigg[\frac{\beta}{2}\sum_{i=1}^N (\lambda_i-1)\boldsymbol{\sigma}^\top\mathbf{q}_i\mathbf{q}_i^\top\boldsymbol{\sigma}+\frac{\beta p}{2}\boldsymbol{\sigma}^\top\Bigg(\lambda_1\mathbf{q}_1\mathbf{q}_1^\top-\frac{\mathbf{1}\mathbf{1}^\top}{N}\Bigg)\boldsymbol{\sigma}+\frac{\beta p}{2}\sum_{i=2}^N \lambda_i\boldsymbol{\sigma}^\top\boldsymbol{q}_i\boldsymbol{q}_i^\top\boldsymbol{\sigma}\Bigg]\\
	\lesssim &o(N)+\sum_{i=2}^N (\boldsymbol{\sigma}^\top\boldsymbol{q}_i\boldsymbol{q}_i^\top\boldsymbol{\sigma})\left(-\frac{\beta}{2}(1-\lambda_i)+\frac{\beta p}{2}\lambda_i\right)
	%\overset{(b)}{\lesssim} o\big(N\big) + \sum_{i=2}^N (\boldsymbol{\sigma}'\boldsymbol{q}_i\boldsymbol{q}_i'\boldsymbol{\sigma})\left(-\frac{\beta}{2}(1-\lambda_i)+\frac{\beta p}{2}\lambda_i\right)\overset{(c)}{\lesssim} o(N).
	\end{align*}
	%\textit{(a)} is an algebraic identity which follows from expressing $I_N$ as $\sum_{i=1}^N \mathbf{q}_i\mathbf{q}_i'$ and \textit{(b)} follows
	where the bound in the last line uses \eqref{eq:A2}, and ~\cref{lem:maxeigvec}. Finally note that \eqref{eq:well_connect} shows the existence of $\rho<1$ such that $\max_{2\le i\le N}\lambda_i\le \rho$, and so there exists $p=p(\rho)$ such that $\max_{2\le i\le N}\left(-\frac{\beta}{2}(1-\lambda_i)+\frac{\beta p}{2}\lambda_i\right)\le 0$. Combining we have $$\sup\limits_{\boldsymbol{\sigma}\in [-1,1]^N} \left(g_p(\boldsymbol{\sigma})-\frac{\beta}{2}\sum_{i=1}^N \sigma_i^2\right)\le o(N),$$ and so 
	%For \textit{(c)}, note that, by~\eqref{eq:well_connect}~there exists $\eta>0$ such that for all large enough $N$, $1-\lambda_i\geq \eta$ for all $i\geq 2$. By choosing $p=(1/2)(\eta/(1-\eta))$, simple calculations yield $-\beta(1-\lambda_i)/2+\beta p\lambda_i/2\leq -\beta \eta/4$ which implies \textit{(c)}.\newline
	%Therefore, the above display implies,
	\begin{align*}
	\sup\limits_{\boldsymbol{\sigma}\in [-1,1]^N}(g_p(\boldsymbol{\sigma})-I(\boldsymbol{\sigma}))&\leq \sup\limits_{\boldsymbol{\sigma}\in [-1,1]^N} \left(g_p(\boldsymbol{\sigma})-\frac{\beta}{2}\sum_{i=1}^N \sigma_i^2\right)+\sup\limits_{\boldsymbol{\sigma}\in [-1,1]^N} \left(\frac{\beta}{2}\sum_{i=1}^N \sigma_i^2-I(\boldsymbol{\sigma})\right)\\ &=o(N)+\mathcal{M}_N(\beta,B),
	\end{align*}
	where $\mathcal{M}_N(\beta,B)$ is the Mean-Field prediction defined in \eqref{eq:mf_2}. Since $|\log Z_N^{CW}(\beta,B)-\mathcal{M}_N(\beta,B)|\lesssim \log N$ by part (a) of~\cref{prop:CWresmain}, \eqref{eq:suffice_covid} follows, thus completing the proof of the lemma.
	%Plugging in the above estimates from into~\eqref{eq:errbound5} and~\eqref{eq:errbound6}, the conclusion follows if $\delta>0$, $p>0$ and $\theta>0$ are chosen small enough. 
\end{proof}
\subsection{Some results on matrices}\label{sec:matlem}
\begin{lemma}\label{lem:maxeigvec}
	Let $\sum_{i=1}^N\lambda_i(A_N)\mathbf{q}_i\mathbf{q}_i^{\top}$  be the spectral decomposition  of $A_N$. Suppose that~\eqref{eq:A2} and \eqref{eq:well_connect} hold, and $\sum_{i=1}^N (R_i-1)=o(N)$. 
	\begin{enumerate}
		\item[(a)]
		Then $\lVert {\mathbf q}_1-\mathbf{e}\rVert_2=o(1)$, where $\mathbf{e}:= N^{-1/2}\vo$.%=\sum_{i=1}^N c_i\mathbf{p}_i$.
		
		\item[(b)]
		Further we have $\limsup\limits_{N\to\infty}\lambda_1(\BN)<1$, where $\BN=A_N-\frac{1}{N}{\bf 1}{\bf 1}^{\top}$. \end{enumerate}
\end{lemma}
\begin{proof}
	\begin{enumerate}
		\item[(a)]
		Write ${\bf e}=\sum_{i=1}^Nc_i{\bf q}_i$ with $c_1>0$ by Perron-Frobenius Theorem, and note that
		\[1+o(1)=\frac{1}{N}\sum_{i=1}^NR_i={\bf e}^\top A_N{\bf e}=\sum_{i=1}^Nc_i^2\lambda_i(A_N)\le \lambda_1(A_N)c_1^2+\lambda_2(A_N)(1-c_1^2)\]
		Along with \eqref{eq:A2} and \eqref{eq:well_connect}, this gives $c_1^2=1+o(1)$, and so $\langle {\bf q}_1,{\bf e}\rangle=c_1=1+o(1)$, thus completing the proof of part (a).
		
		\item[(b)]
		This follows on using part (a) to note that
		\[\lVert \BN\rVert_2\le \lVert\sum\limits_{i=2}^N\lambda_i(A_N){\bf q}_i{\bf q}_i^{\top}\rVert_2+\lVert \lambda_1(A_N){\bf q_1}{\bf q_1}^{\top}-{\bf e}{\bf e}^\top\rVert_2\le \lambda_2(A_N)+o(1),\]
		and using \eqref{eq:well_connect}.
	\end{enumerate}
\end{proof}

\begin{lemma}\label{lem:tracebd}
	Let $\Gamma_N$ be an $N\times N$ symmetric matrix with non-negative entries, such that $\mathbf{1}^{\top}\Gamma_N=\mathbf{1}^{\top}$ and $\Gamma_N$ satisfies \eqref{eq:well_connect}. Then the following conclusions hold: 
	\begin{enumerate}
		\item[(a)] There exists $c>0$ such that for all $\ell\ge 1$ and $N$ large we have \[ \max_{1\leq i \leq N} \Gamma_N^{\ell}(i,i)\leq \frac{2}{N}+\frac{2}{e^{c\ell}}.\]
		\item[(b)] There exists $\delta>0$ such that for all $N$ large enough we have \[\max_{2\le i\le N-1}|\lambda_i(\Gamma_N)|\le 1-\frac{\delta}{\log N}.\]% where $(\lambda_1,\lambda_2,\cdots,\lambda_N)$ are the eigenvalues of $\Gamma_N$ arranged in decreasing order.
	\end{enumerate}
\end{lemma}
\begin{proof}%[Proof of~\cref{lem:tracebd}, part (a)]
	\begin{enumerate}
		\item[(a)]
		Setting $\lambda_i:=\lambda_i(\Gamma_N)$ for simplicity of notation, let $\mathcal{J}_+:=\{j\in [2,N]:\lambda_j>0\}$ and $\mathcal{J}_-:=\{j\in [2,N]:\lambda_j<0\}$, and use spectral theorem to note that for any positive integer $\ell$ we have \[\Gamma_N^{\ell}=\frac{1}{N}{\bf 1}{\bf 1}^\top+\sum_{j\in \mathcal{J}_+}|\lambda_j|^{\ell} {\bf q}_j{\bf q}_j^\top+(-1)^{\ell}\sum_{j\in \mathcal{J}_-}|\lambda_j|^{\ell} {\bf q}_j{\bf q}_j^\top,\]
		where $({\bf q}_1,\cdots,{\bf q}_N)$ are the eigenvectors of $\Gamma_N$. To begin, use \eqref{eq:well_connect} to note the existence of $c>0$ such that for all $N$ large enough we have $\lambda_2\le e^{-c}$, which gives
		\begin{align}\label{eq:wc1}
		\sum_{j\in \mathcal{J}_+}|\lambda_j|^{\ell}q_{ij}^2\le \lambda_2^{\ell}\le e^{-c\ell},
		\end{align}
		where $q_{ij}$ denotes the $i$-th entry of the vector ${\bf q}_j$.
		
		For $\ell$ odd, noting that $\Gamma_N^\ell(i,i)\ge 0$ gives
		\begin{align*}
		\sum_{j\in \mathcal{J}_-}|\lambda_j|^{\ell}q_{ij}^2\le \frac{1}{N}+\sum_{j\in \mathcal{J}_+}|\lambda_j|^{\ell}q_{ij}^2\le \frac{1}{N}+\lambda_2^{\ell}\le \frac{1}{N}+e^{-c\ell},
		\end{align*}
		where the last inequality uses \eqref{eq:wc1}. Using the fact that $\max_{2\le i\le N}|\lambda_i|\le 1$, for $\ell\ge 2$ we have
		\begin{align*}
		\sum_{j\in \mathcal{J}_-}|\lambda_j|^{\ell}q_{ij}^2\le \sum_{j\in \mathcal{J}_-}|\lambda_j|^{\ell-1}q_{ij}^2\le \frac{1}{N}+e^{-c\ell}.
		\end{align*}
		Combining these two bounds, for all $\ell\ge 1$ we have
		\[|\Gamma_N^\ell(i,i)|\le \frac{1}{N}+\sum_{j\in \mathcal{J}_+}|\lambda_j|^{\ell}q_{ij}^2+\sum_{j\in \mathcal{J}_-}|\lambda_j|^{\ell}q_{ij}^2\le \frac{2}{N}+\frac{2}{e^{c\ell}},\]
		thus completing the proof of part (a).	
		\item[(b)]
		Let $\delta>0$ be such that $3e^{-2\delta/c}>2$. Using part (a) with $\ell=\frac{2\log N}{c}$ and even, we have
		\begin{align*}
		\sum_{i=1}^{N} |\lambda_i|^\ell=\sum_{i=1}^N \Gamma_N^{\ell}(i,i)\leq 2+2N e^{-2\log N}\rightarrow 2.
		\end{align*}
		On the other hand if $\max_{2\le i\le N-1}|\lambda_i|>1-\frac{\delta}{\log N}$, then \[\sum_{i=1}^N|\lambda_i|^{\ell}\ge 3 \Big(1-\frac{\delta}{\log N}\Big)^{\frac{2\log N}{c}}\rightarrow 3e^{-\frac{2\delta }{c}}.\]
		These two together imply $3e^{-2\delta/c}\le 2$, a contradiction.
		%	 
		%	Thus choosing $\delta$ such that $3e^{-\delta D}<2$ gives a con
		%	Consider $E_N$ as defined in the statement with $D_2>0$ to be chosen later. By choosing $k_N$ as an even number in the above display, we get:
		%	\begin{align*}
		%	2\geq \limsup\limits_{N\to\infty}\sum_{i=1}^{N} \lambda_i^{k_N}(M_N)\geq \limsup\limits_{N\to\infty} \left(1+\exp(-2\delta D_1)(|E_N|-1)\right).
		%	\end{align*}
		%	Finally, by choosing $\delta>0$ small enough, the conclusion follows.
	\end{enumerate}
\end{proof}

\begin{remark}\label{rem:mineigen}
	Note that if $\Gamma_N$ is the adjacency matrix of a $d_N$ regular bipartite graph scaled by the degree $d_N$, which satisfies the spectral gap condition, see~\eqref{eq:well_connect}), then our lemma implies%$M_N=\mathbf{1}\mathbf{1}^{\top}/N+\mathbf{q}\mathbf{q}^{\top}$ where $q_N$ is a vector of length $N$ with $1/\sqrt{N}$ in the first $N/2$ coordinates and $-1/\sqrt{N}$ in the remaining ones. Then 
	\[\lim_{N\rightarrow\infty}\max_{1\le i\le N}\Big|N\Gamma_N^{2\ell}(i,i)-2\Big|=0\] for $\ell =D\log{N}$ with $D$ large enough. This highlights the asymptotic optimality of the bound obtained in part (a) of~\cref{lem:tracebd}. Part (b) quantifies the graph theoretic fact that for a connected $d_N$ regular graph, say $G_N$, the multiplicity of the eigenvalue $-d_N$ can be at most $1$.  It is easy to check that if $-d_N$ happens to be an eigenvalue the graph must be a bipartite graph, and all other eigenvalues will be strictly larger than $-d_N$ (i.e. there is a unique bipartition for a connected bipartite graph). In fact, our proof can be modified to show the stronger conclusion that for a $d_N$ regular bipartite graphs satisfying the spectral gap condition, the second last eigenvalue is bounded away from $-1$, i.e. \[\liminf_{N\rightarrow\infty}\frac{\lambda_{N-1}(G_N)}{d_N}>-1.\]
	
	%Note that~\cref{lem:tracebd} generalizes the graph theoretic fact that a connected regular bipartite graph admits a unique bipartition, which is equivalent to the statement that there can be at most two eigenvalues of a connected
	%\end{remark}
	%\begin{remark}[Optimality of~\cref{lem:tracebd}]
\end{remark}
\begin{comment}
\begin{lemma}\label{lem:linfop}
Assume that $A_N$ satisfies assumptions~\ta~and~\thh. Define $H_N:= \beta(1-\ut^2)A_N$. Then $\limsup_{l\to\infty} \limsup_{N\to\infty}\lVert H_N^l\rVert_{\infty}=0$.
\end{lemma}
\begin{prop}\label{prop:eigen}
Suppose $A_N$ satisfies assumptions~\ta~and~\thh, then $\lim_{N\to\infty} \lambda_1(A_N)\to 1$ and $\liminf_{N\to\infty} \lambda_N(A_N)\geq -1$.. 
\end{prop}	
\end{comment}
\subsection{Some results for the Curie-Weiss model}\label{sec:addresCW}
The following proposition collects all the results for the Curie-Weiss model which we have used previously. 
%Its proof follows from~\cref{prop:CWbasic} and straight forward calculus, and is omitted.

\begin{prop}\label{prop:CWresmain}
	Suppose $\ms$ is drawn from the Curie-Weiss model. With $\tw_N$ as in~\cref{prop:CWbasic}, the following conclusions hold:
	
	\begin{enumerate}
		\item[(a)] 
		\begin{eqnarray*}
			\log{Z_N^{CW}(\beta,B)}-N\Big\{\frac{\beta }{2}t^2+Bt-I(t)\Big\}\lesssim &1&\text{ if }(\beta,B)\in \Theta_1\cup \Theta_2,\\
			\lesssim &\log N&\text{ if }(\beta,B)\in \Theta_3.
		\end{eqnarray*}
		\item[(b)] For any $\lambda>0$, we have
		\begin{eqnarray*}
			\log{\mmp^{CW}(|\tw_N-M(\ms)|\geq \lambda)}\lesssim &-N\lambda^2&\text{ if }(\beta,B)\in \Theta_1\cup \Theta_2,\\
			\lesssim& -N\min(\lambda^2,\lambda^4)&\text{ if }(\beta,B)\in \Theta_3.
		\end{eqnarray*}
		Consequently for any sequence $\delta_N=o(N)$ we have
		\begin{eqnarray*}
			\log \E^{CW} e^{\delta_N(\tw_N-M(\ms))^2}\lesssim &1&\text{ if }(\beta,B)\in \Theta_1\cup \Theta_2,\\
			\lesssim &\frac{\delta_N^2}{N}&\text{ if }(\beta,B)\in \Theta_3.
		\end{eqnarray*}
		
		\item[(c)] For $(\beta,B)\in\Theta_2$, we have:
		$$\limsup\limits_{N\to\infty}\frac{1}{N}\log{\mmp^{CW}\left(\sum_{i=1}^N \sigma_i\in \{-2,-1,0,1,2\}\right)}< 0.$$
	\end{enumerate}
\end{prop}

\begin{proof}
	\begin{enumerate}
		\item[(a)]
		With $f(w)=\frac{\beta w^2}{2}-\log \cosh(\beta w+B)$  as in~\ref{prop:CWbasic}, a direct computation gives
		$
		Z_N^{CW}(\beta,B)=e^{-\beta/2}\sqrt{ \frac{ n\beta}{2\pi}} \int_\R e^{-nf(w)}dw,
		$
		where the function $f(w)$ has a unique global minimum at $w=t$ for $(\beta,B)\in \Theta_1\cup \Theta_3$, and two global minima at $\pm t$ for $(\beta,B)\in \Theta_2$. Also, it is easy to verify that 
		\begin{eqnarray}\label{eq:laplace}
		\begin{split}
		f(w)-f(t)\cong& (w-t)^2,\quad &\text{ for all }w\in \R, \quad \text{ if }(\beta,B)\in \Theta_1,\\
		f(w)-f(t)\cong& (w-t)^2\quad &\text{ for all }w>0, \quad \text{ if }(\beta,B)\in \Theta_2,\\
		f(w)-f(t)\cong& \min\Big[(w-t)^2,(w-t)^4\Big]\quad &\text{ for all }w\in \R,\quad  \text{ if }(\beta,B)\in \Theta_3.
		\end{split}
		\end{eqnarray}
		The desired estimates follow from these bounds and using the Laplace method for approximating integrals.
		
		\item[(b)]
		Noting that \[|\tilde{W}_N-M(\ms)|=|\tanh(\beta W_N+B)-\tanh(\beta M(\ms)+B)|\le \beta |W_N-M(\ms)|,\] it suffices to prove the desired bounds  $W_N$, which follows from straightforward computations on using \eqref{eq:laplace}.
		
		\item[(c)]
		This follows on using part (b) to note that, when $(\beta,B)\in \Theta_2$, the random variable $W_N$ has an exponential concentration near the points $\pm t$, none of which are near $0$.

	\end{enumerate}
\end{proof}
\subsection{Proof of~\eqref{eq:nonuniquetheo6}}\label{sec:revberry}
In this section, we will prove~\eqref{eq:nonuniquetheo6} using~\cite[Theorem 1.2]{Cha2011} and a soft change of measure argument. Throughout this proof, $c>0$ will denote constants free of $N$ that might change from one line to the next.

\begin{proof}
	Define the set $\tilde{\mathcal{J}}:=\{\ms\in\{-1,1\}^N:\ |\sum_{i=1}^N \sigma_i|\geq 3\}$ and $\mathcal{J}:=\{\ms\in\{-1,1\}^N:\ |\sum_{i=1}^N \sigma_i|\geq 4\}$. Recall the definition of $\P$ from~\eqref{eq:model} and note that, by part (c) of~\cref{prop:CWresmain} and part (b) of~\cref{thm:conc}, we get:
	\begin{align}\label{eq:revberry3}
	\limsup_{N\to\infty}\frac{1}{N}\log\P(\ms\in\tilde{\mathcal{J}}^c)\leq \limsup_{N\to\infty}\frac{1}{N} \log\P(\ms\in\mathcal{J}^c)<0.
	\end{align}
	Next, let $\mmq$ denote the probability measure induced by $\P$ conditioned on the event $\ms\in \tilde{\mathcal{J}}$, i.e., $\mmq(\cdot):=\P(\cdot|\ms\in\tilde{\mathcal{J}})$. Therefore, for any $B\subseteq \{-1,1\}^N$, we have
	$$\mmq(\ms\in B)=\frac{\P(\ms\in B\cap \tilde{\mathcal{J}})}{\P(\ms\in\tilde{\mathcal{J}})}.$$
	Once again, by part (c) of~\cref{prop:CWresmain} and part (b) of~\cref{thm:conc}, we get:
	\begin{align}\label{eq:revberry13}
	\limsup_{N\to\infty}\frac{1}{N}\log\mmq(\ms\in\mathcal{J}^c)\leq  \limsup_{N\to\infty}\frac{1}{N} \log\frac{\P(\ms\in\mathcal{J}^c)}{\P(\ms\in\tilde{\mathcal{J}})}<0.
	\end{align}
	Suppose that we draw $\ms_{\P}\sim\P$ and $\ms_{\mmq}\sim\mmq$. Define $\tnq:=\sqrt{N}(\bar{\boldsymbol{\ms_{\mmq}}}-M(\ms_{\mmq}))$. We will write $(\tnp,\tnpp)\equiv (T_N,T_N')$ under the law of $\P$ (recall the construction of $T_N'$ from the proof of Theorems~\ref{theo:uniq1}~and~\ref{theo:nouniq}~in~\cref{sec:pfmaintwo}).  Construct $\tnqp$ similar to $\tnpp$ as follows: Sample $I$ uniformly from the set $\{1,2,\ldots ,N\}$. Given $I=i$, replace $\sigma_{\mmq,i}$ with an independent $\pm 1$ valued random variable $\sigma_{\mmq,i}'$ with mean $\E_{\mmq}[\sigma_{\mmq,i}|(\sigma_{\mmq,j},j\neq i)]$, and set $\ms_{\mmq}':=(\sigma_{\mmq,1},\ldots ,\sigma_{\mmq,i-1},\sigma_{\mmq,i}',\sigma_{\mmq,i+1},\ldots ,\sigma_{\mmq,N})$, $\tnqp:=\sqrt{N}(\bar{\boldsymbol{\ms_{\mmq}'}}-M(\boldsymbol{\ms_{\mmq}'}))$. 
	
	By construction $(\tnqp,\tnq)$ forms an exchangeable pair under $\mmq$. 
	Moreover,
	\begin{align}\label{eq:berryclaim2}\mmq(\ms_{\mmq}'|\ms_{\mmq}=\ms)=\P(\ms_{\P}'|\ms_{\P}=\ms) \quad \mbox{for}\ \ms\in\mathcal{J},\end{align}
	and 
	$$\mmq(|\tnqp-\tnq|\leq 2N^{-1/2})=1,$$
	which follows by observing that $\max_{i=1}^N |\sigma_{\mmq,i}'-\sigma_{\mmq,i}|\leq 2$.
	
	Define $\delta:=2N^{-1/2}$. By using the above display, coupled with~\cite[Theorem 1.2]{Cha2011}, we get:
	\begin{align}\label{eq:berry1}
	\sup_{z\in\R}& \bigg|\mmq(\tnq\leq z)-\P(Z_{\tau}\leq z)\bigg|\lesssim  \E_{\mmq}\big|1-(c_0/2)\E_{\mmq}((\tnq-\tnqp)^2|\tnq)\big|\nonumber\\ & +c_1\max\{(1,c_3)\}\delta+ (c_0/c_1) \E_{\mmq}|r(\tnq)|+\delta^3 c_0\{(2+c_3/2)\E_{\mmq}|c_0 g(\tnq)|+c_1c_3/2\},
	\end{align} 
	where $Z_{\tau}$ is defined as in~\cref{lem:cw_known} for $(\beta,B)\in \Theta_2$, $r(\cdot):=\sum_{a=1}^3 H_a(\cdot)$ with $g(\cdot)$, $\{H_a(\cdot)\}_{a=1,2,3}$ from~\eqref{eq:nonuniquetheo5}.
	
	In the remainder, we will quantify the cost of moving between the probability measures $\P$ and $\mmq$ in \eqref{eq:berry1}. First, we present a claim which will be used to prove~\eqref{eq:nonuniquetheo6}. The proof of this  claim is deferred to the end of the proof.
	
	\emph{Claim:} Given any function $v(\cdot):\{-1,1\}^N\to\R$, such that $\sup_{\ms\in \{-1,1\}^n}|v(\ms)|\le a N^b$ for constants $a,b$ free of $N$, 		\begin{align}\label{eq:berryclaim1}
	\big|\E_{\mmq} v(\ms_{\mmq})-\E_{\P} v(\ms_{\P})\big|\leq \exp(-cN),
	\end{align}
	where $c$ depends only on $a,b$, and the implied constant in \eqref{eq:revberry3}. We will now complete the rest of the proof assuming Claim~\eqref{eq:berryclaim1}. For any $z\in\R$, with $v_z(\ms):=\mathbbm{1}(\sqrt{N}(\bar{\ms}-M(\ms))\leq z)$, note that $\sup_{z\in\R}\sup_{\ms\in\{-1,1\}^N} v_z(\ms)\leq 1$, which by~\eqref{eq:berryclaim1} yields:
	\begin{align}\label{eq:revberry2}
	\sup_{z\in\R}\big|\E_{\mmq} v_z(\ms_{\mmq})-\E_{\P} v_z(\ms_{\P})\big|=\sup_{z\in\R}\big|\mmq(\tnq\leq z)-\P(\tnp\leq z)\big|\leq \exp(-cN).
	\end{align}
	A similar computation as in~\eqref{eq:revberry2} further yields:		
	\begin{align}\label{eq:berry5}
	\max\left\{\E_{\mmq}|r(\tnq)|-\E_{\P}|r(\tnp)|\bigg|,\bigg|\E_{\mmq}|g(\tnq)|-\E_{\P}|g(\tnp)|\right\}\leq \exp(-cN).
	\end{align}
	Next, we will focus on the term $\E_{\mmq}((\tnq-\tnqp)^2|\tnq)$ in \eqref{eq:berry1}. By~\eqref{eq:berryclaim2}, we have $$\E_{\mmq}[1-(c_0/2)(\tnq-\tnqp)^2|\ms_{\mmq}=\ms]=\E_{\P}[1-(c_0/2)(\tnp-\tnpp)^2|\ms_{\P}=\ms]=:u(\ms), $$ for $\ms\in\mathcal{J}$. Therefore, 		
	\begin{align*} \E_{\mmq}\big|\E_{\mmq}[u(\ms)\mathbbm{1}(\ms\in\mathcal{J})|\tnq]\big|&=(\P(\ms\in\mathcal{J}))^{-1}\E_{\mmp}\big|\E_{\mmp}[u(\ms)\mathbbm{1}(\ms\in\mathcal{J})|\tnp]\big|\\ &=\E_{\mmp}\big|\E_{\mmp}[u(\ms)\mathbbm{1}(\ms\in\mathcal{J})|\tnp]\big|+r_n,
	\end{align*}
	where $|r_n|\leq \exp(-cN)$ by~\eqref{eq:revberry3}.
	
	Using the above observation with~\eqref{eq:revberry3}~and~\eqref{eq:revberry13}, we further get:
	\begin{align*}
	&\;\;\;\;\;\bigg|\E_{\mmq}\big|\E_{\mmq}[1-(c_0/2)(\tnq-\tnqp)^2|\tnq]\big|-\E_{\P}\big|\E_{\P}[1-(c_0/2)(\tnp-\tnpp)^2|\tnp]\big|\bigg|\\ &\lesssim \exp(-cN)+N\mmq(\ms\in\mathcal{J}^c)+N\P(\ms\in\mathcal{J}^c)\lesssim \exp(-cN).
	\end{align*}
	Combining the above observation with~\eqref{eq:berry5},~\eqref{eq:revberry2},~and~\eqref{eq:berry1}, completes the proof of~\eqref{eq:nonuniquetheo6}. 
	
	To complete the proof, it remains to prove~\eqref{eq:berryclaim1}, which is done below.
	
	\emph{Proof of Claim~\eqref{eq:berryclaim1}:} Observe that,
	\begin{align*}
	|\E_{\mmq} v(\ms_{\mmq})-\E_{\P} v(\ms_{\P})|&=\bigg|\frac{\E_{\P}[v(\ms_{\P})\mathbbm{1}(\ms_{\P}\in\mathcal{J})]}{\P(\ms_{\P}\in\mathcal{J})}-\E_{\P} v(\ms_{\P})\bigg|\\ &\leq \frac{\E_{\P}[|v(\ms_{\P})|\mathbbm{1}(v(\ms_{\P})\in\mathcal{J})]\P(\ms_{\P}\in\mathcal{J}^c)}{\P(\ms_{\P}\in\mathcal{J})}+\E_{\P}[|v(\ms_{\P})|\mathbbm{1}(\ms_{\P}\in\mathcal{J}^c)]\\ &\lesssim aN^b\P(\ms_{\P}\in\mathcal{J}^c)\le \exp(-cN),
	\end{align*}
	where the last line follows from~\eqref{eq:revberry3}. This establishes Claim~\eqref{eq:berryclaim1}.
\end{proof}

\section*{Acknowledgment}
The authors would like to thank the Editor, the Associate Editor and the two anonymous reviewers for their constructive suggestions that helped improve the presentation of this paper.
\bibliographystyle{imsart-nameyear}
\bibliography{template}

\end{document}